\documentclass[12pt,reqno]{amsart}
\usepackage{amsmath}
\usepackage{amssymb}
\usepackage{amstext}
\usepackage{mathrsfs}
\usepackage{a4wide}
\usepackage{graphicx}
\usepackage{bbm}
\usepackage{amsfonts}
\usepackage{mathtools}
\usepackage{esint}
\usepackage{color}
\usepackage{hyperref}
\usepackage{cases}

\allowdisplaybreaks \numberwithin{equation}{section}
\hypersetup{colorlinks=true}
\hypersetup{
    colorlinks=true,
    linkcolor=blue,
    filecolor=blue,
    urlcolor=blue,
    anchorcolor=black,
    citecolor={blue},
}
\numberwithin{equation}{section}
\newtheorem{theorem}{Theorem}[section]

\newtheorem{lemma}[theorem]{Lemma}
\theoremstyle{definition}

\theoremstyle{remark}

\allowdisplaybreaks
{\thanks{EC has been partially supported by FAPESP, grant 2021/10769-6. LCFF has been supported by FAPESP and CNPq, BR}}

\begin{document}
\title[Asymmetric vortex patch for gSQG ]{Existence of asymmetric vortex
patch for the generalized SQG equations}
\author{Edison Cuba \ \ and \ \ Lucas C. F. Ferreira}

\begin{abstract}
This paper aims to study the existence of asymmetric solutions for the
two-dimensional generalized surface quasi-geostrophic (gSQG) equations of simply
connected patches for $\alpha\in[1,2)$ in the whole plane, where $\alpha=1$
corresponds to the surface quasi-geostrophic equations (SQG). More precisely, we
construct non-trivial simply connected co-rotating and traveling patches
with unequal vorticity magnitudes. The proof is carried out by means of a
combination of a desingularization argument with the implicit function
theorem on the linearization of contour dynamics equation. Our results
extend recent ones in the range $\alpha\in[0,1)$ by Hassainia-Hmidi (DCDS-A, 2021) and Hassainia-Wheeler (SIAM J. Math. Anal., 2022) to more singular velocities, filling an open gap in the range of $\alpha$.

\end{abstract}

\maketitle


\section{Introduction}

In this paper we consider a family of active scalars in two dimensions which
are driven by an incompressible flow. More specifically, we prove the
existence of non-trivial solutions of co-rotating and traveling asymmetric
vortex patch pairs with unequal vorticity magnitudes of simply connected
patches for the generalized surface quasi-geostrophic (gSQG) equations in
the whole plane $\mathbb{R}^{2}$.

Quasi-geostrophic equations appear as a simplification of some general
geophysical fluid models, under special conditions on vorticity and buoyancy
frequency (see, e.g., \cite{Constantin1994},\cite{Pedlosky-1987}), and have
been used in the modeling of hot and cold air fronts. In this context, the
gSQG equations appear as a generalization of the so-called surface
quasi-geostrophic (SQG) equations. They are nonlocal transport equations
describing the evolution of a potential temperature $\theta $. More
precisely, in the whole plane the gSQG equations read as
\begin{equation}
\begin{cases}
\partial _{t}\theta +\mathbf{v}\cdot \nabla \theta =0 & \text{in}\ \mathbb{R}%
^{2}\times (0,T), \\
\ \mathbf{v}=-\nabla ^{\perp }(-\Delta )^{-1+\frac{\alpha }{2}}\theta &
\text{in}\ \mathbb{R}^{2}\times (0,T), \\
\theta (x,0)=\theta _{0}(x) & \text{in}\ \mathbb{R}^{2}, \\
&
\end{cases}
\label{1-1}
\end{equation}%
where $T$ is an arbitrary time, the vector field $\mathbf{v}(x,t):%
\mathbb{R}^{2}\times (0,T)\rightarrow \mathbb{R}^{2}$ is the flow velocity
of the fluid particles, $\nabla ^{\perp }=(\partial _{2},-\partial _{1})$,
and $\alpha \in \lbrack 0,2)$. The velocity field $\mathbf{v}$ is related to
the potential temperature $\theta $ via the nonlocal operator $(-\Delta
)^{-1+\frac{\alpha }{2}}$ satisfying $\mathbf{v}=\nabla ^{\perp }\psi $ with
\begin{equation*}
\psi (x,t)=-(-\Delta )^{-1+\frac{\alpha }{2}}\theta (x)=\int_{\mathbb{R}%
^{2}}K^{\alpha }(x-y)\theta (y)dy,
\end{equation*}%
where $K^{\alpha }$ stands for the fundamental solution of the fractional
Laplacian $(-\Delta )^{-1+\frac{\alpha }{2}}$ in the whole plane $\mathbb{R}%
^{2}$, namely
\begin{equation*}
K^{\alpha }(x)=\left\{
\begin{array}{lll}
\frac{1}{2\pi }\ln \frac{1}{|x|},\ \  & \ \ \alpha =0, &  \\
\frac{C_{\alpha }}{2\pi }\frac{1}{|x|^{\alpha }},\ \  & \ \ \alpha \in (0,2),
&
\end{array}%
\right.
\end{equation*}%
with $C_{\alpha }=\frac{\Gamma (\alpha /2)}{2^{1-\alpha }\Gamma (\frac{%
2-\alpha }{2})}$ and $\Gamma (x)$ being the Euler gamma function. The case $%
\alpha =0$ corresponds to the 2D incompressible Euler equations and $\alpha
=1$ to the SQG equations. The full range $\alpha \in \lbrack 0,2)$ was
proposed in a combination of works, see C\'{o}rdoba \textit{et al.} \cite%
{cordova} and Chae \textit{et al.} \cite{chae}. In fact, model (\ref{1-1})
with $\alpha \in \lbrack 0,1]$ and $\alpha \in \lbrack 1,2)$ was initially
considered by \cite{cordova} and \cite{chae}, respectively. Here we are
interested in the case $\alpha \in \lbrack 1,2)$ which corresponds to more
singular velocities $\mathbf{v}$.

These aforementioned active scalar equations have attracted a lot of
attention in the last decades. In the sequel, without making a complete
list, we mention some important results on them. Local-in-time
well-posedness results for SQG and gSQG have been obtained in some
functional spaces such as Sobolev, Besov, Triebel-Lizorkin, among others
(see, e.g., \cite{Chae2},\cite{chae},\cite{Wang-1}, and their references).
For instance, there are results for gSQG equations in Sobolev spaces $H^{s}$
with $\alpha \in (0,2)$ and $s>1+\alpha $ and for SQG in Besov spaces $%
B_{p,q}^{s}$ with either $s>\frac{n}{p}+1$ and $q\in \lbrack 1,\infty ]$ or $%
s=\frac{n}{p}+1$ and $q=1$. In the case $\alpha =0,$ global well-posedness
was obtained by Yudovich \cite{Yud} with the initial data belonging to $%
L^{1}\cap L^{\infty }$. However, it is delicate to generalize this theory to
the full range $\alpha \in (0,2)$, since the velocity is more singular and
scales below the Lipschitz class. In this way the global well-posedness
remains open for the full $\alpha \in (0,2)$. On the other hand, global
existence of weak solutions in $L^{2}$ was obtained for SQG with $\alpha =1$
by Resnick \cite{resnick}. An extension of this result to the case $\alpha
\in (0,1)$ was obtained in \cite{chae,lazar}. Results concerning to the
strong ill-posedness were proved recently by C\'{o}rdoba and Mart\'{\i}%
nez-Zoroa \cite{cordoba2}. More precisely, the authors constructed solutions
to the gSQG in $\mathbb{R}^{2}$ that initially are in $C^{k,\beta }\cap
L^{2} $ but are not in $C^{k,\beta }$ for $t>0$. Considering the $n$%
-dimensional Euler equations with $n\geq 2$ (particularly, (\ref{1-1}) with $%
\alpha =0$), Elgindi and Masmoudi \cite{el} showed strong ill-posedness in
the class $C^{k}\cap L^{2}(\Omega )$, where $\Omega $ can be the whole
space, the torus or a smooth bounded domain. Moreover, singularity formation
for the gSQG equations has been obtained in \cite{kiselev} for patch
solutions in the half-plane for small $\alpha >0,$ while the global-in-time
regularity for the 2D Euler equations $(\alpha =0)$ holds.

The above results show that the gSQG equations present an intricate behavior
with respect to the well-posedness and singularity formation. In fact, a
complete description of such properties is far from settled for the full
range $\alpha \in (0,2).$ This scenario naturally motivates to investigate solutions exhibiting some specific properties as well as
presenting a special dynamics. In this direction, an important class of
solutions extensively studied in the last years is concerning to active
scalar equations evolving from initial measure such as the so-called vortex
patch (see \cite{cordova}). More precisely, it is described by patches where
the initial data is the characteristic function of a smooth bounded domain $%
D\subset \mathbb{R}^{2}$, namely $\theta _{0}=\boldsymbol\chi _{D}$. In this
case, the patch structure is preserved for a short time and the boundary
evolves according to the contour dynamics equation, see \cite{Gancedo-1,kiselev} for more details. For the 2D Euler equations ($\alpha =0$),
Chemin \cite{Che} observed that the $C^{1+\gamma }$-regularity of the
boundary of the patch is preserved in time by applying paradifferential
calculus. An effective tool to investigate the boundary regularity is the
contour dynamics equation. Using this technique, Hmidi \textit{et al.} \cite%
{Hmi} proved that the boundary of a rotating vortex patch is smooth provided
that the patch is close enough to the bifurcation circle in the Lipschitz
norm when $\alpha =0$. In \cite{Has} these results were extended to $\alpha
\in (0,1)$ and Castro \textit{et al.} \cite{Cas1} filled the gap to the gSQG
with $\alpha \in \lbrack 0,2)$. The case of symmetric pairs of vortex
patches was treated in \cite{HM} for $\alpha \in (0,1)$ and in \cite{cao}
for $\alpha \in \lbrack 1,2).$

In a similar spirit, there are other types of
studies such as vortex patch solutions bifurcating from ellipses or doubly
connected components (annulus). Castro \textit{et al.} \cite{Cas4} obtained
the existence of analytic rotating solutions for the vortex patch equations
with $\alpha \in \lbrack 0,2)$, bifurcating from ellipses. For the other
kind, de la Hoz \textit{et al.} \cite{de1} proved the existence of
doubly connected rotating patch solutions for gSQG with $\alpha \in (0,1)$. Another point of
view is looking to the level sets of the solutions rather than the solution.
In this way, existence of a family of global solutions for SQG and rotating
vorticities for the 2D Euler equations were obtained in \cite{Cas2} and \cite%
{Cas3}, respectively.

Above we have reviewed results on single patches and symmetric patch pairs.
Furthermore, we have asymmetric configurations for patch pairs which are of
physical interest and have been analyzed numerically by Dritschel \cite{Dritschel}
by considering specific inelastic interactions under the
principle of adiabatic steadiness. They introduce further difficulties in
the analysis of the nonlocal velocity and the contour dynamics equation
associated to (\ref{1-1}). For the 2D Euler equations, the asymmetric case
was considered by Hassainia and Hmidi \cite{has2}, where they obtained the
existence of co-rotating and traveling asymmetric vortex pairs. Considering
the gSQG equations with $\alpha \in \lbrack 0,1),$ Hassainia and Wheeler
\cite{mutipole} obtained the existence of a finite number of multipole
vortex patches when the point vortex equilibrium is non-degenerate,
providing particularly examples of asymmetric rotating and traveling vortex
patch pairs, as well as asymmetric stationary tripoles.

In this paper we are concerned with the existence of asymmetric vortex patch
pairs for (\ref{1-1}) in the range $\alpha \in \lbrack 1,2)$. More
precisely, we prove the existence of global solutions of asymmetric
co-rotating and traveling vortex pairs. For these kinds of vortex pairs, our
results extend those of \cite{has2,mutipole} to more singular velocities,
filling an open gap in the range of $\alpha $. Moreover, they complement
those of \cite{cao,Cas1} by covering the case of vortex patch pairs
with asymmetric configurations.

Inspired by arguments as in \cite{cao,Cas1,has2,HM}, we study the contour
dynamics equation associated to the gSQG equations (\ref{1-1}) with initial
data $\theta _{0}=k_{1}\boldsymbol\chi _{D_{1}}+k_{2}\boldsymbol\chi
_{D_{2}} $ where $k_{1}\neq k_{2}$, among other suitable conditions on
magnitudes $k_{1},k_{2}$ and the sets $D_{1},D_{2}$, see the profile (\ref%
{sole}) for details. By virtue of the transport equation, the solution of (%
\ref{1-1}) preserves this same initial-data structure for any time. To
illustrate the ideas, momentarily consider $\theta _{0}$ being just the
characteristic function $\boldsymbol\chi _{D}$. Then, the solution is given
by $\theta (x,t)=\boldsymbol\chi _{D(t)}.$ The motion of the boundary $%
\partial D(t)$ agrees with
\begin{equation*}
\left( \partial _{t}z(x,t)-\mathbf{v}(z(x,t),t)\right) \cdot \mathbf{n}( x)=0,
\end{equation*}%
where the boundary $\partial D(t)$ is parameterized by $z(x,t)$ and the vector $\mathbf{n}(x)$ stands for the normal vector on the boundary with $x\in \lbrack
0,2\pi )$. By a direct application of the Biot-Savart law, together with the
divergence theorem, we can recover the velocity through the formula
\begin{equation*}
\mathbf{v}(x,t)=\frac{C_{\alpha }}{2\pi }\int_{\partial D_{t}}\frac{1}{%
|x-y|^{\alpha }}\,\mathrm{d}y.
\end{equation*}%
The evolution equation for the interface is parameterized by a $2\pi $
periodic curve $z(x,t)$ that satisfies
\begin{equation*}
\partial _{t}z(x,t)=\frac{C_{\alpha }}{2\pi }\int_{0}^{2\pi }\frac{\partial
_{x}z(y,t)}{|z(x,t)-z(y,t)|^{\alpha }}\,\mathrm{d}y,
\end{equation*}%
which in the literature is known as the contour dynamics equation. The above
integral is divergent for $\alpha \in \lbrack 1,2)$. To eliminate the
singularity we can rewrite the evolution equation for the interface as
follows
\begin{equation*}
\partial _{t}z(x,t)=\frac{C_{\alpha }}{2\pi }\int_{0}^{2\pi }\frac{\partial
_{x}z(y,t)-\partial _{x}z(x,t)}{|z(x,t)-z(y,t)|^{\alpha }}\,\mathrm{d}y.
\end{equation*}%
This extra term on the last integral does not change the evolution of the
patch due to we can add terms in the tangential direction.

Our main task is to construct global solutions of co-rotating and traveling
asymmetric vortex pairs using bifurcation tools, whose supports consist of
simply connected components which are perturbations of small disks.

In what follows, we shall informally state the main result of this paper.
For a complete and more detailed version we refer to the reader to Theorems %
\ref{main}, \ref{mainb} and \ref{trans}.

\begin{theorem}
\label{thm:informal} Suppose $\alpha\in[1,2)$. Then, there exists $%
\varepsilon_0 > 0$ with the following properties:

\begin{enumerate}
\item For any $\varepsilon \in (0,\varepsilon _{0}]$ and any $\gamma
_{1},\gamma _{2}\in \mathbb{R}$ with $\gamma _{1}+\gamma _{2}\neq 0$, (\ref%
{1-1}) has a global co-rotating asymmetric vortex patch pair solution $%
\theta _{\varepsilon }(x,t)=\theta _{0,\varepsilon }\left( Q_{\Omega
_{\alpha }t}(x-\bar{x}\boldsymbol e_{1})+\bar{x}\boldsymbol e_{1}\right) $
centered at $(\bar{x},0)$, where $\theta _{0,\varepsilon }$ is defined in (\ref%
{sole}) and $\Omega _{\alpha }$ satisfies
\begin{equation*}
\Omega _{\alpha }=\Omega _{\alpha }^{\ast }+O(\varepsilon ^{\alpha }),
\end{equation*}%
with $\Omega _{\alpha }^{\ast }$ given in (\ref{ang}). Moreover, the set of
solutions $R_{i}(x)$ defined in (\ref{solutions}) parameterizes convex
patches at least of class $C^{1}$.

\item For any $\varepsilon \in (0,\varepsilon _{0}]$ and any $\gamma _{1}\in
\mathbb{R}$, there exists $\gamma _{2}=\gamma _{2}(\varepsilon )$ such that (%
\ref{1-1}) has a global traveling asymmetric vortex patch pair solution $%
\theta _{\varepsilon }(x,t)=\theta _{0,\varepsilon }(x-tU_{\alpha }%
\boldsymbol e_{2})$, where $\theta _{0,\varepsilon }$ is defined in (\ref%
{initial}) and $U_{\alpha }$ satisfies
\begin{equation*}
U_{\alpha }=U_{\alpha }^{\ast }+O(\varepsilon ^{\alpha }),
\end{equation*}%
with $U_{\alpha }^{\ast }$ given in (\ref{tra}). Moreover, the set of
solutions $R_{i}(x)$ defined in (\ref{solutions}) parameterizes convex
patches at least of class $C^{1}$.
\end{enumerate}
\end{theorem}

Note that we arrive at the vortex point pairs when $\varepsilon =0.$ Then,
it allows us to recover the classical result concerning to the case of two
point vortices with magnitudes $2\pi \gamma _{1}$ and $2\pi \gamma _{2}$
separated from each other by $2d,$ which rotate uniformly about their
centroid $(\bar{x},0)$ with the angular velocity
\begin{equation}
\Omega _{\alpha }^{\ast }:=%
\begin{cases}
\frac{\gamma _{1}+\gamma _{2}}{2d^{3}}, & \quad \alpha =1, \\
\frac{\alpha C_{\alpha }(\gamma _{1}+\gamma _{2})}{2d^{2+\alpha }}, & \quad
1<\alpha <2,%
\end{cases}
\label{ang}
\end{equation}%
provided that $\gamma _{1}+\gamma _{2}\neq 0$. Similarly, we can recover the
classical result concerning to the traveling case with velocity
\begin{equation}
U_{\alpha }^{\ast }:=%
\begin{cases}
\frac{\gamma _{1}}{2d^{2}}, & \quad \alpha =1, \\
\frac{\alpha C_{\alpha }\gamma _{1}}{2d^{1+\alpha }}, & \quad 1<\alpha <2,%
\end{cases}
\label{tra}
\end{equation}%
when $\gamma _{1}$ and $\gamma _{2}$ are opposite.

Comparing with the works \cite{has2,mutipole}, the adaptation for $\alpha
\in \lbrack 1,2)$ in the asymmetric case is non-straightforward and involves
some different ingredients, especially in the analysis of the diffeomorphism
property of the nonlinear functional associated to the
problem (see, e.g., Lemma \ref{lem2-4}). In turn, comparing with the
symmetric case, the analysis of the singular velocities are more subtle
(see, e.g., (\ref{2-3}), (\ref{fi3eps}) and (\ref{Fi3})). Also, considering $%
\gamma _{1}=\gamma _{2}=1$ in Theorem \ref{thm:informal}, we recover the
results of the symmetric case obtained in \cite{cao,Cas1}.

The outline of this paper is as follows. In Section 2, we discuss some
general facts concerning the desingularization of solutions to the gSQG
equations as well as give the definition of the function spaces and the
derivation of suitable functionals. Moreover, we obtain the contour
equations governing the co-rotating and traveling asymmetric patch pairs. In
Section 3, we prove our main result which is related to the global existence
of co-rotating asymmetric patch pair solutions for the gSQG equations with $%
\alpha \in \lbrack 1,2)$. Later, the existence of traveling asymmetric patch
pair solutions to the gSQG equations is developed in Section 4.

\section{Preliminaries and desingularization}

This section is devoted to define the function spaces used in the current
paper together with some technical results related to the desingularization
of co-rotating pairs of vortices when the magnitudes are different one from
the other. We shall also discuss about the equations that governs the
dynamics of the boundary. In this way, one can see it as a system of two
periodic nonlocal equations of nonlinear type which plays a central role on
our spectral study.

\subsection{Functional spaces}

In this part, we specify the notation used throughout the paper. We begin by
defining the mean value of the function $f$ on the unit circle by
\begin{equation*}
\int \!\!\!\!\!\!\!\!\!\;{}-{}p(\tau )d\tau :=\frac{1}{2\pi }\int_{0}^{2\pi
}p(\tau )d\tau .
\end{equation*}%
In order to obtain to desired regularity of the nonlinear functional $%
F^{\alpha }$ and $G^{\alpha }$ we need to define some specific Banach spaces
$X^{k+\log }$, $X^{k}$, $Y^{k}$ and $Y_{0}^{k}.$ We start with
\begin{equation*}
X^{k}=\left\{ p\in H^{k}\times H^{k},\ p(x)=\sum\limits_{j=1}^{\infty
}a_{j}\cos (jx),\,a_{j}\in \mathbb{R}^{2},\,x\in \lbrack 0,2\pi )\right\} .
\end{equation*}%
For the case $\alpha =1,$ we use the space
\begin{equation*}
\begin{split}
X^{k+\log }=& \left\{ p\in H^{k}\times H^{k},\
p(x)=\sum\limits_{j=1}^{\infty }a_{j}\cos (jx),\ \left\Vert \int_{0}^{2\pi }%
\frac{\partial ^{k}p(x-y)-\partial ^{k}p(x)}{|\sin (\frac{y}{2})|}%
dy\right\Vert _{L^{2}}<\infty ,\right. \\
& \qquad \qquad \mbox{for}\,\,a_{j}\in \mathbb{R}^{2},\,x\in \lbrack 0,2\pi )%
\Bigg\}.
\end{split}%
\end{equation*}%
For $1<\alpha <2$, we need the following fractional space
\begin{equation*}
\begin{split}
X^{k+\alpha -1}=& \left\{ p\in H^{k}\times H^{k},\
p(x)=\sum\limits_{j=1}^{\infty }a_{j}\cos (jx),\ \left\Vert \int_{0}^{2\pi }%
\frac{\partial ^{k}p(x-y)-\partial ^{k}p(x)}{|\sin (\frac{y}{2})|^{\alpha }}%
dy\right\Vert _{L^{2}}<\infty \right\} . \\
&
\end{split}%
\end{equation*}%
Also, we define the spaces%
\begin{equation*}
Y^{k}=\left\{ p\in H^{k}\times H^{k},\ p(x)=\sum\limits_{j=1}^{\infty
}a_{j}\sin (jx),\,a_{j}\in \mathbb{R}^{2},\,x\in \lbrack 0,2\pi )\right\}
\end{equation*}%
and
\begin{equation*}
Y_{0}^{k}=Y^{k}/\text{span}\{\sin (x)\}=\left\{ p\in Y^{k},\
p(x)=\sum\limits_{j=2}^{\infty }a_{j}\sin (jx),\,a_{j}\in \mathbb{R}%
^{2},\,x\in \lbrack 0,2\pi )\right\} ,
\end{equation*}%
where $a_{j}:=(a_{j}^{1},a_{j}^{2})$. Notice that the norm of $X^{k}$, $%
Y^{k} $ and $Y_{0}^{k}$ is the well-known $H^{k}$-norm on the torus. For the
fractional spaces $X^{k+\log }$ and $X^{k+\alpha -1}$, their norms are given
by a sum of $H^{k}$-norm plus a finite integral that was defined above. More
precisely, we consider the norms
\begin{equation*}
\left\Vert p\right\Vert _{X^{k+\log }}=\left\Vert p\right\Vert
_{H^{k}}+\left\Vert \int_{0}^{2\pi }\frac{\partial ^{k}p(x-y)-\partial
^{k}p(x)}{|\sin (\frac{y}{2})|}dy\right\Vert _{L^{2}}
\end{equation*}%
and
\begin{equation*}
\left\Vert p\right\Vert _{X^{k+\alpha -1}}=\left\Vert p\right\Vert
_{H^{k}}+\left\Vert \int_{0}^{2\pi }\frac{\partial ^{k}p(x-y)-\partial
^{k}p(x)}{|\sin (\frac{y}{2})|^{\alpha }}dy\right\Vert _{L^{2}}.
\end{equation*}%
We have the useful relations $X^{k+\mu }\subset X^{k+\log }\subset X^{k}$,
for every $\mu >0$. From now on, we always assume $k\geq 3$.

\subsection{Co-rotating vortex pairs}

We begin this section by studying co-rotating asymmetric patch pairs of
simply connected patches $D_{i}^{\varepsilon }$, $i=1,2,$ for the gSQG
equations. To this end, we shall fix the following convention and notation.
Let $\varepsilon \in (0,1),$ $b_{1},b_{2}\in \lbrack 0,+\infty )$ and let $%
\gamma _{1},\gamma _{2}$ be real numbers such that $d>2(b_{1}+b_{2})$. We
denote by $D_{1}^{\varepsilon }$ and $D_{2}^{\varepsilon }$ two simply
connected domains, which are closing to the unit disk and contained in the
disk with radius $2$, where both disks are centered at the origin. Suppose
that the initial data $\theta _{0,\varepsilon }$ is a co-rotating asymmetric
patch pair defined as
\begin{equation}
\theta _{0,\varepsilon }=\frac{\gamma _{1}}{\varepsilon ^{2}b_{1}^{2}}\chi _{%
\tilde{D}_{1}^{\varepsilon }}+\frac{\gamma _{2}}{\varepsilon ^{2}b_{2}^{2}}%
\chi _{\tilde{D}_{2}^{\varepsilon }},  \label{sole}
\end{equation}%
where $\tilde{D}_{1}^{\varepsilon }$, $\tilde{D}_{2}^{\varepsilon }\subset
\mathbb{R}^{2}$ satisfy
\begin{equation*}
\tilde{D}_{1}^{\varepsilon }=\varepsilon b_{1}D_{1}^{\varepsilon }\quad
\text{\textnormal{and}}\quad \tilde{D}_{2}^{\varepsilon }=-\varepsilon
b_{2}D_{2}^{\varepsilon }+d\,\boldsymbol{e}_{1} .
\end{equation*}%
We assume in this section that the initial data $\theta _{0,\varepsilon }$
is of type co-rotating pair of patches around the centroid of the system $%
\bar{x}\boldsymbol e_{1}\in \mathbb{R}^{2}$. Hence, we want to obtain the
existence of a new family of solutions for the gSQG with $\alpha \in \lbrack
1,2)$, namely co-rotating asymmetric vortex patch pairs about the point $%
(\bar{x},0)$. In this situation, a co-rotating patch pair with center at $%
(\bar{x},0)$ is given by
\begin{equation*}
\theta _{\varepsilon }(x-\bar{x}\boldsymbol e_{1},t)=\theta _{0,\varepsilon
}\left( Q_{\Omega t}(x-\bar{x}\boldsymbol e_{1})+\bar{x}\boldsymbol e_{1}\right)
,
\end{equation*}%
where $Q_{\omega }$ denotes the counter-clockwise rotation operator of angle
$\omega $ with respect to the origin, $\Omega $ is a fixed uniform angular
speed and $\boldsymbol e_{1}=(1,0)$ is the unitary vector. Inserting the
above equation into (\ref{1-1}) yields
\begin{equation*}
\left( \mathbf{v}_{0}(x)+\Omega (x-\bar{x}\boldsymbol e_{1})^{\perp }\right)
\cdot \nabla \theta _{0,\varepsilon }(x)=0,
\end{equation*}%
which is equivalent to
\begin{equation*}
\left( \mathbf{v}_{0}(x)+\Omega (x-\bar{x}\boldsymbol{e}_{1})^{\perp }\right)
\cdot \mathbf{n}( x)=0,\ \ \ \forall \,x\in \partial \tilde{D}_{1}^{\varepsilon }\cup
\partial \tilde{D}_{2}^{\varepsilon },
\end{equation*}%
where $\mathbf{n}(x)$ stands for the unitary normal vector. We are interested in
patches closing to the unit disk $\tilde{D}_{i}^{\varepsilon }$. Thus, we
can assume that the patch can be parameterized by
\begin{equation*}
z_{i}(x)=\left( \varepsilon b_{i}R_{i}(x)\cos (x),\varepsilon
b_{i}R_{i}(x)\sin (x)\right), \quad \mbox{for}\quad i=1,2,
\end{equation*}%
where
\begin{equation}\label{solutions}
R_{i}(x)=1+\varepsilon ^{1+\alpha }b_{i}^{1+\alpha }p_{i}(x), \quad \mbox{for}\quad i=1,2,
\end{equation}
with $x\in \lbrack 0,2\pi )$, and $p_{i}$ a $C^{1}$ function. We start by
writing up the system for $R_{i}(x)$ and then we use its expression in order
to obtain our system in terms of $p_{i}(x)$. Note that constructing
co-rotating asymmetric solution pairs for the gSQG equations is equivalent
to finding a zero of the following equations
\begin{equation}
\begin{split}
&-\Omega \left( \varepsilon b_{1}R_{1}(x)R^{\prime }(x)-\bar{x}(R_{1}^{\prime
}(x)\cos (x)-R_{1}(x)\sin (x))\right)  \\
& {\resizebox{.98\hsize}{!}{$+\frac{\gamma_1
C_\alpha}{\varepsilon^{1+\alpha}b_1^{1+\alpha} }\int\!\!\!\!\!\!\!\!\!\;
{}-{}
\frac{\left((R_1(x)R_1(y)+R_1'(x)R_1'(y))%
\sin(x-y)+(R_1(x)R_1'(y)-R_1'(x)R_1(y))\cos(x-y)\right)dy}{\left|
\left(R_1(x)-R_1(y)\right)^2+4R_1(x)R_1(y)\sin^2\left(\frac{x-y}{2}\right)%
\right|^{\frac{\alpha}{2}}}$}} \\
& {\resizebox{.98\hsize}{!}{$+\frac{\gamma_2 C_\alpha}{\varepsilon b_2 }
\int\!\!\!\!\!\!\!\!\!\; {}-{}
\frac{\left((R_2(x)R_2(y)+R_2'(x)R_2'(y))%
\sin(x-y)+(R_2(x)R_2'(y)-R_2'(x)R_2(y))\cos(x-y)\right)dy}{\left|(%
\varepsilon b_2 R_2(y)\cos(y)+\varepsilon b_1
R_1(x)\cos(x)-d)^2+(\varepsilon b_2 R_2(y)\sin(y)+\varepsilon b_1
R_1(x)\sin(x))^2\right|^{\frac{\alpha}{2}}}$}} \\
& \qquad =0,
\end{split}
\label{1-3}
\end{equation}%
and
\begin{equation}
\begin{split}
& -\Omega \left( \varepsilon b_{2}R_{2}(x)R_{2}^{\prime }(x)+\bar{x}
(R_{2}^{\prime }(x)\cos (x)-R_{2}(x)\sin (x))-dR_{2}^{\prime }(x)\cos
(x)+dR_{2}(x)\sin (x))\right)  \\
& {\resizebox{.98\hsize}{!}{$+\frac{\gamma_2
C_\alpha}{\varepsilon^{1+\alpha}b_2^{1+\alpha} }\int\!\!\!\!\!\!\!\!\!\;
{}-{}
\frac{\left((R_2(x)R_2(y)+R_2'(x)R_2'(y))%
\sin(x-y)+(R_2(x)R_2'(y)-R_2'(x)R_2(y))\cos(x-y)\right)dy}{\left|
\left(R_2(x)-R_2(y)\right)^2+4R_2(x)R_2(y)\sin^2\left(\frac{x-y}{2}\right)%
\right|^{\frac{\alpha}{2}}}$}} \\
& {\resizebox{.98\hsize}{!}{$+\frac{\gamma_1 C_\alpha}{\varepsilon b_1 }
\int\!\!\!\!\!\!\!\!\!\; {}-{}
\frac{\left((R_1(x)R_1(y)+R_1'(x)R_1'(y))%
\sin(x-y)+(R_1(x)R_1'(y)-R_1'(x)R_1(y))\cos(x-y)\right)dy}{\left|(%
\varepsilon b_1 R_1(y)\cos(y)+\varepsilon b_2
R_2(x)\cos(x)-d)^2+(\varepsilon b_1 R_1(y)\sin(y)+\varepsilon b_2
R_2(x)\sin(x))^2\right|^{\frac{\alpha}{2}}}$}} \\
& \qquad =0.
\end{split}
\label{1-3b}
\end{equation}%
The above two equations can be unified on a system as follows
\begin{equation}
\begin{split}
& {\resizebox{.98\hsize}{!}{$- \Omega\left(\varepsilon b_i
R_i(x)R_i'(x)+(-1)^i
\cdot(R_i'(x)\cos(x)-R_i(x)\sin(x))-(i-1)d(R_i'(x)\cos(x)-R_i(x)\sin(x))%
\right)$}} \\
& {\resizebox{.98\hsize}{!}{$+\frac{\gamma_i
C_\alpha}{\varepsilon^{1+\alpha}b_i^{1+\alpha} }\int\!\!\!\!\!\!\!\!\!\;
{}-{}
\frac{\left((R_i(x)R_i(y)+R_i'(x)R_i'(y))%
\sin(x-y)+(R_i(x)R_i'(y)-R_i'(x)R_i(y))\cos(x-y)\right)dy}{\left|
\left(R_i(x)-R_i(y)\right)^2+4R_i(x)R_i(y)\sin^2\left(\frac{x-y}{2}\right)%
\right|^{\frac{\alpha}{2}}}$}} \\
& {\resizebox{.98\hsize}{!}{$+\frac{\gamma_{3-i} C_\alpha}{\varepsilon
b_{3-i}} \int\!\!\!\!\!\!\!\!\!\; {}-{}
\frac{\left((R_{3-i}(x)R_{3-i}(y)+R_{3-i}'(x)R_{3-i}'(y))%
\sin(x-y)+(R_{3-i}(x)R_{3-i}'(y)-R_{3-i}'(x)R_{3-i}(y))\cos(x-y)\right)dy}{%
\left|(\varepsilon b_{3-i} R_{3-i}(y)\cos(y)+\varepsilon b_i
R_i(x)\cos(x)-d)^2+(\varepsilon b_{3-i} R_{3-i}(y)\sin(y)+\varepsilon b_i
R_i(x)\sin(x))^2\right|^{\frac{\alpha}{2}}}$}} \\
& \qquad =0,
\end{split}
\label{rot}
\end{equation}%
for $i=1,2.$

\subsection{Traveling vortex pairs}

The second family of solutions in which we are interested in constructing
are the traveling patches associated to the gSQG equations. As in the case
of co-rotating case, we denote by $D_{1}^{\varepsilon }$ and $%
D_{2}^{\varepsilon }$ two simply connected domains, which are small
perturbations of the unit disk and contained in the disk with radius $2$,
where both disks are centered at the origin. Let $b_{1},b_{2}\in \lbrack
0,+\infty )$ and $\varepsilon \in (0,1)$ satisfy $d>2(b_{1}+b_{2})$. Let $%
\tilde{D}_{i}^{\varepsilon }$, $i=1,2,$ be two domains defined as
\begin{equation}
\tilde{D}_{1}^{\varepsilon }=\varepsilon b_{1}D_{1}^{\varepsilon }\quad
\mbox{and}\quad \tilde{D}_{2}^{\varepsilon }=-\varepsilon
b_{2}D_{2}^{\varepsilon }+d\boldsymbol e_{1}.  \label{domains rot}
\end{equation}%
In this situation, for any given real numbers $\gamma _{1},\gamma _{2}$, the
initial vorticity data has the following form
\begin{equation}
\theta _{0,\varepsilon }(x)=\frac{\gamma _{1}}{\varepsilon ^{2}b_{1}^{2}}%
\boldsymbol\chi _{\tilde{D}_{1}^{\varepsilon }}-\frac{\gamma _{2}}{%
\varepsilon ^{2}b_{2}^{2}}\boldsymbol\chi _{\tilde{D}_{2}^{\varepsilon }},
\label{initial}
\end{equation}%
which is composed of a pair of simply connected vorticities with
unequal-size magnitudes of the vorticity, namely $\gamma _{1}$ and $\gamma
_{2}$. We also assume that these patches travel steadily in $(Oy)$ direction
with constant uniform velocity $U$.

Note that the initial vorticity data $\theta _{0,\varepsilon }$ is composed
for a traveling pair of patches with center $\bar{x}_{0}\boldsymbol %
e_{1}\in \mathbb{R}^{2}$. Thus, the traveling patch pair with centroid $%
\bar{x} $ is given by
\begin{equation*}
\theta _{\varepsilon }(x,t)=\theta _{0,\varepsilon }(x-tU\boldsymbol e_{2}),
\end{equation*}%
where $U$ is some uniform fixed speed. According to (\ref{1-1}), we arrive
at
\begin{equation*}
(\mathbf{v}_{0}(x)-U\boldsymbol e_{2})\cdot \nabla \theta _{0,\varepsilon
}(x)=0.
\end{equation*}%
In other words, the above system can be expressed as
\begin{equation*}
(\mathbf{v}_{0}(x)-U\boldsymbol e_{2})\cdot\mathbf{n}( x)=0,\ \ \ \forall \,x\in
\partial \tilde{D}_{1}^{\varepsilon }\cup \partial \tilde{D}%
_{2}^{\varepsilon }.
\end{equation*}%
Once again, our study is reduced to finding an asymmetric traveling solution
patch pairs associated to the gSQG. In other words, we are interesting in
finding a zero of the following equations
\begin{equation}
\begin{split}
& -U\left( R_{1}(x)\cos (x)-R_{1}^{\prime }(x)\sin (x)\right) \\
& {\resizebox{.98\hsize}{!}{$+\frac{\gamma_1
C_\alpha}{\varepsilon^{1+\alpha}b_1^{1+\alpha} }\int\!\!\!\!\!\!\!\!\!\;
{}-{}
\frac{\left((R_1(x)R_1(y)+R_1'(x)R_1'(y))%
\sin(x-y)+(R_1(x)R_1'(y)-R_1'(x)R_1(y))\cos(x-y)\right)dy}{\left|
\left(R_1(x)-R_1(y)\right)^2+4R_1(x)R_1(y)\sin^2\left(\frac{x-y}{2}\right)%
\right|^{\frac{\alpha}{2}}}+$}} \\
& {\resizebox{.98\hsize}{!}{$+\frac{\gamma_2 C_\alpha}{\varepsilon b_2 }
\int\!\!\!\!\!\!\!\!\!\; {}-{}
\frac{\left((R_2(x)R_2(y)+R_2'(x)R_2'(y))%
\sin(x-y)+(R_2(x)R_2'(y)-R_2'(x)R_2(y))\cos(x-y)\right)dy}{\left|(%
\varepsilon b_2 R_2(y)\cos(y)+\varepsilon b_1
R_1(x)\cos(x)-d)^2+(\varepsilon b_2 R_2(y)\sin(y)+\varepsilon b_1
R_1(x)\sin(x))^2\right|^{\frac{\alpha}{2}}}$}} \\
& \qquad =0
\end{split}
\label{t1}
\end{equation}%
and
\begin{equation}
\begin{split}
& -U\left( R_{2}(x)\cos (x)-R_{2}^{\prime }(x)\sin (x)\right) \\
& {\resizebox{.98\hsize}{!}{$+\frac{\gamma_2
C_\alpha}{\varepsilon^{1+\alpha}b_2^{1+\alpha} }\int\!\!\!\!\!\!\!\!\!\;
{}-{}
\frac{\left((R_2(x)R_2(y)+R_2'(x)R_2'(y))%
\sin(x-y)+(R_2(x)R_2'(y)-R_2'(x)R_2(y))\cos(x-y)\right)dy}{\left|
\left(R_2(x)-R_2(y)\right)^2+4R_2(x)R_2(y)\sin^2\left(\frac{x-y}{2}\right)%
\right|^{\frac{\alpha}{2}}}$}} \\
& {\resizebox{.98\hsize}{!}{$+\frac{\gamma_1 C_\alpha}{\varepsilon b_1 }
\int\!\!\!\!\!\!\!\!\!\; {}-{}
\frac{\left((R_1(x)R_1(y)+R_1'(x)R_1'(y))%
\sin(x-y)+(R_1(x)R_1'(y)-R_1'(x)R_1(y))\cos(x-y)\right)dy}{\left|(%
\varepsilon b_1 R_1(y)\cos(y)+\varepsilon b_2
R_2(x)\cos(x)-d)^2+(\varepsilon b_1 R_1(y)\sin(y)+\varepsilon b_2
R_2(x)\sin(x))^2\right|^{\frac{\alpha}{2}}}$}} \\
& \qquad =0.
\end{split}
\label{t2}
\end{equation}%
The above two equations can be unified on a system as in the co-rotating
case as follows
\begin{equation}
\begin{split}
& U\left( R_{i}(x)\cos (x)-R_{i}^{\prime }(x)\sin (x)\right) \\
& {\resizebox{.98\hsize}{!}{$+\frac{\gamma_i
C_\alpha}{\varepsilon^{1+\alpha}b_i^{1+\alpha} }\int\!\!\!\!\!\!\!\!\!\;
{}-{}
\frac{\left((R_i(x)R_i(y)+R_i'(x)R_i'(y))%
\sin(x-y)+(R_i(x)R_i'(y)-R_i'(x)R_i(y))\cos(x-y)\right)dy}{\left|
\left(R_i(x)-R_i(y)\right)^2+4R_i(x)R_i(y)\sin^2\left(\frac{x-y}{2}\right)%
\right|^{\frac{\alpha}{2}}}$}} \\
& {\resizebox{.98\hsize}{!}{$+\frac{\gamma_{3-i} C_\alpha}{\varepsilon
b_{3-i} } \int\!\!\!\!\!\!\!\!\!\; {}-{}
\frac{\left((R_{3-i}(x)R_{3-i}(y)+R_{3-i}'(x)R_{3-i}'(y))%
\sin(x-y)+(R_{3-i}(x)R_{3-i}'(y)-R_{3-i}'(x)R_{3-i}(y))\cos(x-y)\right)dy}{%
\left|(\varepsilon b_{3-i} R_{3-i}(x)\cos(x)+\varepsilon b_i
R_i(y)\cos(y)-d)^2+(\varepsilon b_{3-i} R_{3-i}(x)\sin(x)+\varepsilon b_i
R_i(y)\sin(y))^2\right|^{\frac{\alpha}{2}}}$}} \\
& \qquad =0,
\end{split}
\label{t3}
\end{equation}%
for $i=1,2.$

\section{Existence of asymmetric co-rotating solutions for the gSQG equations%
}

In this section, we try to extend the range of $\varepsilon $ to $(-\frac{1}{%
2},\frac{1}{2})$ in the case of asymmetric patch pairs that are small
perturbation of the unit disk. In this case, we can assume that the patch
can be parameterized as $R_{i}(x)=1+\varepsilon ^{1+\alpha }b_{i}^{1+\alpha
}p_{i}(x)$, for $i=1,2$. With this in mind, now we describe the set of
solutions of the system $F_{i}^{\alpha }=0$ around the point $(\varepsilon ,\Omega
,\bar{x},p_{1},p_{2})$. Thus, we consider $F_{i}^{\alpha }$
with the form
\begin{equation*}
F_{i}^{\alpha }=F_{i1}+F_{i2}+F_{i3},\,\mbox{for}\quad i=1,2,
\end{equation*}%
where
\begin{equation}
\begin{split}
& F_{i1}=-\Omega \left( |\varepsilon |^{2+\alpha }b_{i}^{2+\alpha
}p_{i}^{\prime }(x)+(-1)^{i}\bar{x} \left( \frac{\varepsilon ^{1+\alpha
}b_{i}^{1+\alpha }p_{i}^{\prime }(x)\cos (x)}{1+\varepsilon ^{1+\alpha
}b_{i}^{1+\alpha }p_{i}(x)}-\sin (x)\right) \right. \\
& \qquad \qquad \left. -(i-1)d\left( \frac{|\varepsilon |^{1+\alpha
}b_{i}^{1+\alpha }p_{i}^{\prime }(x)\cos (x)}{1+\varepsilon ^{1+\alpha
}b_{i}^{1+\alpha }p_{i}(x)}-\sin (x)\right) \right. \biggl),
\end{split}
\label{2-1}
\end{equation}

\begin{equation}
\begin{split}
& {\resizebox{.98\hsize}{!}{$
F_{i2}=\frac{C_\alpha\gamma_i}{\varepsilon|\varepsilon|^{\alpha}b_i^{1+%
\alpha}}\int\!\!\!\!\!\!\!\!\!\; {}-{}
\frac{(1+\varepsilon|\varepsilon|^\alpha b^{1+\alpha}_i
p_i(y))\sin(x-y)dy}{\left(
|\varepsilon|^{2+2\alpha}b_i^{2+2\alpha}\left(p_i(x)-p_i(y)\right)^2+4(1+%
\varepsilon|\varepsilon|^\alpha b^{1+\alpha}_i
p_i(x))(1+\varepsilon|\varepsilon|^\alpha b^{1+\alpha}_i
p_i(y))\sin^2\left(\frac{x-y}{2}\right)\right)^{\frac{\alpha}{2}}}$}} \\
& +C_{\alpha }\gamma _{i}\,{\resizebox{.9\hsize}{!}{$
\int\!\!\!\!\!\!\!\!\!\; {}-{} \frac{(p'_i(y)-p'_i(x))\cos(x-y)dy}{\left(
|\varepsilon|^{2+2\alpha}b_i^{2+2\alpha}\left(p_i(x)-p_i(y)\right)^2+4(1+%
\varepsilon|\varepsilon|^\alpha b^{1+\alpha}_i
p_i(x))(1+\varepsilon|\varepsilon|^\alpha b^{1+\alpha}_i
p_i(y))\sin^2\left(\frac{x-y}{2}\right)\right)^{\frac{\alpha}{2}}}$}} \\
& +{\resizebox{.96\hsize}{!}{$\frac{C_\alpha
\gamma_i\varepsilon|\varepsilon|^\alpha b^{1+\alpha}_i
p'_i(x)}{1+\varepsilon|\varepsilon|^\alpha
b^{1+\alpha}_ip_i(x)}\int\!\!\!\!\!\!\!\!\!\; {}-{}
\frac{(p_i(x)-p_i(y))\cos(x-y)dy}{\left(
|\varepsilon|^{2+2\alpha}b_i^{2+2\alpha}\left(p_i(x)-p_i(y)\right)^2+4(1+%
\varepsilon|\varepsilon|^\alpha b^{1+\alpha}_i
p_i(x))(1+\varepsilon|\varepsilon|^\alpha b^{1+\alpha}_i
p_i(y))\sin^2\left(\frac{x-y}{2}\right)\right)^{\frac{\alpha}{2}}}$}} \\
& +{\resizebox{.96\hsize}{!}{$\frac{C_\alpha\gamma_i\varepsilon|%
\varepsilon|^\alpha b^{1+\alpha}_i }{1+\varepsilon|\varepsilon|^\alpha
b^{1+\alpha}_i p_i(x)}\int\!\!\!\!\!\!\!\!\!\; {}-{}
\frac{p'_i(x)p'_i(y)\sin(x-y)dy}{\left(
|\varepsilon|^{2+2\alpha}b_i^{2+2\alpha}\left(p_i(x)-p_i(y)\right)^2+4(1+%
\varepsilon|\varepsilon|^\alpha b^{1+\alpha}_i
p_i(x))(1+\varepsilon|\varepsilon|^\alpha b^{1+\alpha}_i
p_i(y))\sin^2\left(\frac{x-y}{2}\right)\right)^{\frac{\alpha}{2}}}$}} \\
& =F_{i21}+F_{i22}+F_{i23}+F_{i24},\qquad \mbox{for}\,i=1,2,
\end{split}
\label{2-2}
\end{equation}%
and
\begin{equation}
\begin{split}
& {\resizebox{.98\hsize}{!}{$ F_{i3}=\frac{\gamma_{3-i} C_\alpha
(1+\varepsilon|\varepsilon|^\alpha b_{3-i}^{1+\alpha}
p_{3-i}(x))}{\varepsilon b_{3-i} (1+\varepsilon|\varepsilon|^\alpha
b^{1+\alpha}_i p_i(x)) }\int\!\!\!\!\!\!\!\!\!\; {}-{}
\frac{(1+\varepsilon|\varepsilon|^\alpha b_{3-i}^{1+\alpha}
p_{3-i}(y))\sin(x-y)dy}{\left|(\varepsilon b_{3-i}
R_{3-i}(y)\cos(y)+\varepsilon b_i R_i(x)\cos(x)-d)^2+(\varepsilon b_{3-i}
R_{3-i}(y)\sin(y)+\varepsilon b_i
R_i(x)\sin(x))^2\right|^{\frac{\alpha}{2}}}$}} \\
&
{\resizebox{.98\hsize}{!}{$+\frac{\gamma_{3-i} C_\alpha \varepsilon
|\varepsilon|^{2\alpha}b_{3-i}^{1+2\alpha}}{1+\varepsilon|\varepsilon|^%
\alpha b^{1+\alpha}_i p_i(x)}\int\!\!\!\!\!\!\!\!\!\; {}-{} \frac{
p_{3-i}'(x)p_{3-i}'(y)\sin(x-y)dy}{\left|(\varepsilon b_{3-i}
R_{3-i}(y)\cos(y)+\varepsilon b_i R_i(x)\cos(x)-d)^2+(\varepsilon b_{3-i}
R_{3-i}(y)\sin(y)+\varepsilon b_i
R_i(x)\sin(x))^2\right|^{\frac{\alpha}{2}}}$}} \\
& {\resizebox{.98\hsize}{!}{$ +\frac{\gamma_{3-i} C_\alpha
|\varepsilon|^\alpha (1+\varepsilon|\varepsilon|^\alpha b_{3-i}^{1+\alpha}
p_{3-i}(x))}{1+\varepsilon|\varepsilon|^\alpha b^{1+\alpha}_i
p_i(x)}\int\!\!\!\!\!\!\!\!\!\; {}-{}
\frac{b_{3-i}^{\alpha}(p'_{3-i}(y)-p'_{3-i}(x))\cos(x-y)dy}{\left|(%
\varepsilon b_{3-i} R_{3-i}(y)\cos(y)+\varepsilon b_i
R_i(x)\cos(x)-d)^2+(\varepsilon b_{3-i} R_{3-i}(y)\sin(y)+\varepsilon b_i
R_i(x)\sin(x))^2\right|^{\frac{\alpha}{2}}}$}} \\
& {\resizebox{.98\hsize}{!}{$+\frac{\gamma_{3-i} C_\alpha
\varepsilon|\varepsilon|^{2\alpha}p_i'(x)}{1+\varepsilon|\varepsilon|^\alpha
b^{1+\alpha}_i p_i(x)}\int\!\!\!\!\!\!\!\!\!\; {}-{}
\frac{b_{3-i}^{1+2\alpha}(p_{3-i}(x)-p_{3-i}(y))\cos(x-y)dy}{\left|(%
\varepsilon b_{3-i}^{\alpha} R_{3-i}(y)\cos(y)+\varepsilon b_i
R_i(x)\cos(x)-d)^2+(\varepsilon b_{3-i} R_{3-i}(y)\sin(y)+\varepsilon b_i
R_i(x)\sin(x))^2\right|^{\frac{\alpha}{2}}}$}} \\
& =F_{i31}+F_{i32}+F_{i33}+F_{i34},\qquad \mbox{for}\,i=1,2.
\end{split}
\label{2-3}
\end{equation}%
Now, we outline the proof of our main result. To do this, we shall use the
following notation related to the nonlinear functionals
\begin{equation*}
F^{\alpha }(\varepsilon ,\Omega ,\bar{x},p_{1},p_{2}):=\left( F_{1}^{\alpha
}(\varepsilon ,\Omega ,\bar{x},p_{1},p_{2}),F_{2}^{\alpha }(\varepsilon
,\Omega ,\bar{x},p_{1},p_{2})\right)
\end{equation*}%
and
\begin{equation*}
G^{\alpha }(\varepsilon ,U,\gamma _{2},p_{1},p_{2}):=\left( G_{1}^{\alpha
}(\varepsilon ,U,\gamma _{2},p_{1},p_{2}),G_{2}^{\alpha }(\varepsilon
,U,\gamma _{2},p_{1},p_{2})\right) .
\end{equation*}%
From now on, we define $g=(\Omega ,\bar{x},p_{1},p_{2})$. In order to ensure
the existence of co-rotating asymmetric vortex pairs we split out the proof into several
steps. These steps are given in the form of lemmas and they are specifically
dedicated to seeing if the hypotheses of the implicit function theorem for
the functional $F^{\alpha }(\varepsilon ,g)=0$ are fulfilled. More
precisely, we need to verify $F^{\alpha }(\varepsilon ,g)$ and its Gateaux
derivatives on the direction $h_{i}$ and $h_{3-i},$ for $i=1,2,$ are both
continuous from $\left( -\frac{1}{2},\frac{1}{2}\right) \times \mathbb{R}%
\times \mathbb{R}\times B_{r}$ to $Y^{k-1}$, where $B_{r}$ is a neighborhood
centered at the origin in $X^{k+\log }$ $(\alpha =1)$, or $X^{k+\alpha -1}$ $%
(1<\alpha <2)$.

To this end, we need to compute the linearized operator $%
D_{(p_{1},p_{2})}F^{\alpha }(0,\Omega ,\bar{x},0,0)(h_{1},h_{2})(x)$ which is
given by
\begin{equation*}
D_{(}p_{1},p_{2})F^{\alpha }(0,\Omega ,\bar{x},0,0)(h_{1},h_{2})(x)=%
\begin{pmatrix}
-{\gamma _{1}}\sum\limits_{j=2}^{\infty }a_{j}^{1}\sigma _{j}j\sin (jx), \\
-{\gamma _{2}}\sum\limits_{j=2}^{\infty }a_{j}^{2}\sigma _{j}j\sin (jx),%
\end{pmatrix}%
,
\end{equation*}%
with
\begin{equation}
\sigma _{j}:=%
\begin{cases}
2^{\alpha }\frac{\Gamma (1-\alpha )}{\Gamma (\frac{\alpha }{2})\Gamma (1-%
\frac{\alpha }{2})}\left( \frac{\Gamma (\frac{\alpha }{2})}{\Gamma (1-\frac{%
\alpha }{2})}-\frac{\Gamma (j+\frac{\alpha }{2})}{\Gamma (j+1-\frac{\alpha }{%
2})}\right) , & \text{for}\ \alpha \neq 1, \\
\sum\limits_{i=1}^{j}\frac{8}{2i-1} & \text{for}\ \alpha =1. \\
&
\end{cases}
\label{sigma}
\end{equation}%
This operator is not invertible from $X^{k+\log }$ $(\alpha =1)$, or $%
X^{k+\alpha -1}$ $(1<\alpha <2),$ to $Y^{k-1}$ but it is invertible from $%
X^{k+\log }$ $(\alpha =1)$, or $X^{k+\alpha -1}$ $(1<\alpha <2),$ to $%
Y_{0}^{k}$, for more details see \cite{cao,Cas1}.
One way to deal with this defect is considering the operator $D_{(\Omega
,\bar{x},f_{1},f_{2})}F(0,\Omega^*_{\alpha},\bar{x}_\alpha^*,0,0)h(x)$, where $h=(\beta _{1},\beta
_{2},h_{1},h_{2})\in \mathbb{R}\times \mathbb{R}\times X^{k+\log }$ $(\alpha
=1)$ or $h=(\beta _{1},\beta _{2},h_{1},h_{2})\in \mathbb{R}\times \mathbb{R}%
\times X^{k+\alpha -1}$ $(1<\alpha <2),$ where $g= (\Omega ,\bar{x},p_{1},p_{2})$ and $ g_{0}=
(\Omega^{\ast }_\alpha,\bar{x}^{\ast }_\alpha,0,0),$ satisfying:

\noindent For $\alpha =1,$
\begin{align*}
D_{g}F^{1}(0,g_{0})h(x)=& -\dfrac{\beta _{1}d}{\gamma _{1}+\gamma _{2}}%
\begin{pmatrix}
\gamma _{2} \\
\gamma _{1}%
\end{pmatrix}%
\sin (x)-\beta _{2}\dfrac{\gamma _{1}+\gamma _{2}}{2d^{3}}%
\begin{pmatrix}
1 \\
-1%
\end{pmatrix}%
\sin (x) \\
& \qquad -\sum_{j=2}^{\infty }\,j\sigma _{j}%
\begin{pmatrix}
{\gamma _{1}}a_{j}^{1} \\
{\gamma _{2}}a_{j}^{2}%
\end{pmatrix}%
\sin (jx).
\end{align*}%
For $1<\alpha <2,$
\begin{align*}
D_{g}F^{\alpha }(0,g_{0})h(x)=& -\dfrac{\beta _{1}d}{\gamma _{1}+\gamma _{2}}%
\begin{pmatrix}
\gamma _{2} \\
\gamma _{1}%
\end{pmatrix}%
\sin (x)-\beta _{2}\dfrac{\alpha C_{\alpha }(\gamma _{1}+\gamma _{2})}{%
2d^{2+\alpha }}%
\begin{pmatrix}
1 \\
-1%
\end{pmatrix}%
\sin (x) \\
& \qquad -\sum_{j=2}^{\infty }\,j\sigma _{j}%
\begin{pmatrix}
{\gamma _{1}}a_{j}^{1} \\
{\gamma _{2}}a_{j}^{2}%
\end{pmatrix}%
\sin (jx),
\end{align*}%
which are invertible from $\mathbb{R}\times \mathbb{R}\times X^{k+\log },$
or $\mathbb{R}\times \mathbb{R}\times X^{k+\alpha -1},$ to $Y^{k}$,
depending on the value of $\alpha $ is $(\alpha =1)$ or $(1<\alpha <2)$,
respectively, where $\sigma _{j}$ is defined in \eqref{sigma}. Furthermore,
one also verifies that $D_{g}F^{\alpha }(0,g_{0})$ is an isomorphism from $%
X^{k+\alpha -1}$ $(1<\alpha <2),$ or $X^{k+\log }$ $(\alpha =1),$ to $%
Y^{k-1} $. Hence, we are in position to invoke the implicit function theorem
around the point $(\varepsilon ,\Omega ,\bar{x},p_{1},p_{2})=(0,\Omega
_{\alpha }^{\ast },\bar{x}^{\ast }_\alpha,0,0)$ to guarantee the existence of
non-trivial co-rotating global asymmetric solutions with unequal magnitudes
associated to the gSQG in the range $\alpha \in \lbrack 1,2)$, where $\Omega
_{\alpha }^{\ast }$ and $\bar{x}^{\ast }_\alpha$ are defined in \eqref{ang} and %
\eqref{posb}, respectively.

\subsection{SQG equations ($\protect\alpha=1$)}

Let $B_{r}$ be the open neighborhood centered at zero in the space $
X^{k+\log }$ defined by
\begin{equation*}
B_{r}:=\left\{ p\in X^{k+\log }:\ \Vert p\Vert _{X^{k+\log }}<r\right\}
\end{equation*}%
with $0<r<1$. In what follows, we obtain the $C^{1}-$regularity of the
nonlinear functional $F_{i}^{1}(\varepsilon ,\Omega ,\bar{x},p_{1},p_{2})$.
First, we begin by proving the continuity of the functional and we also see
how to extend to the case $\varepsilon \leq 0$.

\begin{lemma}
\label{lem2-2} The nonlinear functional $F^1_{i}(\varepsilon,
\Omega,\bar{x},p_1,p_2): \left(-\frac{1}{2}, \frac{1}{2}\right)\times \mathbb{R}
\times \mathbb{R}\times B_r \rightarrow Y^{k-1}$ is continuous.
\end{lemma}

\begin{proof}
By applying \eqref{2-1}, we can rewrite $F_{i1}:\left(-\frac{1}{2}, \frac{1}{%
2}\right)\times \mathbb{R}\times \mathbb{R} \times B_r \rightarrow Y^{k-1}$
as $p_1$ as
\begin{equation}  \label{2-4}
F_{i1}=-\Omega \left((-1)^i\sin(x)+(i-1)d\sin(x)+\varepsilon|\varepsilon|%
\mathcal{R}_1(\varepsilon,p_1,p_2)\right),
\end{equation}
where $\mathcal{R}_1(\varepsilon,p_1,p_2): \left(-\frac{1}{2}, \frac{1}{2}%
\right)\times \mathbb{R} \times \mathbb{R} \times B_r \rightarrow Y^{k-1}$
is continuous. From this, we can conclude that $F_{i1}$ is continuous.

Now, we focus on $F_{i2}$ given in \eqref{2-2}. First we prove that the
range of $F_{i2}$ belongs to $Y^{k-1}$. Notice that $p_{i}(x)\in B_{r}$ is
an even function, for $i=1,2$, then we can easily deduce that $p_{i}^{\prime
}(x)$ is an odd function, for $i=1,2$, since $b_{i}\in \mathbb{R}^{+},$ for $%
i=1,2$. Now, by replacing $y$ with $-y$ in $F_{i2}$, one can ensures that $%
F_{i2}(\varepsilon ,p_{1},p_{2})$ is an odd function.

Recall that $R_{i}(x)=1+\varepsilon |\varepsilon |b_{i}^{2}p_{i}(x)$, for $%
i=1,2$, then, we observe that the possible singularity could be happen only
when we take zeroth derivative of $F_{i21}$ for $\varepsilon =0$. To fix it,
we invoke the Taylor formula
\begin{equation}
\frac{1}{(A+B)^{\lambda }}=\frac{1}{A^{\lambda }}-\lambda \int_{0}^{1}\frac{B%
}{(A+tB)^{1+\lambda }}dt  \label{2-5}
\end{equation}%
with
\begin{equation*}
A:=4\sin ^{2}\left( \frac{x-y}{2}\right)
\end{equation*}%
and
\begin{equation*}
B:=|\varepsilon |^{4}b_{i}^{4}(p_{i}(x)-p_{i}(y))^{2}+\sin ^{2}\left( \frac{%
x-y}{2}\right) \left( 4\varepsilon |\varepsilon
|b_{i}^{2}p_{i}(x)+4\varepsilon |\varepsilon
|b_{i}^{2}p_{i}(y)+4|\varepsilon |^{4}b_{i}^{4}p_{i}(x)p_{i}(y)\right) ,
\end{equation*}%
which leads us to
\begin{equation*}
\begin{split}
& \frac{1}{\left( |\varepsilon |^{4}b_{i}^{4}\left( p_{i}(x)-p_{i}(y)\right)
^{2}+4(1+\varepsilon |\varepsilon |b_{i}^{2}p_{i}(x))(1+\varepsilon
|\varepsilon |b_{i}^{2}p_{i}(y))\sin ^{2}\left( \frac{x-y}{2}\right) \right)
^{\frac{1}{2}}} \\
& \qquad \qquad \qquad \qquad \qquad \qquad \qquad \qquad
\qquad \qquad =\frac{1}{\left( 4\sin ^{2}\left( \frac{x-y}{2}\right) \right)
^{1/2}}+\varepsilon |\varepsilon |b_{i}^{2}K_{\varepsilon }(x,y).
\end{split}%
\end{equation*}%
An important fact on this decomposition is that $K_{\varepsilon }(x,y)$ is
not singular at $x=y$. Since the sine function is odd, one can concludes
\begin{equation}
F_{i21}=\frac{\gamma _{i}}{4}\int \!\!\!\!\!\!\!\!\!\;{}-{}\frac{%
p_{i}(y)\sin (x-y)dy}{\left\vert \sin (\frac{x-y}{2})\right\vert ^{1/2}}%
+\varepsilon |\varepsilon |\mathcal{R}_{11}(\varepsilon ,p_{1},p_{2}),
\label{F21}
\end{equation}%
where $\mathcal{R}_{11}(\varepsilon ,p_{1},p_{2})$ is not singular in $%
\varepsilon $. We treat first the term $F_{i2}$. We deal with the most
singular term $F_{i22}$. To this end, considering the change $x-y\rightarrow
y$, computing the derivatives $\partial ^{k-1}$ of $F_{i22}$ w.r.t. the
spatial variable $x$, and then changing back $y\rightarrow x-y,$ we get
\begin{equation*}
\begin{split}
& \partial ^{k-1}F_{i22}=\gamma _{i}\int \!\!\!\!\!\!\!\!\!\;{}-{}\frac{%
(\partial ^{k}p_{i}(y)-\partial ^{k}p_{i}(x))\cos (x-y)dy}{\left(
|\varepsilon |^{4}b_{i}^{4}\left( p_{i}(x)-p_{i}(y)\right)
^{2}+4(1+\varepsilon |\varepsilon |b_{i}^{2}p_{i}(x))(1+\varepsilon
|\varepsilon |b_{i}^{2}p_{i}(y))\sin ^{2}\left( \frac{x-y}{2}\right) \right)
^{\frac{1}{2}}} \\
& \ \ \ \ -\gamma _{i}\varepsilon |\varepsilon |b_{i}^{2}\int
\!\!\!\!\!\!\!\!\!\;{}-{}\frac{\cos (x-y)}{\left( |\varepsilon
|^{4}b_{i}^{4}\left( p_{i}(x)-p_{i}(y)\right) ^{2}+4(1+\varepsilon
|\varepsilon |b_{i}^{2}p_{i}(x))(1+\varepsilon |\varepsilon
|b_{i}^{2}p_{i}(y))\sin ^{2}\left( \frac{x-y}{2}\right) \right) ^{\frac{3}{2}%
}} \\
& \ \ \ \ {\resizebox{.96\hsize}{!}{$\times \left(\varepsilon |\varepsilon|
b_i^2(p_i(x)-p_i(y))(p'_i(x)-p'_i(y))+2((1+\varepsilon|\varepsilon|
b_i^2p_i(x))p'_i(y)+(1+\varepsilon|\varepsilon|
b_i^2p_i(y))p'_i(x))\sin^2(\frac{x-y}{2})\right)$}} \\
& \ \ \ \ \times (\partial ^{k-1}p_{i}(y)-\partial ^{k-1}p_{i}(x))dy+l.o.t,
\end{split}%
\end{equation*}%
where $l.o.t$ stands for lower order terms. Notice that $\Vert \partial
^{j}p_{i}\Vert _{L^{\infty }}\leq C\Vert p_{i}\Vert _{X^{k+\log }}<\infty $
for $j=0,1,2$, since $p_{i}(x)\in X^{k+\log }$, for $k\geq 3$. Now, by
combining H\"{o}lder inequality and mean value theorem, we arrive at
\begin{equation*}
\begin{split}
\left\Vert \partial ^{k-1}F_{i22}\right\Vert _{L^{2}}& \leq C\left\Vert \int
\!\!\!\!\!\!\!\!\!\;{}-{}\frac{\partial ^{k}p_{i}(x)-\partial ^{k}p_{i}(y)}{%
|\sin (\frac{x-y}{2})|}dy\right\Vert _{L^{2}}+C\left\Vert \int
\!\!\!\!\!\!\!\!\!\;{}-{}\frac{\partial ^{k-1}p_{i}(x)-\partial
^{k-1}p_{i}(y)}{|\sin (\frac{x-y}{2})|}dy\right\Vert _{L^{2}} \\
& \leq C\Vert p_{i}\Vert _{X^{k+\log }}+C\Vert p_{i}\Vert _{X_{\log
}^{k-1}}<\infty .
\end{split}%
\end{equation*}%
Now, note that $F_{i23}$ is less singular than $F_{i22}$. Hence, it is
straightforward to obtain an upper bound for $\left\Vert \partial
^{k-1}F_{i23}\right\Vert _{L^{2}}$. We focus now on the last term $F_{i24}$,
and take $k-1$ derivatives with respect to the spatial variable $x$ to get%
\begin{equation*}
\begin{split}
& {\resizebox{.99\hsize}{!}{$
\partial^{k-1}F_{i24}=\frac{-\gamma_i\varepsilon^4
b_i^4\,\partial^{k-1}p_i(x)}{(1+\varepsilon|\varepsilon|b_i^2
p_i(x))^2}\int\!\!\!\!\!\!\!\!\!\;
{}-{}\frac{p'_i(x)p'_i(y)\sin(x-y)dy}{\left(
|\varepsilon|^{4}b_i^4\left(p_i(x)-p_i(y)\right)^2+4(1+\varepsilon|%
\varepsilon| b_i^2 p_i(x))(1+\varepsilon|\varepsilon| b_i^2
p_i(y))\sin^2\left(\frac{x-y}{2}\right)\right)^{\frac{1}{2}}}$}} \\
&
{\resizebox{.97\hsize}{!}{$\quad +\frac{\gamma _{i}\varepsilon |\varepsilon |b_{i}^{2}}{1+\varepsilon
|\varepsilon |b_{i}^{2}p_{i}(x)}\int \!\!\!\!\!\!\!\!\!\;{}-{}\frac{\sin
(x-y)}{\left( |\varepsilon |^{4}b_{i}^{4}\left( p_{i}(x)-p_{i}(y)\right)
^{2}+4(1+\varepsilon |\varepsilon |b_{i}^{2}p_{i}(x))(1+\varepsilon
|\varepsilon |b_{i}^{2}p_{i}(y))\sin ^{2}\left( \frac{x-y}{2}\right) \right)
^{\frac{1}{2}}}$}} \\
& \qquad\qquad\qquad \times (p_{i}^{\prime }(x)\partial ^{k}p_{i}(y)+\partial
^{k}p_{i}(x)p_{i}^{\prime }(y))dy \\
& \quad {\resizebox{.96\hsize}{!}{$-\frac{2\gamma _{i}\varepsilon ^{4}b_{i}^{4}}{1+\varepsilon
|\varepsilon |b_{i}^{2}p_{i}(x)}\int \!\!\!\!\!\!\!\!\!\;{}-{}\frac{\sin
(x-y)}{\left( |\varepsilon |^{4}b_{i}^{4}\left( p_{i}(x)-p_{i}(y)\right)
^{2}+4(1+\varepsilon |\varepsilon |b_{i}^{2}p_{i}(x))(1+\varepsilon
|\varepsilon |b_{i}^{2}p_{i}(y))\sin ^{2}\left( \frac{x-y}{2}\right) \right)
^{\frac{3}{2}}}$}} \\
&
\quad {\resizebox{.98\hsize}{!}{$\times \left(\varepsilon |\varepsilon|
b_i^2(p_i(x)-p_i(y))(p'_i(x)-p'_i(y))+2((1+\varepsilon|\varepsilon|
b_i^2p_i(x))p'_i(y)+(1+\varepsilon|\varepsilon|
b_i^2p_i(y))p'_i(x))\sin^2(\frac{x-y}{2})\right)$}} \\
&
\qquad \qquad \times (p_{i}^{\prime }(x)\partial ^{k-1}p_{i}(y)+\partial
^{k-1}p_{i}(x)p_{i}^{\prime }(y))+l.o.t.
\end{split}%
\end{equation*}%
By applying the definition of the space $X^{k+\log }$, one can obtain the
following upper bound
\begin{equation*}
\begin{split}
&{\resizebox{.96\hsize}{!}{$\left\Vert \partial ^{k-1}F_{i24}\right\Vert _{L^{2}} \leq C\varepsilon
|\varepsilon |\left( \varepsilon |\varepsilon |\Vert p_{i}^{\prime }\Vert
_{L^{\infty }}^{2}\Vert \partial ^{k-1}p_{i}\Vert _{L^{2}}+\Vert
p_{i}^{\prime }\Vert _{L^{\infty }}\Vert \partial ^{k}p_{i}\Vert
_{L^{2}}+\varepsilon |\varepsilon |\Vert p_{i}\Vert _{L^{\infty }}\Vert
p_{i}^{\prime }\Vert _{L^{\infty }}^{2}\Vert \partial ^{k-1}p_{i}\Vert
_{L^{2}}\right)$}} \\
&\qquad \qquad\quad \leq C\varepsilon |\varepsilon |\Vert p_{i}\Vert _{X^{k+\log }}<\infty .
\end{split}%
\end{equation*}%
From the above computations we can conclude that the functional $F_{i2}$ belongs to $Y^{k-1}$.

We now move on to verify the continuity of $F_{i2}$. Once again we only
treat with the most singular term $F_{i22}$. Now, we shall use the following
notation. For a general function $g_{i}(x)$, we define
\begin{equation*}
\Delta g_{i}=g_{i}(x)-g_{i}(y),\ \mbox{where}\ g_{i}=g_{i}(x),\ \mbox{and}%
\tilde{g}_{i}=g_{i}(y),\,\mbox{for}\,i=1,2
\end{equation*}%
and
\begin{equation*}
D_{\alpha }(g_{i})=b_{i}^{2+2\alpha }\varepsilon ^{2+2\alpha }\Delta
g_{i}^{2}+4(1+\varepsilon |\varepsilon |^{\alpha }b_{i}^{1+\alpha
}g_{i})(1+\varepsilon |\varepsilon |^{\alpha }b_{i}^{1+\alpha }\tilde{g}%
_{i})\sin ^{2}\left( \frac{x-y}{2}\right) .
\end{equation*}%
Then for $p_{i1},p_{i2}\in B_{r}$, $i=1,2,$ we have
\begin{equation*}
\begin{split}
F_{i22}(\varepsilon ,p_{i1})& -F_{i22}(\varepsilon ,p_{i2})=\gamma _{i}\int
\!\!\!\!\!\!\!\!\!\;{}-{}\frac{(\Delta p_{i1}^{\prime }-\Delta
p_{i2}^{\prime })\cos (x-y)dy}{D_{1}(f_{i1})^{\frac{1}{2}}} \\
& +\left( \gamma _{i}\int \!\!\!\!\!\!\!\!\!\;{}-{}\frac{\Delta
p_{i2}^{\prime }\cos (x-y)dy}{D_{1}(f_{i1})^{\frac{1}{2}}}-\gamma _{i}\int
\!\!\!\!\!\!\!\!\!\;{}-{}\frac{\Delta p_{i2}^{\prime }\cos (x-y)dy}{%
D_{1}(f_{i2})^{\frac{1}{2}}}\right) \\
& =I_{1}+I_{2}.
\end{split}%
\end{equation*}%
It is straightforward to see that $I_{1}$ has the following upper bound $%
\Vert I_{1}\Vert _{Y^{k-1}}\leq C\Vert p_{i1}-p_{i2}\Vert _{X^{k+\log }}$.
To estimate $I_{2}$, we can apply the mean value theorem as follows
\begin{equation}
\begin{split}
& \frac{1}{D_{\alpha }(p_{i1})^{\frac{\alpha }{2}}}-\frac{1}{D_{\alpha
}(p_{i2})^{\frac{\alpha }{2}}}=\frac{\alpha }{2}\frac{D_{\alpha
}(p_{i2})-D_{\alpha }(p_{i1})}{D_{\alpha }(\delta _{x,y}p_{i1}+(1-\delta
_{x,y})p_{i2})^{1-\frac{\alpha }{2}}D_{\alpha }(p_{i1})^{\frac{\alpha }{2}%
}D_{\alpha }(p_{i2})^{\frac{\alpha }{2}}} \\
& {\resizebox{.98\hsize}{!}{$=\frac{\alpha}{2}\frac{b_i^{2+2\alpha}|%
\varepsilon|^{2+2\alpha}(\Delta p_{i2}^2-\Delta
p_{i1}^2)+4\varepsilon|\varepsilon|^\alpha b^{1+\alpha}_i
((p_{i2}-p_{i1})(1+\varepsilon|\varepsilon|^\alpha b^{1+\alpha}_i \tilde
p_{i2})+(\tilde p_{i2}-\tilde p_{i1})(1+\varepsilon|\varepsilon|^\alpha
b^{1+\alpha}_i
p_{i1}))\sin^2(\frac{x-y}{2})}{D_\alpha(\delta_{x,y}p_{i1}+(1-%
\delta_{x,y})p_{i2})^{1-\frac{\alpha}{2}}D_\alpha(p_{i1})^\frac{\alpha}{2}D_%
\alpha(p_{i2})^\frac{\alpha}{2}}$}},
\end{split}
\label{2-7}
\end{equation}%
for some $\delta _{x,y}\in (0,1)$. Set $\alpha =1$, then $D_{1}(g)\sim \sin
^{2}(\frac{x-y}{2})\sim |x-y|^{2}/4$ as $|x-y|\rightarrow 0$ and
\begin{equation*}
\begin{split}
& {\resizebox{.9\hsize}{!}{$ \partial^{k-1}I_2\sim C\int\!\!\!\!\!\!\!\!\!\;
{}-{}\frac{\partial^{k-1}p_{i2}(x)-\partial^{k-1}p_{i2}(y)}{|\sin(%
\frac{x-y}{2})|^\alpha}\times \left(\frac{|\varepsilon|^4b_i^4(\Delta
p_{i2}^2-\Delta
p_{i1}^2)}{|x-y|^2}+4\varepsilon|\varepsilon|b_i^2(p_{i2}-p_{i1}+\tilde
p_{i2}-\tilde p_{i1})\right)dy$}} \\
& \qquad \qquad \qquad +l.o.t.
\end{split}%
\end{equation*}%
Hence, we can easily deduce that $\Vert I_{2}\Vert _{Y^{k-1}}\leq C\Vert
p_{i1}-p_{i2}\Vert _{X^{k+\log }}$. With that in hand, we can conclude that $%
F_{i2}(\varepsilon ,p_{1},p_{2}):\left( -\frac{1}{2},\frac{1}{2}\right)
\times B_{r}\rightarrow Y^{k-1}$ is continuous. Now, one can proceed as
before in applying Taylor formula on $F_{i22}$ and $F_{i23}$ to arrive at
\begin{equation}
\begin{split}
F_{i2}& =\frac{\gamma _{i}}{4}\int \!\!\!\!\!\!\!\!\!\;{}-{}\frac{%
p_{i}(y)\sin (x-y)dy}{|\sin (\frac{x-y}{2})|}+\frac{\gamma _{i}}{2}\int
\!\!\!\!\!\!\!\!\!\;{}-{}\frac{(p_{i}^{\prime }(y)-p_{i}^{\prime }(x))\cos
(x-y)dy}{|\sin (\frac{x-y}{2})|}+\varepsilon |\varepsilon |\mathcal{R}%
_{2}(\varepsilon ,p_{1},p_{2}) \\
& =\frac{\gamma _{i}}{4}\int \!\!\!\!\!\!\!\!\!\;{}-{}\frac{p_{i}(x-y)\sin
(y)dy}{|\sin (\frac{y}{2})|}-\frac{\gamma _{i}}{2}\int
\!\!\!\!\!\!\!\!\!\;{}-{}\frac{(p_{i}^{\prime }(x)-p_{i}^{\prime }(x-y))\cos
(y)dy}{|\sin (\frac{y}{2})|}+\varepsilon |\varepsilon |\mathcal{R}%
_{2}(\varepsilon ,p_{1},p_{2}),
\end{split}
\label{2-8}
\end{equation}%
where $\mathcal{R}_{2}(\varepsilon ,p_{1},p_{2}):\left( -\frac{1}{2},\frac{1%
}{2}\right) \times B_{r}\rightarrow Y^{k-1}$ is also continuous.

Now, we focus on the last component of the functional $F_{i}^{1}$ which is $%
F_{i3}$. Once again to manage the singularities at $\varepsilon =0$, we
resort to the Taylor formula (\ref{2-5}) on the element $F_{i31}$ since it
is the most delicate term. Define
\begin{equation*}
A=d^{2}
\end{equation*}%
and
\begin{equation*}
B=-2d\left( b_{3-i}R_{3-i}(y)\cos (y)+b_{i}R_{i}(x)\cos (x)\right) .
\end{equation*}%
From \eqref{2-3}, we can arrive at%
\begin{equation}
\begin{split}
& F_{i31}=\frac{\gamma _{3-i}(1+\varepsilon |\varepsilon
|b_{3-i}^{2}p_{3-i}(x))}{\varepsilon b_{3-i}(1+\varepsilon |\varepsilon
|b_{i}^{2}p_{i}(x))}\int \!\!\!\!\!\!\!\!\!\;{}-{}\frac{\sin (x-y)dy}{A^{%
\frac{1}{2}}} \\
& \qquad -\frac{1}{2}\frac{\gamma _{3-i}(1+\varepsilon |\varepsilon
|b_{3-i}^{2}p_{3-i}(x))}{b_{3-i}(1+\varepsilon |\varepsilon
|b_{i}^{2}p_{i}(x))}\int \!\!\!\!\!\!\!\!\!\;{}-{}\int_{0}^{1}\frac{%
(B+O(\varepsilon ))\sin (x-y)dtdy}{\left( A+\varepsilon tB+O(\varepsilon
^{2})\right) ^{\frac{3}{2}}} \\
& \qquad \qquad +\frac{\gamma _{3-i}(1+\varepsilon |\varepsilon
|b_{3-i}^{2}p_{3-i}(x))}{b_{3-i}(1+\varepsilon |\varepsilon
|b_{i}^{2}p_{i}(x))}\int \!\!\!\!\!\!\!\!\!\;{}-{}\frac{|\varepsilon
|b_{3-i}^{2}p_{3-i}(y)\sin (x-y)dy}{\left( A+\varepsilon B+O(\varepsilon
^{2})\right) ^{\frac{1}{2}}} \\
& =-\frac{1}{2}\frac{\gamma _{3-i}}{b_{3-i}}\int
\!\!\!\!\!\!\!\!\!\;{}-{}\int_{0}^{1}\frac{B\sin (x-y)dtdy}{\left(
A+\varepsilon tB\right) ^{\frac{3}{2}}}+\varepsilon \mathcal{R}%
_{31}(\varepsilon ,p_{1},p_{2}) \\
& =-\frac{1}{2}\frac{\gamma _{3-i}}{b_{3-i}}\int \!\!\!\!\!\!\!\!\!\;{}-{}%
\frac{B\sin (x-y)dy}{A^{\frac{3}{2}}}+\frac{3}{4}\frac{\gamma _{3-i}}{b_{3-i}%
}\int \!\!\!\!\!\!\!\!\!\;{}-{}\int_{0}^{1}\int_{0}^{1}\frac{(\varepsilon
tB)B\sin (x-y)dsdtdy}{\left( A+\varepsilon stB\right) ^{\frac{5}{2}}}%
+\varepsilon \mathcal{R}_{31}(\varepsilon ,p_{1},p_{2}) \\
& =-\frac{1}{2}\frac{\gamma _{3-i}}{b_{3-i}}\int \!\!\!\!\!\!\!\!\!\;{}-{}%
\frac{B\sin (x-y)dy}{A^{\frac{3}{2}}}+\varepsilon \mathcal{R}%
_{31}(\varepsilon ,p_{1},p_{2}) \\
& =\frac{\gamma _{3-i}}{b_{3-i}d^{2}}\int
\!\!\!\!\!\!\!\!\!\;{}-{}b_{3-i}\cos (y)\sin (x-y)dy+\frac{\gamma
_{3-i}b_{i}\cos (x)}{b_{3-i}d^{2}}\int \!\!\!\!\!\!\!\!\!\;{}-{}\sin
(x-y)dy+\varepsilon \mathcal{R}_{31}(\varepsilon ,p_{1},p_{2}) \\
& =\frac{\gamma _{3-i}}{d^{2}}\int \!\!\!\!\!\!\!\!\!\;{}-{}\cos (y)\sin
(x-y)dy+\varepsilon \mathcal{R}_{31}(\varepsilon ,p_{1},p_{2}) \\
& =\frac{\gamma _{3-i}}{2d^{2}}\sin (x)+\varepsilon \mathcal{R}%
_{31}(\varepsilon ,p_{1},p_{2}),
\end{split}
\label{fi3eps}
\end{equation}%
where we used in some lines the fact that the sine function is odd. Here $%
\mathcal{R}_{31}(\varepsilon ,p_{1},p_{2})$ is not singular in $\varepsilon $%
. In other words, one can conclude that
\begin{equation}
F_{i3}=\frac{\gamma _{3-i}\sin (x)}{2d^{2}}+\varepsilon \mathcal{R}%
_{3}(\varepsilon ,p_{1},p_{2}),  \label{Fi3}
\end{equation}%
where $\mathcal{R}_{3}(\varepsilon ,p_{1},p_{2}):\left( -\frac{1}{2},\frac{1%
}{2}\right) \times B_{r}\rightarrow Y^{k-1}$ is continuous.

Notice that for $z_i\in \partial \tilde{D}_{1}^{\varepsilon }$, we have that
\begin{equation*}
\left\vert (\varepsilon b_{3-i}R_{3-i}(y)\cos (y)+\varepsilon
b_{i}R_{i}(x)\cos (x)-d)^{2}+(\varepsilon b_{3-i}R_{3-i}(y)\sin
(y)+\varepsilon b_{i}R_{i}(x)\sin (x))^{2}\right\vert
\end{equation*}%
has positive lower bounds. Then, we can deduce that $F_{i3}$ is less
singular than $F_{i2}$ and we also have that $F_{i3}(\varepsilon
,p_{1},p_{2}):\left( -\frac{1}{2},\frac{1}{2}\right) \times B_{r}\rightarrow
Y^{k-1}$ is continuous. With these observations, one can concludes that the
nonlinear functional $F_{i}^{1}(\varepsilon ,\Omega
,\bar{x},p_{1},p_{2}):\left( -\frac{1}{2},\frac{1}{2}\right) \times \mathbb{R}%
\times \mathbb{R}\times B_{r}\rightarrow Y^{k-1}$ is continuous.
\end{proof}

For $(\varepsilon ,\Omega ,\bar{x},p_{1},p_{2})\in \left( -\frac{1}{2},\frac{1%
}{2}\right) \times \mathbb{R}\times \mathbb{R}\times B_{r}$ and $%
h_{i},h_{3-i}\in X^{k+\log }$, let
\begin{equation*}
\partial _{p_{i}}F_{i}^{1}(\varepsilon ,\Omega
,\bar{x},p_{1},p_{2})h_{i}:=\lim\limits_{t\rightarrow 0}\frac{1}{t}\left(
F_{i}^{1}(\varepsilon ,\Omega ,\bar{x},p_{i}+th_{i})-F_{i}^{1}(\varepsilon
,\bar{x},\Omega ,p_{1},p_{2})\right),
\end{equation*}%
and
\begin{equation*}
\partial _{p_{3-i}}F_{i}^{1}(\varepsilon ,\Omega
,\bar{x},p_{1},p_{2})h_{3-i}:=\lim\limits_{t\rightarrow 0}\frac{1}{t}\left(
F_{i}^{1}(\varepsilon ,\Omega ,\bar{x},p_{3-i}+th_{3-i})-F_{i}^{1}(\varepsilon
,\bar{x},\Omega ,p_{1},p_{2})\right)
\end{equation*}%
be the Gateaux derivative of $F_{i}^{1}(\varepsilon,\Omega,\bar{x}
,p_{1},p_{2})$ on the direction $h_i$ and $h_{3-i}$, respectively. In the following lemma, we compute the Gateaux derivative of
the nonlinear functional $F_{i}^{1}$ and then we prove that it is in fact
continuous. In other words, we prove that $F_{i}^{1}$ is a $C^{1}$ function.

\begin{lemma}
\label{lem2-3}
For each $(\varepsilon ,\Omega ,\bar{x},p_{1},p_{2})\in \left( -\frac{1}{2},%
\frac{1}{2}\right) \times \mathbb{R}\times \mathbb{R}\times B_{r}$, the
partial derivatives of the functional $F$ that are given by $\partial _{p_{3-i}}F_{i}^{1}(\varepsilon
,\Omega,\bar{x} ,p_{1},p_{2})h_{3-i}$, $\partial _{p_{i}}F_{i}^{1}(\varepsilon
,\Omega,\bar{x} ,p_{1},p_{2})h_{3-i},$ $\partial _{\Omega
}F_{i}^{1}(\varepsilon,\Omega,\bar{x} ,p_{1},p_{2})$ and
$\partial
_{\bar{x}}F_{i}^{1}(\varepsilon,\Omega,\bar{x} ,p_{1},p_{2})$ exist and are
continuous.
\end{lemma}

\begin{proof}
First of all, notice that the most delicate part is to guarantee the
existence and continuity of the Gateaux derivatives $\partial
_{p_{3-i}}F_{i}^{1}(\varepsilon,\Omega,\bar{x} ,p_{1},p_{2})h_{3-i}$, $%
\partial _{p_{i}}F_{i}^{1}(\varepsilon,\Omega,\bar{x} ,p_{1},p_{2})h_{3-i}$
on the direction $h_{i}$ and $h_{3-i}$. Notice that functional $F^{1}_i$ is
linearly dependent on $\Omega $ and $\bar{x}$. Then, it is easy to compute
their derivatives with respect to $\Omega $ and $\bar{x}$ and its proof is
left to the reader. From (\ref{2-1}), it is easy to see that the first term
of the Gateaux derivative with respect to $p_{i}$ is
\begin{equation}
\partial _{p_{i}}F_{i1}(\varepsilon,\Omega,\bar{x} ,p_{1},p_{2})h_{i}=\Omega
\varepsilon |\varepsilon |\partial _{p_{i}}\mathcal{R}_{1}(\varepsilon
,p_{1},p_{2})h_{i},  \label{2-10}
\end{equation}%
which is continuous, where $\mathcal{R}_{1}(\varepsilon ,p_{1},p_{2})$ is
given in (\ref{2-4}). Similarly, for the derivative $\partial _{p_{3-i}}$ we
have
\begin{equation}
\partial _{p_{3-i}}F_{i1}(\varepsilon,\Omega,\bar{x} ,p_{1},p_{2})h_{3-i}=0.
\label{2-10b}
\end{equation}%
Next, we claim that $\partial _{p_{i}}F_{i2}(\varepsilon
,p_{1},p_{2})h_{i}=\partial _{p_{i}}F_{i21}h_{i}+\partial
_{p_{i}}F_{i22}h_{i}+\partial _{p_{i}}F_{i23}h_{i}+\partial
_{p_{i}}F_{i24}h_{i}$ is continuous, where%
\begin{equation}
\begin{split}
& \partial _{p_{i}}F_{i21}h_{i}=\gamma _{i}\int \!\!\!\!\!\!\!\!\!\;{}-{}%
\frac{h_{i}(y)\sin (x-y)dy}{\left( |\varepsilon |^{4}b_{i}^{4}\left(
p_{i}(x)-p_{i}(y)\right) ^{2}+4(1+\varepsilon |\varepsilon
|b_{i}^{2}p_{i}(x))(1+\varepsilon |\varepsilon |b_{i}^{2}p_{i}(y))\sin
^{2}\left( \frac{x-y}{2}\right) \right) ^{\frac{1}{2}}} \\
& \ \ \ \ -\frac{2\varepsilon |\varepsilon |b_{i}^{2}\gamma _{i}}{%
2\varepsilon |\varepsilon |b_{i}^{2}}\int \!\!\!\!\!\!\!\!\!\;{}-{}\frac{%
(1+\varepsilon |\varepsilon |b_{i}^{2}p_{i}(y))\sin (x-y)}{\left(
|\varepsilon |^{4}b_{i}^{4}\left( p_{i}(x)-p_{i}(y)\right)
^{2}+4(1+\varepsilon |\varepsilon |b_{i}^{2}p_{i}(x))(1+\varepsilon
|\varepsilon |b_{i}^{2}p_{i}(y))\sin ^{2}\left( \frac{x-y}{2}\right) \right)
^{\frac{3}{2}}} \\
& \ \ \ \ -\gamma _{i}\int \!\!\!\!\!\!\!\!\!\;{}-{}\frac{(1+\varepsilon
|\varepsilon |b_{i}^{2}p_{i}(y))\sin (x-y)}{\left( |\varepsilon
|^{4}b_{i}^{4}\left( p_{i}(x)-p_{i}(y)\right) ^{2}+4(1+\varepsilon
|\varepsilon |b_{i}^{2}p_{i}(x))(1+\varepsilon |\varepsilon
|b_{i}^{2}p_{i}(y))\sin ^{2}\left( \frac{x-y}{2}\right) \right) ^{\frac{3}{2}%
}} \\
& \ \ \ \ \times \bigg(b_{i}^{2}\varepsilon |\varepsilon
|(p_{i}(x)-p_{i}(y))(h_{i}(x)-h_{i}(y)) \\
& \ \ \ \ \ \ \ \ +2(h_{i}(x)(1+\varepsilon |\varepsilon
|b_{i}^{2}p_{i}(y))+h_{i}(y)(1+\varepsilon |\varepsilon
|b_{i}^{2}p_{i}(x)))\sin ^{2}(\frac{x-y}{2})\bigg)dy,
\end{split}
\label{2-11}
\end{equation}%
\begin{equation}
\begin{split}
& \partial _{p_{i}}F_{i22}h_{i}=\gamma _{i}\int \!\!\!\!\!\!\!\!\!\;{}-{}%
\frac{(h_{i}^{\prime }(y)-h_{i}^{\prime }(x))\cos (x-y)dy}{\left(
|\varepsilon |^{4}b_{i}^{4}\left( p_{i}(x)-p_{i}(y)\right)
^{2}+4(1+\varepsilon |\varepsilon |b_{i}^{2}p_{i}(x))(1+\varepsilon
|\varepsilon |b_{i}^{2}p_{i}(y))\sin ^{2}\left( \frac{x-y}{2}\right) \right)
^{\frac{1}{2}}} \\
& {\resizebox{.87\hsize}{!}{$\quad-\frac{\varepsilon|\varepsilon|b_i^2
\gamma_i}{2}\int\!\!\!\!\!\!\!\!\!\; {}-{}
\frac{(p'_i(y)-p'_i(x))\cos(x-y)}{\left(
|\varepsilon|^4b_i^4\left(p_i(x)-p_i(y)\right)^2+4(1+\varepsilon|%
\varepsilon|b_i^2 p_i(x))(1+\varepsilon|\varepsilon|b_i^2
p_i(y))\sin^2\left(\frac{x-y}{2}\right)\right)^{\frac{3}{2}}}$}} \\
& \ \ \ \ \times \bigg(2b_{i}\varepsilon |\varepsilon
|(p_{i}(x)-p_{i}(y))(h_{i}(x)-h_{i}(y)) \\
& \ \ \ \ \ \ \ \ +4(h_{i}(x)(1+\varepsilon |\varepsilon
|b_{i}^{2}p_{i}(y))+h_{i}(y)(1+\varepsilon |\varepsilon
|b_{i}^{2}p_{i}(x)))\sin ^{2}(\frac{x-y}{2})\bigg)dy,
\end{split}
\label{2-12}
\end{equation}%
\begin{equation}
\begin{split}
& \partial _{p_{i}}F_{i23}h_{i}={\resizebox{.86\hsize}{!}{$\frac{%
\varepsilon|\varepsilon|b_i^2h'_i(x)\gamma_i}{1+\varepsilon|%
\varepsilon|b_i^2p_i(x)}\int\!\!\!\!\!\!\!\!\!\; {}-{}
\frac{(p_i(x)-p_i(y))\cos(x-y)dy}{\left(
|\varepsilon|^4b_i^4\left(p_i(x)-p_i(y)\right)^2+4(1+\varepsilon|%
\varepsilon|b_i^2 p_i(x))(1+\varepsilon|\varepsilon|b_i^2
p_i(y))\sin^2\left(\frac{x-y}{2}\right)\right)^{\frac{1}{2}}}$}} \\
& {\resizebox{.97\hsize}{!}{$\quad
-\frac{|\varepsilon|^4b_i^4p'_i(x)h_i(x)\gamma_i}{(1+\varepsilon|%
\varepsilon|b_i^2p_i(x))^2}\int\!\!\!\!\!\!\!\!\!\; {}-{}
\frac{(p_i(x)-p_i(y))\cos(x-y)dy}{\left(
|\varepsilon|^4b_i^4\left(p_i(x)-p_i(y)\right)^2+4(1+\varepsilon|%
\varepsilon|b_i^2 p_i(x))(1+\varepsilon|\varepsilon|b_i^2
p_i(y))\sin^2\left(\frac{x-y}{2}\right)\right)^{\frac{1}{2}}}$}} \\
& {\resizebox{.97\hsize}{!}{$ \quad
+\frac{\varepsilon|\varepsilon|b_i^2p'_i(x)\gamma_i}{1+\varepsilon|%
\varepsilon|b_i^2p_i(x)}\int\!\!\!\!\!\!\!\!\!\; {}-{}
\frac{(h_i(x)-h_i(y))\cos(x-y)dy}{\left(
|\varepsilon|^4b_i^4\left(p_i(x)-p_i(y)\right)^2+4(1+\varepsilon|%
\varepsilon|b_i^2 p_i(x))(1+\varepsilon|\varepsilon|b_i^2
p_i(y))\sin^2\left(\frac{x-y}{2}\right)\right)^{\frac{1}{2}}}$}} \\
& {\resizebox{.97\hsize}{!}{$\quad
-\frac{|\varepsilon|^4b_i^4p'_i(x)\gamma_i}{1+\varepsilon|%
\varepsilon|b_i^2p_i(x)}\int\!\!\!\!\!\!\!\!\!\; {}-{}
\frac{(p_i(x)-p_i(y))\cos(x-y)}{\left(
|\varepsilon|^4b_i^4\left(p_i(x)-p_i(y)\right)^2+4(1+\varepsilon|%
\varepsilon|b_i^2 p_i(x))(1+\varepsilon|\varepsilon|b_i^2
p_i(y))\sin^2\left(\frac{x-y}{2}\right)\right)^{\frac{3}{2}}}$}} \\
& \quad \times \bigg(b_{i}^{2}\varepsilon |\varepsilon
|(p_{i}(x)-p_{i}(y))(h_{i}(x)-h_{i}(y)) \\
& \qquad \quad +2(h_{i}(x)(1+\varepsilon |\varepsilon
|b_{i}^{2}p_{i}(y))+h_{i}(y)(1+\varepsilon |\varepsilon
|b_{i}^{2}p_{i}(x)))\sin ^{2}(\frac{x-y}{2})\bigg)dy,
\end{split}
\label{2-13}
\end{equation}%
and
\begin{equation}
\begin{split}
& \partial _{p_{i}}F_{i24}h_{i}={\resizebox{.87\hsize}{!}{$\frac{%
\varepsilon|\varepsilon|b_i^2\gamma_i}{1+\varepsilon|%
\varepsilon|b_i^2p_i(x)}\int\!\!\!\!\!\!\!\!\!\; {}-{}
\frac{h'_i(y)p'_i(x)\sin(x-y)dy}{\left(
|\varepsilon|^4b_i^4\left(p_i(x)-p_i(y)\right)^2+4(1+\varepsilon|%
\varepsilon|b_i^2 p_i(x))(1+\varepsilon|\varepsilon|b_i^2
p_i(y))\sin^2\left(\frac{x-y}{2}\right)\right)^{\frac{1}{2}}}$}} \\
& {\resizebox{.97\hsize}{!}{$\quad+\frac{\varepsilon|\varepsilon|b_i^2%
\gamma_i}{1+\varepsilon|\varepsilon|b_i^2p_i(x)}\int\!\!\!\!\!\!\!\!\!\;
{}-{} \frac{h'_i(x)p'_i(y)\sin(x-y)dy}{\left(
|\varepsilon|^4b_i^4\left(p_i(x)-p_i(y)\right)^2+4(1+\varepsilon|%
\varepsilon|b_i^2 p_i(x))(1+\varepsilon|\varepsilon|b_i^2
p_i(y))\sin^2\left(\frac{x-y}{2}\right)\right)^{\frac{1}{2}}}$}} \\
& {\resizebox{.97\hsize}{!}{$\quad
-\frac{|\varepsilon|^4b_i^4\gamma_i}{1+\varepsilon|\varepsilon|b_i^2p_i(x)}%
\int\!\!\!\!\!\!\!\!\!\; {}-{} \frac{p'_i(x)p'_i(y)\sin(x-y)}{\left(
|\varepsilon|^4b_i^4\left(p_i(x)-p_i(y)\right)^2+4(1+\varepsilon|%
\varepsilon|b_i^2 p_i(x))(1+\varepsilon|\varepsilon|b_i^2
p_i(y))\sin^2\left(\frac{x-y}{2}\right)\right)^{\frac{3}{2}}}$}} \\
& \ \ \ \quad \times \bigg(b_{i}^{2}\varepsilon |\varepsilon
|(p_{i}(x)-p_{i}(y))(h_{i}(x)-h_{i}(y)) \\
& \qquad \quad +2(h_{i}(x)(1+\varepsilon |\varepsilon
|b_{i}^{2}p_{i}(y))+h_{i}(y)(1+\varepsilon |\varepsilon
|b_{i}^{2}p_{i}(x)))\sin ^{2}(\frac{x-y}{2})\bigg)dy.
\end{split}
\label{2-14}
\end{equation}%
By similar computations as before, we deduce that
\begin{equation}
\partial _{p_{i}}F_{i3}^{1}(\varepsilon ,p_{1},p_{2})h_{i}=|\varepsilon
|\gamma _{i}\partial _{p_{i}}\mathcal{R}_{i3}(\varepsilon
,p_{1},p_{2})h_{i},\quad \mbox{for}\,i=1,2.  \label{fh}
\end{equation}%
Notice that it is continuous, since $\mathcal{R}_{i3}(\varepsilon
,p_{1},p_{2},h_{1},h_{2})$ is a continuous function which is not singular at
$x=y$. Also, we have that
\begin{equation}
\partial _{p_{3-i}}F^{1}(\varepsilon ,p_{1},p_{2})h_{3-i}=|\varepsilon
|\gamma _{3-i}\partial _{p_{3-i}}\mathcal{R}_{3-i}(\varepsilon
,p_{1},p_{2})h_{3-i},\quad \mbox{for}\,i=1,2.  \label{f3h}
\end{equation}%
Similarly, we deduce that it is continuous, since $\mathcal{R}%
_{3-i}(\varepsilon ,p_{1},p_{2})$ is a continuous function. To this aim, the
first step is to show
\begin{equation}
\lim\limits_{t\rightarrow 0}\left\Vert \frac{F_{i2j}(\varepsilon ,\Omega
,\bar{x},p_{i}+th_{i},p_{2})-F_{i2j}(\varepsilon ,p_{1},p_{2})}{t}-\partial
_{p_{i}}F_{i2j}h_{i}\right\Vert _{Y^{k-1}}\rightarrow 0,  \label{gateaux}
\end{equation}%
for $j=1,2,3,4$. The only term that one should care about is the most
singular which is when $j=2$. To this end, one uses the notations given in
Lemma \ref{2-2}. Then, we have
\begin{equation*}
\begin{split}
& \frac{F_{i22}(\varepsilon ,p_{i}+th_{i})-F_{i22}(\varepsilon ,p_{1},p_{2})%
}{t}-\partial _{p_{i}}F_{i22}(\varepsilon ,p_{1},p_{2})h_{i} \\
& =\frac{\gamma _{i}}{t}\int \!\!\!\!\!\!\!\!\!\;{}-{}(p_{i}^{\prime
}(x)-p_{i}^{\prime }(y))\cos (x-y)\left( \frac{1}{D_{1}(p_{i}+th_{i})}-\frac{%
1}{D_{1}(p_{i})}\right) \\
& +\frac{\gamma _{i}}{t}\int \!\!\!\!\!\!\!\!\!\;{}-{}\frac{t\varepsilon
|\varepsilon |b_{i}^{2}}{2}\cdot \frac{\Delta p_{i}\Delta
h_{i}+2(h_{i}(1+\varepsilon |\varepsilon |b_{i}^{2}\tilde{p}_{i})+\tilde{h}%
_{i}(1+\varepsilon |\varepsilon |b_{i}^{2}p_{i}))\sin ^{2}(\frac{x-y}{2})}{%
D_{1}(p_{i})}\bigg)dy \\
& \ \ \ \ +\gamma _{i}\int \!\!\!\!\!\!\!\!\!\;{}-{}(h_{i}^{\prime
}(x)-h_{i}^{\prime }(y))\cos (x-y)\bigg(\frac{1}{D_{1}(p_{i}+th_{i})}-\frac{1%
}{D_{1}(p_{i})}\bigg)dy \\
& =G_{i21}+G_{i22}.
\end{split}%
\end{equation*}%
Taking $\partial ^{k-1}$ derivatives of $G_{i21}$ yields
\begin{equation*}
\begin{split}
\partial ^{k-1}G_{21}& ={\resizebox{.85\hsize}{!}{$\frac{\gamma_i}{t}\int\!%
\!\!\!\!\!\!\!\!\;
{}-{}\bigg(\frac{1}{D_1(p_i+th_i)}-\frac{1}{D_1(p_i)}+\frac{t\varepsilon|%
\varepsilon|b_i^2}{2}\frac{\Delta p_i\Delta
h_i+2(h_i(1+\varepsilon|\varepsilon|b_i^2\tilde p_i)+\tilde
h_i(1+\varepsilon|\varepsilon|b_i^2
p_i))\sin^2(\frac{x-y}{2})}{D_1(p_i)}\bigg)$}} \\
& \ \ \ \ \times (\partial ^{k}p_{i}(x)-\partial ^{k}p_{i}(y))\cos
(x-y)dy+l.o.t.
\end{split}%
\end{equation*}%
By applying the mean value theorem, we arrive at
\begin{equation*}
\begin{split}
\frac{1}{D_{1}(p_{i}+th_{i})}-& \frac{1}{D_{1}(p_{i})}+\frac{t\varepsilon
|\varepsilon |b_{i}^{2}}{2}\frac{\Delta p_{i}\Delta
h_{i}+2(h_{i}(1+\varepsilon |\varepsilon |b_{i}^{2}\tilde{p}_{i})+\tilde{h}%
_{i}(1+\varepsilon |\varepsilon |b_{i}^{2}p_{i}))\sin ^{2}(\frac{x-y}{2})}{%
D_{1}(p_{i})} \\
& \qquad \qquad \qquad \qquad \qquad \qquad \qquad \qquad \qquad \sim \frac{%
Ct^{2}}{|\sin (\frac{x-y}{2})|}\zeta (\varepsilon ,p_{1},p_{2},h_{1},h_{2}),
\end{split}%
\end{equation*}%
where $\Vert \zeta (\varepsilon ,p_{1},p_{2},h_{1},h_{2})\Vert _{L^{\infty
}}<\infty $. It follows that
\begin{equation*}
\Vert G_{i21}\Vert _{Y^{k-1}}\leq Ct\left\Vert \int \!\!\!\!\!\!\!\!\!\;{}-{}%
\frac{\partial ^{k}p_{i}(x)-\partial ^{k}p_{i}(y)}{|\sin (\frac{x-y}{2})|}%
dy\right\Vert _{L^{2}}\leq Ct\Vert p_{i}\Vert _{X^{k+\log }}.
\end{equation*}%
Similar computations as before gives the bound $\Vert G_{i22}\Vert
_{Y^{k-1}}\leq Ct\Vert p_{i}\Vert _{X^{k+\log }}$ by using \eqref{2-7}.
Then, by sending $t\rightarrow 0$, we arrive at \eqref{gateaux}. The
continuity of $\partial _{p_{i}}F_{i2}(\varepsilon ,p_{1},p_{2})h_{i}$
follows the same spirit as \eqref{2-7}. This completes the proof of Lemma %
\ref{lem2-3}.
\end{proof}

From \eqref{2-10}-\eqref{f3h}, by letting $\varepsilon =0$ and $p_{i}\equiv
0 $, for $i=1,2$, we obtain that
\begin{equation}
\partial _{p_{i}}F_{i}^{1}(0,\Omega ,\bar{x},0,0)h_{i}=\frac{\gamma _{i}}{2}%
\int \!\!\!\!\!\!\!\!\!\;{}-{}\frac{h_{i}(x-y)\sin (y)dy}{\left( 4\sin ^{2}(%
\frac{y}{2})\right) ^{\frac{1}{2}}}-\gamma _{i}\int \!\!\!\!\!\!\!\!\!\;{}-{}%
\frac{(h^{\prime }(x)-h_{i}^{\prime }(x-y))\cos (y)dy}{\left( 4\sin ^{2}(%
\frac{y}{2})\right) ^{\frac{1}{2}}}  \label{gat1}
\end{equation}%
and
\begin{equation}
\partial _{p_{3-i}}F_{i}^{1}(0,\Omega ,\bar{x},0,0)h_{3-i}=0.  \label{gat2}
\end{equation}

\begin{lemma}
\label{critical} For $\gamma _{1}\neq \gamma _{2}$, there exist two initial
point vortices $\gamma _{1}\pi \delta _{0}$ and $\gamma _{2}\pi \delta _{d}$
rotating uniformly around the centroid
\begin{equation}
\bar{x}^{\ast }:=\frac{d\gamma _{2}}{\gamma _{1}+\gamma _{2}},  \label{pos}
\end{equation}%
with angular speed
\begin{equation}
\Omega ^{\ast }:=\frac{\gamma _{1}+\gamma _{2}}{2d^{3}}.  \label{vel}
\end{equation}
\end{lemma}

\begin{proof}
It is well-known that the two points vortex system is a solution to the
following system of equations
\begin{equation*}
F_{i}(0,\Omega ,\bar{x},0,0)=0,\,\mbox{for}\,i=1,2.
\end{equation*}%
By using the computations obtained in Section 3, specifically \eqref{2-4}-%
\eqref{Fi3}, with $p_{i}=0$, for $i=1,2$ and $\varepsilon =0$, we can easily
verify that
\begin{equation*}
F_{i}(0,\Omega ,\bar{x},0,0)=-\Omega \left( -(-1)^{i}\bar{x}\sin (x)+(i-1)d\sin
(x)\right) +\frac{\gamma _{3-i}\sin (x)}{2d^{2}},\,\mbox{for}\,i=1,2.
\end{equation*}%
Therefore $F_{i}(0,\Omega ,\bar{x},0,0)=0$, if and only if, we have the
following two equalities
\begin{equation}
F_{1}(0,\Omega ,\bar{x},0,0)=-\Omega \bar{x}\sin (x)+\frac{\gamma _{2}\sin (x)}{%
2d^{2}}=0  \label{f1}
\end{equation}%
and
\begin{equation}
F_{2}(0,\Omega ,\bar{x},0,0)=-\Omega \left( -\bar{x}\sin (x)+d\sin (x)\right) +%
\frac{\gamma _{1}\sin (x)}{2d^{2}}=0.  \label{f2}
\end{equation}%
Adding the above equations yields
\begin{equation*}
\Omega =\Omega ^{\ast }:=\frac{\gamma _{1}+\gamma _{2}}{2d^{3}}.
\end{equation*}%
Now, by replacing $\Omega ^{\ast }$ into \eqref{f1}, we obtain
\begin{equation*}
\bar{x}=\bar{x}^{\ast }:=\frac{d\gamma _{2}}{\gamma _{1}+\gamma _{2}}.
\end{equation*}
\end{proof}

Now, we are in position to prove that $F^1_i(\varepsilon,\Omega,\bar{x},p_1,p_2)$
at the point $(0,\Omega^*,\bar{x}^*,0,0)$ is an isomorphism.

\begin{lemma}
\label{lem2-4} Let $h=(\beta _{1},\beta _{2},h_{1},h_{2})\in \mathbb{R}%
\times \mathbb{R}\times X^{k+\log }$, with $h_{i}(x)=\sum_{j=1}^{\infty
}a_{n}^{i}\sin (jx)$, $i=1,2$. Then, it holds
\begin{align*}
D_{g}F(0,g_{0})h(x)=& -\dfrac{\beta _{1}d}{\gamma _{1}+\gamma _{2}}%
\begin{pmatrix}
\gamma _{2} \\
\gamma _{1}%
\end{pmatrix}%
\sin (x)-\dfrac{\beta _{2}(\gamma _{1}+\gamma _{2})}{2d^{3}}%
\begin{pmatrix}
1 \\
-1%
\end{pmatrix}%
\sin (x) \\
& \qquad -\sum_{j=1}^{\infty }\,j\sigma _{j}%
\begin{pmatrix}
{\gamma _{1}}a_{j}^{1} \\
{\gamma _{2}}a_{j}^{2}%
\end{pmatrix}%
\sin (jx)
\end{align*}%
with
\begin{equation*}
g\triangleq (\Omega ,\bar{x},p_{1},p_{2})\quad \text{and}\quad g_{0}\triangleq
(\Omega ^{\ast },\bar{x}^{\ast },0,0),
\end{equation*}%
where $\sigma _{j}$ is given by
\begin{equation*}
\sigma _{j}=\frac{2}{\pi }\sum\limits_{i=1}^{j}\frac{1}{2i-1}.
\end{equation*}%
Moreover, the linearized operator $D_{g}F(0,g_{0}):\mathbb{R}\times \mathbb{%
\ R}\times X^{k+\log }\rightarrow Y^{k-1}$ is an isomorphism.
\end{lemma}

\begin{proof}
Since $F^{1}=(F_{1}^{1},F_{2}^{1})$, choose $h=(h_{1},h_{2})\in X^{k+\log }$
to obtain
\begin{equation*}
D_{p}F^{1}(0,\Omega ,\bar{x},0,0)h(x)=%
\begin{pmatrix}
\partial _{p_{1}}F_{1}^{1}(0,\Omega ,\bar{x},0,0)h_{1}(x)+\partial
_{p_{2}}F_{1}^{1}(0,\Omega ,\bar{x},0,0)h_{2}(x) \\
\partial _{p_{1}}F_{2}^{1}(0,\Omega ,\bar{x},0,0)h_{1}(x)+\partial
_{p_{2}}F_{2}^{1}(0,\Omega ,\bar{x},0,0)h_{2}(x)%
\end{pmatrix}%
,
\end{equation*}%
with $p=(p_{1},p_{2})$. By using \eqref{gat1}, the Gateaux derivative of the
functional $F_{i}^{1}$ at the point $(0,\Omega ,\bar{x},0,0)$ in the direction
$h_{i}$ has the following form
\begin{equation*}
\partial _{p_{i}}F_{i}^{1}(0,\Omega ,\bar{x},0,0)h_{i}(x)=\frac{\gamma _{i}}{2}%
\int \!\!\!\!\!\!\!\!\!\;{}-{}\frac{h_{i}(x-y)\sin (y)dy}{\left( 4\sin ^{2}(%
\frac{y}{2})\right) ^{\frac{1}{2}}}-\gamma _{i}\int \!\!\!\!\!\!\!\!\!\;{}-{}%
\frac{(h_{i}^{\prime }(x)-h_{i}^{\prime }(x-y))\cos (y)dy}{\left( 4\sin ^{2}(%
\frac{y}{2})\right) ^{\frac{1}{2}}}.
\end{equation*}%
On the other hand, the Gateaux derivative of the functional $F_{i}^{1}$ at
the point $(0,\Omega ,\bar{x},0,0)$ with direction $h_{3-i}$ was obtained in %
\eqref{gat1} which is given by
\begin{equation*}
\partial _{p_{3-i}}F_{i}^{1}(0,\Omega ,\bar{x},0,0)h_{3-i}(x)=0.
\end{equation*}%
Hence, the differential of the nonlinear functional $F^{1}$ around $%
(0,\Omega ,\bar{x},0,0)$ in the direction $h=(h_{1},h_{2})\in X^{k+\log }$ has
the following form
\begin{equation*}
D_{p}F^{1}(0,\Omega ,\bar{x},0,0)h(x)=%
\begin{pmatrix}
-{\gamma _{1}}\sum\limits_{j=2}^{\infty }a_{j}^{1}\sigma _{j}j\sin (jx), \\
-{\gamma _{2}}\sum\limits_{j=2}^{\infty }a_{j}^{2}\sigma _{j}j\sin (jx),%
\end{pmatrix}%
,
\end{equation*}%
where $p=(p_{1},p_{2})$. According to \cite[Proposition 2.4]{Cas1}, $\sigma
_{j}$ can take different expressions depending on the value of $\alpha $.
For $\alpha \neq 1$, we have that
\begin{equation*}
\sigma _{j}=2^{\alpha }\frac{2\pi \Gamma (1-\alpha )}{\Gamma (\frac{\alpha }{%
2})\Gamma (1-\frac{\alpha }{2})}\left( \frac{\Gamma (\frac{\alpha }{2})}{%
\Gamma (1-\frac{\alpha }{2})}-\frac{\Gamma (j+\frac{\alpha }{2})}{\Gamma
(j+1-\frac{\alpha }{2})}\right) ,
\end{equation*}%
and for $\alpha =1,$%
\begin{equation*}
\sigma _{j}=\frac{2}{\pi }\sum\limits_{i=1}^{j}\frac{1}{2i-1}.
\end{equation*}%
An important fact here is that $\{\sigma _{j}\}$ is increasing with respect
to $j$. Moreover, the asymptotic behavior of $\{\sigma _{j}\},$ when $j$ is
large, is $\sigma _{j}\sim \ln (j)$ if $\alpha =1,$ and $\sigma _{j}\sim
j^{\alpha -1}$ if $1<\alpha <2$.\newline
\newline
From Lemma \ref{lem2-2}, one can easily check that
\begin{align}
F_{i}^{1}(0,\Omega ,\bar{x},p_{1},p_{2})& =\frac{\gamma _{i}}{4}\int
\!\!\!\!\!\!\!\!\!\;{}-{}\frac{p_{i}(x-y)\sin (y)dy}{|\sin (\frac{y}{2})|}-%
\frac{\gamma _{i}}{2}\int \!\!\!\!\!\!\!\!\!\;{}-{}\frac{(p_{i}^{\prime
}(x)-p_{i}^{\prime }(x-y))\cos (y)dy}{|\sin (\frac{y}{2})|}  \notag \\
& \qquad +\Omega \left( (-1)^{i}\bar{x}\sin (x)-(i-1)d\sin (x)\right) -\frac{%
\gamma _{3-i}\sin (x)}{2d^{2}}.  \label{gatf2}
\end{align}%
Differentiating \eqref{gatf2} with respect to the angular speed $\Omega $
around the point $(\Omega ^{\ast },\bar{x}^{\ast })$ yields
\begin{equation*}
\partial _{\Omega }F_{i}^{1}(0,\Omega ^{\ast },\bar{x}^{\ast
},p_{1},p_{2})=(-1)^{i}\bar{x}^{\ast }\sin (x)-(i-1)d\sin (x).
\end{equation*}%
Next, differentiating \eqref{gatf} with respect to $\bar{x}$ at the point $%
(\Omega ^{\ast },\bar{x}^{\ast }),$ we get
\begin{equation*}
\partial _{\bar{x}}F_{i}^{1}(0,\Omega ^{\ast },\bar{x}^{\ast
},p_{1},p_{2})=(-1)^{i}\Omega ^{\ast }\sin (x).
\end{equation*}%
Therefore, for all $(\beta _{1},\beta _{2})\in \mathbb{R}^{2}$
\begin{equation}
D_{(\Omega ,\bar{x})}F^{1}(0,\Omega ^{\ast },\bar{x}^{\ast },p_{1},p_{2})(\beta
_{1},\beta _{2})=-\beta _{1}%
\begin{pmatrix}
\dfrac{\gamma _{2}d}{\gamma _{1}+\gamma _{2}} \\
\dfrac{\gamma _{1}d}{\gamma _{1}+\gamma _{2}}%
\end{pmatrix}%
\sin (x)-\beta _{2}%
\begin{pmatrix}
\dfrac{\gamma _{1}+\gamma _{2}}{2d^{3}} \\
-\dfrac{\gamma _{1}+\gamma _{2}}{2d^{3}}%
\end{pmatrix}%
\sin (x).  \label{deriv22}
\end{equation}

Now, we obtain the explicit expression of the inverse of $D_{(\Omega
,\bar{x},p_{1},p_{2})}F^{1}(0,\Omega ^{\ast },\bar{x}^{\ast },p_{1},p_{2})$,
namely $D_{g}F^{1}(0,g_{0})^{-1}$. For all $h=(\beta _{1},\beta
_{2},h_{1},h_{2})\in \mathbb{R}\times \mathbb{R}\times X^{k+\log }$ with
\begin{equation*}
h_{i}(x)=\sum_{j=1}^{\infty }a_{j}^{i}\cos (jx),\quad i=1,2,
\end{equation*}%
one gets
\begin{align*}
D_{g}F^{1}(0,g_{0})h(x)=& -\dfrac{\beta _{1}d}{\gamma _{1}+\gamma _{2}}%
\begin{pmatrix}
\gamma _{2} \\
\gamma _{1}%
\end{pmatrix}%
\sin (x)-\beta _{2}\dfrac{\gamma _{1}+\gamma _{2}}{2d^{3}}%
\begin{pmatrix}
1 \\
-1%
\end{pmatrix}%
\sin (x) \\
& \qquad -\sum_{j=2}^{\infty }\,j\sigma _{j}%
\begin{pmatrix}
{\gamma _{1}}a_{j}^{1} \\
{\gamma _{2}}a_{j}^{2}%
\end{pmatrix}%
\sin (jx).
\end{align*}%
Let $k\in Y^{k-1}$ satisfy the following expansion
\begin{equation*}
k(x)=\sum_{j=1}^{\infty }%
\begin{pmatrix}
A_{j} \\
B_{j}%
\end{pmatrix}%
\sin (jx).
\end{equation*}%
Notice that we can easily obtain the inverse of the linearized operator $%
D_{g}F^{1}(0,g_{0})k(x)$ which has the following expression
\begin{align}
D_{g}& F^{1}(0,g_{0})^{-1}k(x)=  \notag  \label{g-1} \\
& -\bigg(\dfrac{A_{1}+B_{1}}{d}\cos (x),\dfrac{2d^{3}(A_{1}\gamma
_{1}-B_{1}\gamma _{2})}{(\gamma _{1}+\gamma _{2})^{2}}\cos (x),\displaystyle%
\sum_{j=2}^{\infty }\dfrac{A_{j}}{j\gamma _{1}}\cos (jx),\displaystyle%
\sum_{j=2}^{\infty }\dfrac{B_{j}}{j\gamma _{2}}\cos (jx)\bigg).
\end{align}%
Now we are going to prove $D_{(\Omega ,\bar{x},p_{1},p_{2})}F^{1}(0,\Omega
^{\ast },\bar{x}^{\ast },p_{1},p_{2})$ is an isomorphism from $\mathbb{R}%
\times \mathbb{\ R}\times X^{k+\log }$ to $Y^{k-1}$. From Lemma \ref{2-2}
and \eqref{deriv22}, it is obvious that $D_{(\Omega
,\bar{x},p_{1},p_{2})}F^{1}(0,\Omega ^{\ast },\bar{x}^{\ast },p_{1},p_{2})$ maps
$\mathbb{R}\times \mathbb{\ R}\times X^{k+\log }$ into $Y^{k-1}$ . Hence
only the invertibility needs to be considered. In fact, the restricted
linear operator $D_{(\Omega ,\bar{x},p_{1},p_{2})}F^1(0,\Omega ^{\ast
},\bar{x}^{\ast },p_{1},p_{2})$ is invertible if and only if the determinant
of the two vectors obtained in \eqref{deriv22} is non-vanishing. In fact, the
determinant is equal to $-\displaystyle\frac{1}{2d^{2}}({\gamma _{1}+\gamma
_{2}}),$ which is different from zero for all $\gamma _{1}+\gamma _{2}\neq
0. $
\end{proof}

Now, we see how to obtain the critical angular velocity $\Omega ^{\ast }$
and the critical centroid $\bar{x}^{\ast }$.

\begin{lemma}
\label{lem2-5} There exist
\begin{equation*}
\Omega (\varepsilon ,p_{1},p_{2}):=\Omega ^{\ast }+\varepsilon \mathcal{R}%
_{\Omega }(\varepsilon ,p_{1},p_{2})
\end{equation*}%
and
\begin{equation*}
\bar{x}(\varepsilon ,p_{1},p_{2}):=\bar{x}^{\ast }+\varepsilon \mathcal{R}%
_{\Omega }(\varepsilon ,p_{1},p_{2}),
\end{equation*}%
with $\Omega ^{\ast }$ and $\bar{x}^{\ast }$ given in \eqref{pos} and %
\eqref{vel}, respectively, and a continuous function $\mathcal{R}_{\Omega
}(\varepsilon ,p_{1},p_{2}):X^{k+\log }\rightarrow \mathbb{R}$ such that $%
\tilde{F}^{1}(\varepsilon ,p_{1},p_{2}):(-\frac{1}{2},\frac{1}{2})\times
\mathbb{R}\times \mathbb{R}\times B_{r}\rightarrow Y^{k-1}$ is given by
\begin{equation*}
\tilde{F}^{1}(\varepsilon,p_{1},p_{2}):=F^{1}(\varepsilon
,\Omega (\varepsilon ,p_{1},p_{2}),\bar{x},p_{1},p_{2}).
\end{equation*}%
Moreover, $D_{p}\mathcal{R}_{\Omega }(\varepsilon ,p_{1},p_{2})\bar{h}%
:X^{k+\log }\rightarrow \mathbb{R}$ is continuous, where $\bar{h}%
=(h_{1},h_{2})$ and $p=(p_{1},p_{2})$.
\end{lemma}

\begin{proof}
It is enough to find $\Omega (\varepsilon ,p_{1},p_{2}):(-\frac{1}{2},\frac{1%
}{2})\times \mathbb{R}\times \mathbb{R}\times B_{r}\rightarrow \mathbb{R}$
and $\bar{x}(\varepsilon ,p_{1},p_{2})$ for the space where the first Fourier
coefficient vanishes in the functional $F^{1}(\varepsilon ,\Omega
(\varepsilon ,p_{1},p_{2}),\bar{x},p_{1},p_{2})$. From \eqref{2-8} and Lemma %
\ref{lem2-4}, we can deduce that $\varepsilon |\varepsilon |\tilde{\mathcal{R%
}}_{2}(\varepsilon ,p_{1},p_{2})$ is the contribution of $F_{i2}$ to the
first Fourier coefficient, with $\tilde{\mathcal{R}}_{2}$ the contribution
of $\mathcal{R}_{2}$ in \eqref{2-8}. Hence, we can deduce that $\Omega
(\varepsilon ,p_{1},p_{2})$ satisfy the following equation, for $i=1,2,$
\begin{equation}
-\Omega \left( (-1)^{i}\bar{x}+(i-1)d+\varepsilon |\varepsilon |\tilde{%
\mathcal{R}}_{1}(\varepsilon ,p_{1},p_{2})\right) +\varepsilon |\varepsilon |%
\tilde{\mathcal{R}}_{2}(\varepsilon ,p_{1},p_{2})+\frac{\gamma _{3-i}}{2d^{2}%
}+\varepsilon \tilde{\mathcal{R}}_{3}(\varepsilon ,p_{1},p_{2})=0,
\label{coef}
\end{equation}%
where $\tilde{\mathcal{R}}_{i}(\varepsilon ,p_{1},p_{2})$ $(i=1,2,3)$ is the
contribution of $\mathcal{R}_{i}$ to the first Fourier coefficient by %
\eqref{2-4} and \eqref{Fi3}. Therefore, \eqref{coef} is satisfied if and
only if, the following two equations holds
\begin{equation}
-\Omega \left( \bar{x}+\varepsilon |\varepsilon |\tilde{\mathcal{R}}%
_{1}(\varepsilon ,p_{1},p_{2})\right) +\varepsilon |\varepsilon |\tilde{%
\mathcal{R}}_{2}(\varepsilon ,p_{1},p_{2})+\frac{\gamma _{2}}{2d^{2}}%
+\varepsilon \tilde{\mathcal{R}}_{3}(\varepsilon ,p_{1},p_{2})=0
\label{vsing}
\end{equation}%
and
\begin{equation}
-\Omega \left( -\bar{x}+d+\varepsilon |\varepsilon |\tilde{\mathcal{R}}%
_{1}(\varepsilon ,p_{1},p_{2})\right) +\varepsilon |\varepsilon |\tilde{%
\mathcal{R}}_{2}(\varepsilon ,p_{1},p_{2})+\frac{\gamma _{1}}{2d^{2}}%
+\varepsilon \tilde{\mathcal{R}}_{3}(\varepsilon ,p_{1},p_{2})=0.
\label{xsing}
\end{equation}%
From this, we can add the above two equations in order to obtain
\begin{equation*}
\Omega (\varepsilon ,p_{1},p_{2}):=\Omega ^{\ast }+\varepsilon \mathcal{R}%
_{\Omega }(\varepsilon ,p_{1},p_{2}),
\end{equation*}%
where $\mathcal{R}_{\Omega }(\varepsilon ,p_{1},p_{2}):X^{k+\log
}\rightarrow \mathbb{R}$ is a continuous function.\newline
\newline
Similarly, substituting $\Omega ^{\ast }$ into \eqref{vsing} we arrive at
\begin{equation*}
\bar{x}(\varepsilon ,p_{1},p_{2}):=\bar{x}^{\ast }+\varepsilon \mathcal{R}%
_{\Omega }(\varepsilon ,p_{1},p_{2}),
\end{equation*}%
where $\mathcal{R}_{\Omega }(\varepsilon ,p_{1},p_{2}):X^{k+\log
}\rightarrow \mathbb{R}$ is also a continuous function. From Lemma \ref%
{lem2-3}, we have that $D_{p}{\mathcal{R}}_{i}(\varepsilon ,p_{1},p_{2})$ $%
(i=1,2,3)$ are continuous. Thus, we can conclude that $D_{p}\mathcal{R}%
_{\Omega }(\varepsilon ,p_{1},p_{2})\bar{h}:X^{k+\log }\rightarrow \mathbb{R}
$ is also continuous.
\end{proof}

Now, we provide a detailed and complete version of the first part of Theorem %
\ref{thm:informal}, which is concerned to the co-rotating case with unequal
circulations.

\begin{theorem}[SQG equations]
\label{main} The following assertions hold true.

\begin{enumerate}
\item There exist $\varepsilon _{0}>0$ and a unique $C^{1}$ function $%
g=(\Omega ,\bar{x},p_{1},p_{2}):[-\varepsilon _{0},\varepsilon
_{0}]\longrightarrow \mathbb{R}\times \mathbb{R}\times B_{r}$ such that
\begin{equation*}
F^{1}\Big(\varepsilon ,\Omega (\varepsilon ),\bar{x}(\varepsilon
),p_{1}(\varepsilon ),p_{2}(\varepsilon )\Big)=0.
\end{equation*}%
Moreover,
\begin{equation*}
\Big(\Omega (0),\bar{x}(0),p_{1}(0),p_{2}(0)\Big)=\big(\Omega ^{\ast },\bar{x}^{\ast
},0,0\big).
\end{equation*}%
In other words, the solution passes through the origin.

\item For each $\varepsilon\in[-\varepsilon_0,\varepsilon_0]\backslash \{0\}$%
, it always holds
\begin{equation*}
\big(p_1(\varepsilon),p_2(\varepsilon)\big)\neq (0,0).
\end{equation*}

\item If $(\varepsilon, p_1,p_2)$ is a solution to $F^1(\varepsilon,g)=0 $,
then $(-\varepsilon, \tilde{p}_1,\tilde{p}_2)$ is also a solution, where
\begin{equation*}
\tilde{p}_i(x) = p_i(-x), \quad i=1,2.
\end{equation*}

\item For all $\varepsilon \in [-\varepsilon_0, \varepsilon]\backslash\{0\}$%
, the set of solutions $R_i(x)$ parameterizes convex patches at least of
class $C^1$.
\end{enumerate}
\end{theorem}

\begin{proof}
\textbf{(1)} According to Lemmas \ref{lem2-2} and \ref{lem2-3}, we can
guarantee that $F:(-\frac{1}{2},\frac{1}{2})\times \mathbb{R}\times \mathbb{R%
}\times B_{r}\rightarrow {Y^{k-1}}$ is a $C^{1}$ function. Also, an
application of Lemma \ref{lem2-4} gives that $D_{g}F^{1}\big(0,g_{0}\big)%
:X^{k+\log }\rightarrow {Y^{k-1}}$ is in fact an isomorphism. Now, we can
apply the implicit function theorem and obtain that there exist $\varepsilon
_{0}>0$ and a unique $C^{1}$ parametrization given by $g=(\Omega
,\bar{x},p_{1},p_{2}):[-\varepsilon _{0},\varepsilon _{0}]\rightarrow B_{r}$
such that the set of solutions
\begin{equation*}
F^{1}\big(\varepsilon ,\Omega (\varepsilon ),\bar{x}(\varepsilon
),p_{1}(\varepsilon ),p_{2}(\varepsilon )\big)=0
\end{equation*}%
is not empty. Furthermore, we can set $\varepsilon =0$ such that
\begin{equation*}
\big(\Omega ,\bar{x},p_{1},p_{2}\big)(0)=(\Omega ^{\ast },\bar{x}^{\ast },0,0).
\end{equation*}

\textbf{(2)} Now, we want to show that for each small enough $\varepsilon
\in \lbrack -\varepsilon _{0},\varepsilon _{0}]\setminus \{0\}$ and any
angular velocity $\Omega $, we never get an asymmetric co-rotating vortex
patch pair with values $p_{1}=0$ and $p_{2}=0$. To this end, one needs to
prove that
\begin{equation*}
F^{1}(\varepsilon ,\Omega ,\bar{x},0,0)=(F^1_{1}(\varepsilon ,\Omega ,\bar{x},0,0),F^1_{2}(\varepsilon ,\Omega ,\bar{x},0,0))\neq 0.
\end{equation*}%
Notice that the functional $F_{i}^{1}$ can be split as
\begin{equation*}
F_{i}^{1}(\varepsilon ,\Omega ,\bar{x},0,0)=F_{i1}+F_{i2}+F_{i3},
\end{equation*}%
where $i=1,2.$ For this purpose, once again we resort to the Taylor formula %
\eqref{2-5} in \eqref{fi3eps} in order to obtain a suitable expansion of $%
\mathcal{R}_{3}(\varepsilon ,p_{1},p_{2})$. More precisely,
\begin{equation*}
\begin{split}
\mathcal{R}_{3}& (\varepsilon ,p_{1},p_{2})=\frac{3}{4}\frac{\gamma _{3-i}}{%
b_{3-i}}\int \!\!\!\!\!\!\!\!\!\;{}-{}\int_{0}^{1}\int_{0}^{1}\frac{%
(\varepsilon tB)B\sin (x-y)dsdtdy}{\left( A+\varepsilon stB\right) ^{\frac{5%
}{2}}}+o(\varepsilon ) \\
& \,{\resizebox{.95\hsize}{!}{$ = \frac{3}{4}\frac{\gamma_{3-i}\varepsilon}{
b_{3-i}} \int\!\!\!\!\!\!\!\!\!\; {}-{} \frac{ B^2\sin(x-y)dy}{
A^{\frac{5}{2}}}-\frac{15}{8}\frac{\gamma_{3-i}}{ b_{3-i}}
\int\!\!\!\!\!\!\!\!\!\; {}-{}\int_0^1 \int_0^1 \frac{(\varepsilon st
B)(\varepsilon tB)B\sin(x-y)drdsdtdy}{\left( A+\varepsilon rstB
\right)^{\frac{5}{2}}} +o(\varepsilon)$}} \\
& =\frac{3}{4}\frac{\gamma _{3-i}\varepsilon }{b_{3-i}}\int
\!\!\!\!\!\!\!\!\!\;{}-{}\frac{4d^{2}(b_{3-i}\cos (y)+b_{i}\cos (x))^{2}\sin
(x-y)dy}{d^{5}}+o(\varepsilon ) \\
& =\frac{3\varepsilon }{d^{3}}\frac{\gamma _{3-i}}{b_{3-i}}\int
\!\!\!\!\!\!\!\!\!\;{}-{}b_{3-i}^{2}\cos ^{2}(y)\sin (x-y)dy+\frac{%
3\varepsilon }{d^{3}}\frac{\gamma _{3-i}}{b_{3-i}}\int
\!\!\!\!\!\!\!\!\!\;{}-{}b_{i}^{2}\cos ^{2}(x)\sin (x-y)dy \\
& \qquad +\frac{3\varepsilon }{d^{3}}\frac{\gamma _{3-i}}{b_{3-i}}\int
\!\!\!\!\!\!\!\!\!\;{}-{}2b_{i}b_{3-i}\cos (x)\cos (y)\sin
(x-y)dy+o(\varepsilon ) \\
& =\frac{6\varepsilon }{d^{3}}\gamma _{3-i}b_{i}\cos (x)\int
\!\!\!\!\!\!\!\!\!\;{}-{}\cos (y)\sin (x-y)dy+o(\varepsilon ) \\
& =\frac{3\varepsilon }{d^{3}}\gamma _{3-i}b_{i}\sin (x)\cos
(x)+o(\varepsilon ). \\
&
\end{split}%
\end{equation*}%
From the above, we can conclude that
\begin{equation*}
\mathcal{R}_{3}(\varepsilon ,p_{1},p_{2})=\varepsilon \bar{C}\sin
(2x)+o(\varepsilon ),
\end{equation*}%
where $\bar{C}$ is a positive constant that depends only on $\gamma _{3-i}$,
$\alpha $, $d$ and $b_{i}$. By using \eqref{2-4} and \eqref{2-8}, we have
\begin{equation*}
F_{i1}(\varepsilon ,\Omega ,\bar{x},0,0)=-\Omega \left(
(-1)^{i}\bar{x}+(i-1)d\right) \sin (x)\ \text{\ \ \ and}\ \ \
F_{i2}(\varepsilon ,\Omega,\bar{x},0,0)=0.
\end{equation*}%
Thus, we can conclude that
\begin{equation*}
\begin{split}
  F^1_{i}(\varepsilon ,\Omega ,\bar{x},0,0)&=(F_{i1}+F_{i2}+F_{i3})(\varepsilon ,\Omega ,\bar{x},0,0)\\
   &=-\Omega \left(
(-1)^{i}\bar{x}+(i-1)d\right) \sin (x)+\frac{\gamma _{3-i}\sin (x)}{2d^{2}}+\varepsilon \bar{C}\sin
(2x)
\end{split}
\end{equation*}%
Thus,
\[F^1_{i}(\varepsilon ,\Omega ,\bar{x},0,0)\neq 0,\ \ \ \ \ \forall \,\varepsilon
\in \lbrack -\varepsilon _{0},\varepsilon _{0}]\setminus \{0\},\]
when $\varepsilon _{0}$ is chosen small enough. Consequently, $F_{i}$ is
different from zero if $\varepsilon \neq 0$ is small enough.\newline

\textbf{(3)} It is enough to verify that
\begin{equation*}
F_{i}^{1}(\varepsilon ,\Omega
,\bar{x},p_{1},p_{2})(-x)=-F_{i}^{1}(-\varepsilon ,\Omega ,\bar{x},\tilde{p}_{1},%
\tilde{p}_{2})(x)\text{,}\quad \text{\textnormal{for} }i=1,2,
\end{equation*}%
where $F_{i}$ is a sum between \eqref{2-4}, \eqref{2-8} and \eqref{Fi3}. Let
$\tilde{p}_{i}(x)=p_{i}(-x)$. We need to verify that if $(\varepsilon
,p_{1},p_{2})$ is a solution to $F_{i}^{1}(\varepsilon ,p_{1},p_{2})=0$,
then $(-\varepsilon ,\tilde{p_{1}},\tilde{p_{2}})$ is also a solution of the
same functional. Considering the change $y$ to $-y$ in the nonlinear
functional $F_{i}^{1}(\varepsilon ,\Omega ,\bar{x},p_{1},p_{2})$ and using
that $p_{i}$ are even functions, we deduce that $\Omega (\varepsilon
,p_{1},p_{2})=\Omega (-\varepsilon ,\tilde{p_{1}},\tilde{p_{2}})$. Inserting
it into $F_{i}^{1}(\varepsilon ,\Omega ,\bar{x},p_{1},p_{2}),$ we conclude
that $F_{i}^{1}(-\varepsilon ,\Omega ,\bar{x},\tilde{p_{1}},\tilde{p_{2}})=0$.%
\newline

\textbf{(4)} Now we prove that the set of solutions $R_{i}$ parameterizes
convex patches. To this end, we only need to compute the signed curvature of
the interface of the patch $z_{i}(x)=(z_{i}^{1}(x),z_{i}^{2}(x))=((1+%
\varepsilon ^{2}b_{i}^{2}p_{i}(x))\cos (x),(1+\varepsilon
^{2}b_{i}^{2}p_{i}(x))\sin (x))$ at the point $x$. Indeed, we have that
\begin{equation*}
\partial _{x}z_{i}^{1}(x)=\varepsilon ^{2}b_{i}^{2}p_{i}^{\prime }(x)\cos
(x)-(1+\varepsilon ^{2}b_{i}^{2}p_{i}(x))\sin (x),
\end{equation*}%
\begin{equation*}
\partial _{xx}z_{i}^{1}(x)=\varepsilon ^{2}b_{i}^{2}p_{i}^{\prime \prime
}(x)\cos (x)-\varepsilon ^{2}b_{i}^{2}p_{i}^{\prime }(x)\sin
(x)-(1+\varepsilon ^{2}b_{i}^{2}p_{i}(x))\cos (x)-\varepsilon
^{2}b_{i}^{2}p_{i}^{\prime }(x)\sin (x),
\end{equation*}%
and
\begin{equation*}
\partial _{x}z_{i}^{2}(x)=\varepsilon ^{2}b_{i}^{2}p_{i}^{\prime }(x)\sin
(x)+(1+\varepsilon ^{2}b_{i}^{2}p_{i}(x))\cos (x),
\end{equation*}%
\begin{equation*}
\partial _{xx}z_{i}^{2}(x)=\varepsilon ^{2}b_{i}^{2}p_{i}^{\prime \prime
}(x)\sin (x)+\varepsilon ^{2}b_{i}^{2}p_{i}^{\prime }(x)\cos
(x)-(1+\varepsilon ^{2}b_{i}^{2}p_{i}(x))\sin (x)+\varepsilon
^{2}b_{i}^{2}p_{i}^{\prime }(x)\cos (x).
\end{equation*}%
Then, inserting the above computations into the formula of the signed
curvature $\kappa _{i}(x)$ yields
\begin{equation*}
\begin{split}
\kappa _{i}(x)& =\frac{\partial _{xx}z_{i}^{2}(x)\partial
_{x}z_{i}^{1}(x)-\partial _{xx}z_{i}^{1}(x)\partial _{x}z_{i}^{2}(x)}{%
(\partial _{x}z_{i}^{1}(x)+\partial _{x}z_{i}^{2}(x))^{3/2}} \\
& =\frac{(1+\varepsilon ^{2}b_{i}^{2}p_{i}(x))^{2}+2\varepsilon
^{4}b_{i}^{4}(f_{i}^{\prime }(x))^{2}-\varepsilon ^{2}b_{i}^{2}f_{i}^{\prime
\prime }(x)(1+\varepsilon ^{2}b_{i}^{2}p_{i}(x))}{\left( (1+\varepsilon
^{2}b_{i}^{2}p_{i}(x))^{2}+\varepsilon ^{4}b_{i}^{4}(f_{i}^{\prime
}(x))^{2}\right) ^{\frac{3}{2}}}=\frac{1+O(\varepsilon )}{1+O(\varepsilon )}%
>0,
\end{split}%
\end{equation*}%
for $\varepsilon $ small and each $x\in \lbrack 0,2\pi )$. The quantity
obtained is non-negative if $\varepsilon \in (-1/2,1/2)$. Then, the signed
curvature is strictly positive and we obtain the desired result. Hence the
proof of Theorem \ref{main} is completed.
\end{proof}

\subsection{gSQG equations ($1<\protect\alpha<2$)}

Let $B_{r}$ be an open neighborhood of zero in the space $X^{k+\alpha -1}$,
i.e.,
\begin{equation*}
B_{r}:=\left\{ f\in X^{k+\alpha -1}:\ \Vert f\Vert _{X^{k+\alpha
-1}}<r\right\}
\end{equation*}%
with $0<r<1$. The next lemma gives us the continuity of $F_{i}$. This result
is equivalent to Lemma \ref{lem2-2} but in this case we are working on the
range $1<\alpha <2$.

\begin{lemma}
\label{lem2-7} The functional $F_i^\alpha(\varepsilon,\Omega_\alpha,\bar{x}_\alpha,p_1,p_2):
\left(-\frac{1}{2}, \frac{1}{2}\right)\times \mathbb{R} \times\mathbb{R}%
\times B_r \rightarrow Y^{k-1}$ is continuous.
\end{lemma}

\begin{proof}
Similarly to the case $\alpha =1$, it is straightforward to see that $%
F_{i1}:\left( -\frac{1}{2},\frac{1}{2}\right) \times \mathbb{R}\times
\mathbb{R}\times B_{r}\rightarrow Y^{k-1}$ is continuous and
\begin{equation}
F_{i1}=-\Omega_\alpha \left( (-1)^{i}\bar{x}_\alpha\sin (x)+(i-1)d\sin (x)+\varepsilon
|\varepsilon |^{\alpha }\mathcal{R}_{1}(\varepsilon ,p_{1},p_{2})\right) ,
\label{2-20}
\end{equation}%
where $\mathcal{R}_{1}(\varepsilon ,p_{1},p_{2}):\left( -\frac{1}{2},\frac{1%
}{2}\right) \times B_{r}\rightarrow Y^{k-1}$ is continuous. Since $p_{i}\in
B_{r}$, by a simple change of variables $y$ to $-y$, it is easy to see that $%
F_{i2}(\varepsilon ,p_{1},p_{2})$ is odd. By the application of the Taylor
formula \eqref{2-5} in $F_{i21}$, we can treat with the possible
singularities at $\varepsilon =0$. As was noticed in the proof of Lemma \ref%
{lem2-2}, the singular term is $F_{i22}$. Hence, we need to compute the $%
\partial ^{k-1}$ derivatives of $F_{i22}:$
\begin{equation*}
\begin{split}
& {\resizebox{.98\hsize}{!}{$
\partial^{k-1}F_{i22}=C_\alpha\gamma_i\int\!\!\!\!\!\!\!\!\!\;
{}-{}\frac{(\partial^kp_i(y)-\partial^kp_i(x))\cos(x-y)dy}{\left(
|\varepsilon|^{2+2\alpha}b_i^{2+2\alpha}\left(p_i(x)-p_i(y)\right)^2+4(1+%
\varepsilon|\varepsilon|^\alpha b^{1+\alpha}_i
p_i(x))(1+\varepsilon|\varepsilon|^\alpha b^{1+\alpha}_i
p_i(y))\sin^2\left(\frac{x-y}{2}\right)\right)^{\frac{\alpha}{2}}}$}} \\
& \ \ -{\resizebox{.94\hsize}{!}{$\alpha
C_\alpha\gamma_i\varepsilon|\varepsilon|^\alpha
b_i^{1+\alpha}\int\!\!\!\!\!\!\!\!\!\; {}-{}\frac{\cos(x-y)}{\left(
|\varepsilon|^{2+2\alpha}b_i^{2+2\alpha}\left(p_i(x)-p_i(y)\right)^2+4(1+%
\varepsilon|\varepsilon|^\alpha b^{1+\alpha}_i
p_i(x))(1+\varepsilon|\varepsilon|^\alpha b^{1+\alpha}_i
p_i(y))\sin^2\left(\frac{x-y}{2}\right)\right)^{1+\frac{\alpha}{2}}}$}} \\
& \ \ \ {\resizebox{.98\hsize}{!}{$\times
\left(\varepsilon|\varepsilon|^\alpha
b^{1+\alpha}_i(p_i(x)-p_i(y))(p'_i(x)-p'_i(y))+2((1+\varepsilon|%
\varepsilon|^\alpha
b_i^{1+\alpha}p_i(x))p'_i(y)+(1+\varepsilon|\varepsilon|^\alpha
b_i^{1+\alpha}p_i(y))p'_i(x))\sin^2(\frac{x-y}{2})\right)$}} \\
& \ \ \ \times (\partial ^{k-1}p_{i}(y)-\partial ^{k-1}p_{i}(x))dy+l.o.t.
\end{split}%
\end{equation*}%
By Sobolev embedding, we know that $\Vert \partial ^{i}p_{i}\Vert
_{L^{\infty }}\leq C\Vert p_{i}\Vert _{X^{k+\alpha -1}}<\infty $ for $%
i=0,1,2 $ and $k\geq 3$. Therefore, by combining the mean value theorem
together with H\"{o}lder inequality, one gets
\begin{equation*}
\begin{split}
\left\Vert \partial ^{k-1}F_{i22}\right\Vert _{L^{2}}& \leq C\left\Vert \int
\!\!\!\!\!\!\!\!\!\;{}-{}\frac{\partial ^{k}p_{i}(x)-\partial ^{k}p_{i}(y)}{%
|\sin (\frac{x-y}{2})|^{\alpha }}dy\right\Vert _{L^{2}}+C\left\Vert \int
\!\!\!\!\!\!\!\!\!\;{}-{}\frac{\partial ^{k-1}p_{i}(x)-\partial
^{k-1}p_{i}(y)}{|\sin (\frac{x-y}{2})|^{\alpha }}dy\right\Vert _{L^{2}} \\
& \leq C\Vert p_{i}\Vert _{X^{k+\alpha -1}}+C\Vert p_{i}\Vert _{X^{k+\alpha
-2}}<\infty .
\end{split}%
\end{equation*}%
Hence, we deduce that the range of $F_{i2}$ belongs to $Y^{k-1}$.

By the computations and notations obtained in Lemma \ref{lem2-2},
specifically in \eqref{2-7}, one only need to obtain the continuity of the
most singular term $F_{i22}$. In other words, for $f_{i1},f_{i2}\in B_{r}$,
we have $\Vert F_{i22}(\varepsilon ,p_{i1})-F_{i22}(\varepsilon
,p_{i2})\Vert _{Y^{k-1}}\leq C\Vert p_{i1}-p_{i2}\Vert _{X^{k+\alpha -1}}$.
In fact, the continuity of the nonlinear functional $F_{i2}(\varepsilon
,p_{1},p_{2}):\left( -\frac{1}{2},\frac{1}{2}\right) \times B_{r}\rightarrow
Y^{k-1}$ can be treated in the same manner as in the proof of Lemma \ref%
{lem2-2}. Finally, by invoking the Taylor formula \eqref{2-5}, we obtain the
following expression of $F_{i2}$ which will be used later
\begin{equation}
{\resizebox{.94\hsize}{!}{$
F_{i2}=C_\alpha(1-\frac{\alpha}{2})\gamma_i\int\!\!\!\!\!\!\!\!\!\; {}-{}
\frac{p_i(x-y)\sin(y)dy}{\left(4\sin^2(\frac{y}{2})\right)^{\frac{%
\alpha}{2}}}-C_\alpha\gamma_i\int\!\!\!\!\!\!\!\!\!\; {}-{}
\frac{(p'_i(x)-p'_i(x-y))\cos(y)dy}{\left(4\sin^2(\frac{y}{2})\right)^{%
\frac{\alpha}{2}}}+\varepsilon|\varepsilon|^\alpha
\mathcal{R}_{2}(\varepsilon,p_1,p_2),$}}  \label{2-21}
\end{equation}%
where $\mathcal{R}_{2}(\varepsilon ,p_{1},p_{2}):\left( -\frac{1}{2},\frac{1%
}{2}\right) \times B_{r}\rightarrow Y^{k-1}$ is continuous.

On the same spirit, one can prove that $F_{i3}(\varepsilon
,p_{1},p_{2}):\left( -\frac{1}{2},\frac{1}{2}\right) \times B_{r}\rightarrow
Y^{k-1}$ and it has the following closing form
\begin{equation}
F_{i3}=\frac{\gamma _{3-i}\sin (x)}{2d^{1+\alpha }}+|\varepsilon |^{\alpha }%
\mathcal{R}_{3}(\varepsilon ,p_{1},p_{2}),  \label{2-22}
\end{equation}%
where $\mathcal{R}_{3}(\varepsilon ,p_{1},p_{2}):\left( -\frac{1}{2},\frac{1%
}{2}\right) \times B_{r}\rightarrow Y^{k-1}$ is also continuous. Hence, the
proof is completed.
\end{proof}

Next, we obtain the linearization of the functional $F_{i}^{\alpha }$ and
show that in fact this functional is a $C^{1}$ function.

For $(\varepsilon ,\Omega_\alpha ,\bar{x}_\alpha,p_{1},p_{2})\in \left( -\frac{1}{2},\frac{1%
}{2}\right) \times \mathbb{R}\times \mathbb{R}\times B_{r}$ and $%
h_{i},h_{3-i}\in X^{k+\alpha -1}$, the Gateaux derivatives of the functional
$F_{i}^{\alpha }$ at the point $(\varepsilon ,\Omega_\alpha ,\bar{x}_\alpha,p_{1},p_{2})$
and directions $h_{i}$ and $h_{3-i}$ are given respectively by
\begin{equation*}
\partial _{p_{i}}F_{i}^{\alpha }(\varepsilon ,\Omega_\alpha
,\bar{x}_\alpha,p_{1},p_{2})h_{i}:=\lim\limits_{t\rightarrow 0}\frac{1}{t}\left(
F_{i}^{1}(\varepsilon ,\Omega_\alpha ,\bar{x}_\alpha,p_{i}+th_{i})-F_{i}^{1}(\varepsilon
    ,\Omega_\alpha,\bar{x}_\alpha ,p_{1},p_{2})\right)
\end{equation*}%
and
\begin{equation*}
\partial _{p_{3-i}}F_{i}^{\alpha }(\varepsilon ,\Omega_\alpha
,\bar{x}_\alpha,p_{1},p_{2})h_{3-i}:=\lim\limits_{t\rightarrow 0}\frac{1}{t}\left(
F_{i}^{1}(\varepsilon ,\Omega_\alpha ,\bar{x}_\alpha,p_{3-i}+th_{3-i})-F_{i}^{1}(\varepsilon
,\Omega_\alpha,\bar{x}_\alpha ,p_{1},p_{2})\right) .
\end{equation*}%
The continuity of $\partial _{p_{i}}F_{i}^{\alpha }(\varepsilon ,\Omega_\alpha
,\bar{x}_\alpha,p_{1},p_{2})h_{i}$ and $\partial _{p_{3-i}}F_{i}^{\alpha
}(\varepsilon ,\Omega_\alpha ,\bar{x}_\alpha,p_{1},p_{2})h_{3-i}$ is obtained in the same
way as in Lemma \ref{lem2-3} by employing the same arguments as in the case $%
\alpha =1$. For this reason, the proof is left to the reader.

\begin{lemma}
\label{lem2-8b} For each $(\varepsilon ,\Omega_\alpha ,\bar{x}_\alpha,p_{1},p_{2})\in \left(
-\frac{1}{2},\frac{1}{2}\right) \times \mathbb{R}\times \mathbb{R}\times
B_{r}$, the partial derivatives of the functional $F$ that are given by  $\partial _{p_{3-i}}F_{i}^{\alpha
}(\varepsilon,\Omega_\alpha,\bar{x} ,p_{1},p_{2})h_{3-i}$, $\partial
_{p_{i}}F_{i}^{\alpha }(\varepsilon,\Omega_\alpha,\bar{x}_\alpha,p_{1},p_{2})h_{3-i},$ $%
\partial _{\Omega_\alpha }F_{i}^{\alpha }(\varepsilon,\Omega_\alpha,\bar{x}_\alpha ,p_{1},p_{2})$
and $\partial _{\bar{x}_\alpha}F_{i}^{\alpha }(\varepsilon,\Omega_\alpha,\bar{x}_\alpha
,p_{1},p_{2}) $ exist and are continuous.
\end{lemma}

In fact, proceeding as above, we can compute the following Gateaux
derivatives of the functional $F_{i}^{\alpha }(\varepsilon ,\Omega_\alpha
,\bar{x}_\alpha,p_{1},p_{2})$
\begin{equation*}
{\resizebox{.97\hsize}{!}{$
\partial_{p_i}F^\alpha_i(0,\Omega_\alpha,\bar{x}_\alpha,0,0)h_i=C_\alpha(1-\frac{\alpha}{2})%
\gamma_i\int\!\!\!\!\!\!\!\!\!\; {}-{}
\frac{h_i(x-y)\sin(y)dy}{\left(4\sin^2(\frac{y}{2})\right)^{\frac{%
\alpha}{2}}}-C_\alpha\gamma_i\int\!\!\!\!\!\!\!\!\!\; {}-{}
\frac{(h'_i(x)-h'_i(x-y))\cos(y)dy}{\left(4\sin^2(\frac{y}{2})\right)^{%
\frac{\alpha}{2}}},$}}
\end{equation*}%
and
\begin{equation*}
\partial _{p_{3-i}}F_{i}^{\alpha }(0,\Omega_\alpha ,\bar{x}_\alpha,0,0)h_{3-i}=0.
\end{equation*}%
Now we are in position to show the property below about point vortices.

\begin{lemma}
\label{traveling} For $\gamma _{1}\neq \gamma _{2}$, there exist two initial
point vortices $\gamma _{1}\pi \delta _{0}$ and $\gamma _{2}\pi \delta _{d}$
rotating uniformly about the centroid
\begin{equation}
\bar{x}^{\ast }_\alpha:=\frac{d\gamma _{2}}{\gamma _{1}+\gamma _{2}},  \label{posb}
\end{equation}%
with angular velocity
\begin{equation}
\Omega ^{\ast }_\alpha:=\frac{\alpha C_{\alpha }(\gamma _{1}+\gamma _{2})}{%
2d^{2+\alpha }}.  \label{velb}
\end{equation}
\end{lemma}

\begin{proof}
By using similar computations as in the case $\alpha =1$, we can easily
verify that
\begin{equation*}
F_{i}^{\alpha }(0,\Omega_\alpha ,\bar{x}_\alpha,0,0)=-\Omega_\alpha \left( -(-1)^{i}\bar{x}_\alpha\sin
(x)+(i-1)d\sin (x)\right) +\frac{\alpha C_{\alpha }\gamma _{3-i}\sin (x)}{%
2d^{1+\alpha }},\,\mbox{for}\,i=1,2.
\end{equation*}%
Therefore $F_{i}^{\alpha }(0,\Omega_\alpha ,\bar{x}_\alpha,0,0)=0$ if and only if, the
following two equations are satisfied
\begin{equation}
F_{1}^{\alpha }(0,\Omega_\alpha ,\bar{x}_\alpha,0,0)=-\Omega_\alpha \bar{x}_\alpha\sin (x)+\frac{\alpha
C_{\alpha }\gamma _{2}\sin (x)}{2d^{1+\alpha }}=0  \label{f1b}
\end{equation}%
and
\begin{equation}
F_{2}^{\alpha }(0,\Omega_\alpha ,\bar{x}_\alpha,0,0)=-\Omega_\alpha \left( -\bar{x}_\alpha\sin (x)+d\sin
(x)\right) +\frac{\alpha C_{\alpha }\gamma _{1}\sin (x)}{2d^{1+\alpha }}=0.
\label{f2b}
\end{equation}%
Adding the above equations, we obtain that
\begin{equation*}
\Omega_\alpha=\Omega _{\alpha }^{\ast }:=\frac{\alpha C_{\alpha }(\gamma
_{1}+\gamma _{2})}{2d^{2+\alpha }}.
\end{equation*}%
Now by replacing $\Omega ^{\ast }_\alpha$ into \eqref{f1b}, we arrive at
\begin{equation*}
\bar{x}_\alpha=\bar{x}^{\ast }_\alpha:=\frac{d\gamma _{2}}{\gamma _{1}+\gamma _{2}}.
\end{equation*}
\end{proof}

In the next lemma, we show that the linearization of the functional is in
fact an isomorphism in the case $\alpha \in (1,2)$.

\begin{lemma}
\label{lem2-9b} Let $h=(\beta _{1},\beta _{2},h_{1},h_{2})\in \mathbb{R}%
\times \mathbb{R}\times X^{k+\alpha -1}$ with $h_{i}(x)=\sum_{j=1}^{\infty
}a_{n}^{i}\sin (jx)$, $i=1,2$. Then, the linearization has the following
form
\begin{align*}
D_{g}F^{\alpha }(0,g_{0})h(x)=& -\dfrac{\beta _{1}d}{\gamma _{1}+\gamma _{2}}%
\begin{pmatrix}
\gamma _{2} \\
\gamma _{1}%
\end{pmatrix}%
\sin (x)-\dfrac{\alpha C_{\alpha }\beta _{2}(\gamma _{1}+\gamma _{2})}{%
2d^{2+\alpha }}%
\begin{pmatrix}
1 \\
-1%
\end{pmatrix}%
\sin (x) \\
& \qquad -\sum_{j=2}^{\infty }\,j\sigma _{j}%
\begin{pmatrix}
{\gamma _{1}}a_{j}^{1} \\
{\gamma _{2}}a_{j}^{2}%
\end{pmatrix}%
\sin (jx),
\end{align*}%
where
\begin{equation*}
g\triangleq (\Omega_\alpha ,\bar{x}_\alpha,p_{1},p_{2}),\quad g_{0}\triangleq (\Omega
_{\alpha }^{\ast },\bar{x}^{\ast }_\alpha,0,0),
\end{equation*}%
and
\begin{equation*}
\sigma _{j}=2^{\alpha -1}\frac{\Gamma (1-\alpha )}{(\Gamma (1-\frac{\alpha }{%
2}))^{2}}\left( \frac{\Gamma (1+\frac{\alpha }{2})}{\Gamma (2-\frac{\alpha }{%
2})}-\frac{\Gamma (j+\frac{\alpha }{2})}{\Gamma (1+j-\frac{\alpha }{2})}%
\right) .
\end{equation*}%
Moreover, the linearized operator $D_{g}F^{\alpha }(0,g_{0}):\mathbb{R}%
\times \mathbb{\ R}\times X^{k+\alpha -1}\rightarrow Y^{k-1}$ is an
isomorphism.
\end{lemma}

\begin{proof}
The Gateaux derivative of $F_{i}^{\alpha }(0,\Omega ,\bar{x},0,0)$ at the
directions $h_{i}$ and $h_{3-i}$ are given by
\begin{equation*}
{\resizebox{.98\hsize}{!}{$\partial _{p_{i}}F_{i}^{\alpha }(0,\Omega_\alpha ,\bar{x}_\alpha,0,0)h_{i}=C_{\alpha }(1-%
\frac{\alpha }{2})\gamma _{i}\int \!\!\!\!\!\!\!\!\!\;{}-{}\frac{%
h_{i}(x-y)\sin (y)dy}{\left( 4\sin ^{2}(\frac{y}{2})\right) ^{\frac{\alpha }{%
2}}}-C_{\alpha }\gamma _{i}\int \!\!\!\!\!\!\!\!\!\;{}-{}\frac{%
(h_{i}^{\prime }(x)-h_{i}^{\prime }(x-y))\cos (y)dy}{\left( 4\sin ^{2}(\frac{%
y}{2})\right) ^{\frac{\alpha }{2}}}$}},
\end{equation*}%
and
\begin{equation*}
\partial _{p_{3-i}}F_{i}^{\alpha }(0,\Omega_\alpha ,\bar{x}_\alpha,0,0)h_{3-i}=0.
\end{equation*}%
So, from Lemma \ref{lem2-7}, we can obtain the closed form of the functional
$F_{i}^{\alpha }(0,\Omega ,\bar{x},p_{1},p_{2})$ which is given by
\begin{align}
F_{i}^{\alpha }(0,\Omega_\alpha ,\bar{x}_\alpha,p_{1},p_{2})& =C_{\alpha }\gamma _{i}\int
\!\!\!\!\!\!\!\!\!\;{}-{}\frac{p_{i}(x-y)\sin (y)dy}{\left( 4\sin ^{2}(\frac{%
y}{2})\right) ^{\frac{\alpha }{2}}}-C_{\alpha }\gamma _{i}\int
\!\!\!\!\!\!\!\!\!\;{}-{}\frac{(p_{i}^{\prime }(x)-p_{i}^{\prime }(x-y))\cos
(y)dy}{\left( 4\sin ^{2}(\frac{y}{2})\right) ^{\frac{\alpha }{2}}}  \notag \\
& \qquad -\Omega_\alpha \left( (-1)^{i}\bar{x}\sin (x)-(i-1)d\sin (x)\right) +\frac{%
\alpha C_{\alpha }\gamma _{3-i}\sin (x)}{2d^{1+\alpha }},  \label{gatf}
\end{align}%
since $F^{\alpha }=(F_{1}^{\alpha },F_{2}^{\alpha })$ and $p=(p_{1},p_{2})$.
Hence, for $h=(h_{1},h_{2})\in X^{k+\alpha -1}$, one gets
\begin{equation*}
D_{p}F^{\alpha }(0,\Omega_\alpha ,\bar{x}_\alpha,0,0)h(x)=%
\begin{pmatrix}
-{\gamma _{1}}\sum\limits_{j=2}^{\infty }a_{j}^{1}\sigma _{j}j\sin (jx), \\
-{\gamma _{2}}\sum\limits_{j=2}^{\infty }a_{j}^{2}\sigma _{j}j\sin (jx),%
\end{pmatrix}%
\end{equation*}%
with
\begin{equation*}
\sigma _{j}=2^{\alpha -1}\frac{\Gamma (1-\alpha )}{(\Gamma (1-\frac{\alpha }{%
2}))^{2}}\left( \frac{\Gamma (1+\frac{\alpha }{2})}{\Gamma (2-\frac{\alpha }{%
2})}-\frac{\Gamma (j+\frac{\alpha }{2})}{\Gamma (1+j-\frac{\alpha }{2})}%
\right) .
\end{equation*}%
Differentiating (\ref{gatf}) with respect to $\Omega_\alpha $ around the point $%
(\Omega^{\ast }_\alpha,\bar{x}^{\ast }_\alpha)$ yields
\begin{equation*}
\partial_{\Omega_\alpha }F_{i}^{\alpha }(0,\Omega ^{\ast }_\alpha,\bar{x}^{\ast
}_\alpha,p_{1},p_{2})=(-1)^{i}\bar{x}^{\ast }_\alpha\sin (x)-(i-1)d\sin (x).
\end{equation*}%
In turn, differentiating (\ref{gatf}) with respect to $\bar{x}$ at $(\Omega
^{\ast }_\alpha,\bar{x}^{\ast }_\alpha),$ we get
\begin{equation*}
\partial _{\bar{x}}F_{i}^{\alpha }(0,\Omega ^{\ast }_\alpha,\bar{x}^{\ast
}_\alpha,p_{1},p_{2})=(-1)^{i}\Omega ^{\ast }_\alpha\sin (x).
\end{equation*}%
As before, for all $(\beta _{1},\beta _{2})\in \mathbb{R}^{2}$ it follows
that
\begin{equation}  \label{derivb}
D_{(\Omega_\alpha ,\bar{x}_\alpha)}F^{\alpha }(0,\Omega ^{\ast }_\alpha,\bar{x}^{\ast
}_\alpha,p_{1},p_{2})(\beta _{1},\beta _{2})=-\beta _{1}%
\begin{pmatrix}
\dfrac{\gamma _{2}d}{\gamma _{1}+\gamma _{2}} \\
\dfrac{\gamma _{1}d}{\gamma _{1}+\gamma _{2}}%
\end{pmatrix}%
\sin (x)-\beta _{2}%
\begin{pmatrix}
\dfrac{\alpha C_{\alpha }(\gamma _{1}+\gamma _{2})}{2d^{2+\alpha }} \\
-\dfrac{\alpha C_{\alpha }(\gamma _{1}+\gamma _{2})}{2d^{2+\alpha }}%
\end{pmatrix}%
\sin (x).
\end{equation}%
Therefore, similarly as in the case $\alpha =1,$ we arrive at
\begin{align}
D_{g}F^{\alpha }(0,g_{0})h(x)=& -\dfrac{\beta _{1}d}{\gamma _{1}+\gamma _{2}}%
\begin{pmatrix}
\gamma _{2} \\
\gamma _{1}%
\end{pmatrix}%
\sin (x)-\beta _{2}\dfrac{\alpha C_{\alpha }(\gamma _{1}+\gamma _{2})}{%
2d^{2+\alpha }}%
\begin{pmatrix}
1 \\
-1%
\end{pmatrix}%
\sin (x)  \label{dfb} \\
& \qquad -\sum_{j=2}^{\infty }\,j%
\begin{pmatrix}
{\gamma _{1}}a_{j}^{1} \\
{\gamma _{2}}a_{j}^{2}%
\end{pmatrix}%
\sin (jx).  \notag
\end{align}%
For $k\in Y^{k-1}$ satisfying the expansion
\begin{equation*}
k(x)=\sum_{j=1}^{\infty }%
\begin{pmatrix}
A_{j} \\
B_{j}%
\end{pmatrix}%
\sin (jx),
\end{equation*}%
we have that
\begin{align}
D_{g}& F^{\alpha }(0,g_{0})^{-1}k(x)=  \notag \\
& -\bigg(\dfrac{A_{1}+B_{1}}{d}\cos (x),\dfrac{2d^{2+\alpha }(A_{1}\gamma
_{1}-B_{1}\gamma _{2})}{\alpha C_{\alpha }(\gamma _{1}+\gamma _{2})^{2}}\cos
(x),\displaystyle\sum_{j=2}^{\infty }\dfrac{A_{j}}{j\gamma _{1}}\cos (jx),%
\displaystyle\sum_{j=2}^{\infty }\dfrac{B_{j}}{j\gamma _{2}}\cos (jx)\bigg).
\notag
\end{align}%
In order to prove that $D_{(\Omega_\alpha ,\bar{x}_\alpha,p_{1},p_{2})}F^{\alpha }(0,\Omega
^{\ast }_\alpha,\bar{x}^{\ast }_\alpha,p_{1},p_{2})$ is an isomorphism from $\mathbb{R}%
\times \mathbb{\ R}\times X^{k+\alpha -1}$ to $Y^{k-1},$ we only need to
show its invertibility. In fact, the restricted linear operator $D_{(\Omega_\alpha
,\bar{x}_\alpha,p_{1},p_{2})}F^{\alpha }(0,\Omega ^{\ast }_\alpha,\bar{x}^{\ast }_\alpha,p_{1},p_{2})$
is invertible if and only if the determinant of the two vectors obtained in (%
\ref{derivb}) is non-vanishing. This determinant is equal to $-\displaystyle
\frac{\alpha C_{\alpha }}{2d^{1+\alpha }}({\gamma _{1}+\gamma _{2}})$ and
different from zero for all $\gamma _{1}+\gamma _{2}\neq 0.$
\end{proof}

According to \eqref{2-20}, \eqref{2-21} and \eqref{2-22}, if we denote the
centroid by
\begin{equation}
\bar{x}^{\ast }_\alpha:=\frac{d\gamma _{2}}{\gamma _{1}+\gamma _{2}},  \label{2-23}
\end{equation}%
and the angular velocity by
\begin{equation*}
\Omega ^{\ast }_\alpha:=\frac{\gamma _{1}+\gamma _{2}}{2d^{2+\alpha }},
\end{equation*}%
then $F_{i}^{\alpha }(0,\Omega _{\alpha }^{\ast },\bar{x}^{\ast }_\alpha,0,0)=0$. As
before, we need to adjust the angular velocity $\Omega _{\alpha }$.

\begin{lemma}
\label{lem2-10} There exist
\begin{equation*}
\Omega _{\alpha }(\varepsilon ,p_{1},p_{2}):=\Omega _{\alpha }^{\ast
}+\varepsilon ^{\alpha }\mathcal{R}_{\Omega }(\varepsilon ,p_{1},p_{2})
\end{equation*}%
and
\begin{equation*}
\bar{x}_\alpha(\varepsilon ,p_{1},p_{2}):=\bar{x}^{\ast }_\alpha+\varepsilon ^{\alpha }%
\mathcal{R}_{\Omega }(\varepsilon ,p_{1},p_{2}),
\end{equation*}%
where $\Omega _{\alpha }^{\ast }$ and $\bar{x}^{\ast }_\alpha$ are as in \eqref{velb}
and \eqref{2-23}, respectively, and a continuous function $\mathcal{R}%
_{\Omega }(\varepsilon ,p_{1},p_{2}):X^{k+\alpha -1}\rightarrow \mathbb{R}$
such that $\tilde{F}^{\alpha }(\varepsilon ,p_{1},p_{2}):(-\frac{1}{2},\frac{%
1}{2})\times \mathbb{R}\times \mathbb{R}\times B_{r}\rightarrow Y^{k-1}$ is
given by
\begin{equation*}
\tilde{F}^{1}(\varepsilon,p_{1},p_{2}):=F^{1}(\varepsilon
,\Omega _{\alpha }(\varepsilon ,p_{1},p_{2}),\bar{x}_\alpha,p_{1},p_{2}).
\end{equation*}%
Moreover, $D_{p}\mathcal{R}_{\Omega }(\varepsilon ,p_{1},p_{2})\bar{h}%
:X^{k+\alpha -1}\rightarrow \mathbb{R}$ is continuous, where $\bar{h}%
=(h_{1},h_{2})$ and $p=(p_{1},p_{2})$.
\end{lemma}

\begin{proof}
Proceeding similarly to the proof of Lemma \ref{lem2-5} and using equations %
\eqref{2-20},\eqref{2-21}, and \eqref{2-22}, we have that $\Omega _{\alpha
}(\varepsilon ,p_{1},p_{2})$ must satisfy
\begin{equation}
-\Omega _{\alpha }\left( \bar{x}+\varepsilon |\varepsilon |^{\alpha }\tilde{%
\mathcal{R}}_{1}(\varepsilon ,p_{1},p_{2})\right) +\varepsilon |\varepsilon
|^{\alpha }\tilde{\mathcal{R}}_{2}(\varepsilon ,p_{1},p_{2})+\frac{\gamma
_{2}}{2d^{1+\alpha }}+\varepsilon ^{\alpha }\tilde{\mathcal{R}}%
_{3}(\varepsilon ,p_{1},p_{2})=0  \label{vsingb}
\end{equation}%
and
\begin{equation}
-\Omega _{\alpha }\left( -\bar{x}+d+\varepsilon |\varepsilon |^{\alpha }\tilde{%
\mathcal{R}}_{1}(\varepsilon ,p_{1},p_{2})\right) +\varepsilon |\varepsilon
|^{\alpha }\tilde{\mathcal{R}}_{2}(\varepsilon ,p_{1},p_{2})+\frac{\gamma
_{1}}{2d^{1+\alpha }}+\varepsilon ^{\alpha }\tilde{\mathcal{R}}%
_{3}(\varepsilon ,p_{1},p_{2})=0,  \label{xsingb}
\end{equation}%
where $\tilde{\mathcal{R}}_{i}(\varepsilon ,p_{1},p_{2})$ $(i=1,2,3)$ is the
contribution associated to $\mathcal{R}_{i}$ to the first Fourier
coefficient. By adding \eqref{vsingb} and \eqref{xsingb}, and then
substituting $\Omega _{\alpha }^{\ast }$ into \eqref{vsingb}, we obtain that
\begin{equation*}
\Omega _{\alpha }(\varepsilon ,p_{1},p_{2}):=\Omega _{\alpha }^{\ast
}+\varepsilon ^{\alpha }\mathcal{R}_{\Omega }(\varepsilon ,p_{1},p_{2}),
\end{equation*}%
and
\begin{equation*}
\bar{x}_\alpha(\varepsilon ,p_{1},p_{2}):=\bar{x}_\alpha^{\ast }+\varepsilon ^{\alpha }%
\mathcal{R}_{\Omega }(\varepsilon ,p_{1},p_{2}),
\end{equation*}%
where $\mathcal{R}_{\Omega }(\varepsilon ,p_{1},p_{2}):X^{k+\alpha
-1}\rightarrow \mathbb{R}$ is continuous. By Lemma \ref{lem2-8b}, we know
that $D_{p}{\mathcal{R}}_{i}(\varepsilon ,p_{1},p_{2})$ $(i=1,2,3)$ is
continuous. Hence, we conclude that $D_{p}\mathcal{R}_{\Omega }(\varepsilon
,p_{1},p_{2})\bar{h}:X^{k+\alpha -1}\rightarrow \mathbb{R}$ is also
continuous.
\end{proof}

Now, we are in position to state our main result concerning to the
asymmetric co-rotating vortex pairs with unequal magnitudes associated to
the gSQG equations with $1<\alpha <2$.

\begin{theorem}[gSQG equations]
\label{mainb} The following assertions hold true.

\begin{enumerate}
\item There exists $\varepsilon_0>0$ and a unique $C^1$ function $%
g=(\Omega_\alpha,\bar{x}_\alpha,p_1,p_2): [-\varepsilon_0,\varepsilon_0]\longrightarrow
\mathbb{R}\times\mathbb{R}\times B_r$ such that
\begin{equation*}  \label{gf1f2b}
F^\alpha\Big(\varepsilon,
\Omega_\alpha(\varepsilon),\bar{x}_\alpha(\varepsilon),p_1(\varepsilon), p_2(\varepsilon)\Big)=0.
\end{equation*}
Moreover,
\begin{equation*}
\Big(\Omega_\alpha(0),\bar{x}_\alpha(0),p_1(0), p_2(0)\Big)=\big(\Omega^*_\alpha,\bar{x}^*_\alpha,0,0\big).
\end{equation*}
In other words, the solution passes through the origin.

\item For each $\varepsilon \in \lbrack -\varepsilon _{0},\varepsilon
_{0}]\backslash \{0\},$ one has
\begin{equation*}
\big(p_{1}(\varepsilon ),p_{2}(\varepsilon )\big)\neq (0,0).
\end{equation*}

\item If $(\varepsilon ,p_{1},p_{2})$ is a solution to $F^{\alpha
}(\varepsilon ,g)=0$, then $(-\varepsilon ,\tilde{p}_{1},\tilde{p}_{2})$ is
also a solution, where
\begin{equation*}
\tilde{p}_{i}(x)=p_{i}(-x),\quad i=1,2.
\end{equation*}

\item For all $\varepsilon \in \lbrack -\varepsilon _{0},\varepsilon
]\backslash \{0\}$, the set of solutions $R_{i}(x)$ parameterizes convex
patches at least of class $C^{1}$. 
\end{enumerate}
\end{theorem}

\begin{proof}
\textbf{(1)} Following the same steps as in Theorem \ref{main}, we can
guarantee that $F^{\alpha }:(-\frac{1}{2},\frac{1}{2})\times \mathbb{R}%
\times \mathbb{R}\times B_{r}\rightarrow {Y^{k-1}}$ is a $C^{1}$ function
and, by applying Lemma \ref{lem2-9b}, we conclude that $D_{g}F^{\alpha }\big(%
0,g_{0}\big):X^{k+\alpha -1}\rightarrow {Y^{k-1}}$ is an isomorphism. Now,
we can apply the implicit function theorem and obtain that there exist $%
\varepsilon _{0}>0$ and a unique $C^{1}$ parametrization $g=(\Omega_\alpha
,\bar{x}_\alpha,p_{1},p_{2}):[-\varepsilon _{0},\varepsilon _{0}]\rightarrow B_{r}$
such that the set of solutions
\begin{equation*}
F^{\alpha}\big(\varepsilon ,\Omega_\alpha (\varepsilon ),\bar{x}_\alpha(\varepsilon
),p_{1}(\varepsilon ),p_{2}(\varepsilon )\big)=0
\end{equation*}%
is not empty. Moreover, we can set $\varepsilon =0$ such that the solution
passes through the origin, that is,
\begin{equation*}
\big(\Omega_\alpha ,\bar{x}_\alpha,p_{1},p_{2}\big)(0)=(\Omega _{\alpha }^{\ast
},\bar{x}^{\ast }_\alpha,0,0).
\end{equation*}%
\newline
\textbf{(2)} The proof of this part follows the same steps that the proof of
item (2) in Theorem \ref{main}. In this way, one can show that for each $%
\varepsilon \in \lbrack -\varepsilon _{0},\varepsilon _{0}]\setminus \{0\}$,
we have $(p_{1}(\varepsilon ),p_{2}(\varepsilon ))\neq (0,0)$. In fact, by
applying the Taylor formula in $\mathcal{R}_{3}$, we arrive at
\begin{equation*}
\mathcal{R}_{3}(\varepsilon ,p_{1},p_{2})=\varepsilon \bar{C}\sin
(2x)+o(\varepsilon ),
\end{equation*}%
where $\bar{C}$ is a positive constant depending only on $\gamma _{3-i}$, $%
\alpha $, $d$ and $b_{i}$. On the other hand, by using \eqref{2-20} and %
\eqref{2-21}, one gets
\begin{equation*}
F_{i1}(\varepsilon ,\Omega_\alpha ,\bar{x}_\alpha,0,0)=-\Omega_\alpha \left(
(-1)^{i}\bar{x}+(i-1)d\right) \sin (x)\ \ \ \ \ \mbox{and}\ \ \ \
F_{i2}(\varepsilon ,\Omega_\alpha ,0,0)=0.
\end{equation*}%
With the above arguments, we can conclude that
\begin{equation*}
F^{\alpha }(\varepsilon ,\Omega_\alpha ,\bar{x}_\alpha,0,0)\neq 0,\ \ \ \ \ \forall
\,\varepsilon \in \lbrack -\varepsilon _{0},\varepsilon _{0}]\setminus \{0\},
\end{equation*}%
as long as $\varepsilon _{0}$ is chosen sufficiently small. Consequently,
for $\varepsilon \neq 0$ small enough we have that $F_{i}^{\alpha }$ is not
equal to zero.\newline
\textbf{(3)} It is enough to verify that
\begin{equation*}
    F_{i}^{\alpha}(\varepsilon ,\Omega_\alpha
,\bar{x},p_{1},p_{2})(-x)=-F_{i}^{\alpha}(-\varepsilon ,\Omega_\alpha ,\bar{x},\tilde{p}_{1},%
\tilde{p}_{2})(x),\text{ \ for }i=1,2,
\end{equation*}%
where $F_{i}$ is a sum between \eqref{2-4}, \eqref{2-8} and \eqref{Fi3}. Let
$\tilde{p}_{i}(x)=p_{i}(-x)$. We need to verify that if $(\varepsilon
,p_{1},p_{2})$ is a solution to $F_{i}^{\alpha }(\varepsilon ,p_{1},p_{2})=0$%
, then $(-\varepsilon ,\tilde{p_{1}},\tilde{p_{2}})$ is also a solution of
the same functional. Replacing $y$ with $-y$ in the nonlinear functional $%
F_{i}^{1}(\varepsilon ,\Omega_\alpha ,\bar{x}_\alpha,p_{1},p_{2})$ and using that $p_{i}$
are even functions, we deduce that $\Omega_\alpha (\varepsilon ,p_{1},p_{2})=\Omega_\alpha
(-\varepsilon ,\tilde{p_{1}},\tilde{p_{2}})$. So, considering it in $%
F_{i}^{\alpha}(\varepsilon ,\Omega_\alpha ,\bar{x}_\alpha,p_{1},p_{2})$, it follows that $%
F_{i}^{\alpha}(-\varepsilon ,\Omega_\alpha ,\bar{x}_\alpha,\tilde{p_{1}},\tilde{p_{2}})=0$.%
\newline

\textbf{(4)} As in the case of the surface quasi-geostrophic equations ($%
\alpha =1$), we can obtain the desired result by computing the signed
curvature $\kappa _{i}$ $(i=1,2)$ of the interface of the patch
\begin{equation*}
z_{i}(x)=(z_{i}^{1}(x),z_{i}^{2}(x))=((1+\varepsilon
^{2}b_{i}^{2}p_{i}(x))\cos (x),(1+\varepsilon ^{2}b_{i}^{2}p_{i}(x))\sin (x))
\end{equation*}%
at the point $x$ and verify that it is positive for all $\varepsilon \in
\left( -\frac{1}{2},\frac{1}{2}\right) $.
\end{proof}

\section{Existence of traveling vortex pairs for the gSQG equations}

In this part we are concerning in finding a traveling global pair solutions
for the gSQG equations. In other words, it can be done by finding a zero of $%
G_{i}^{\alpha }(\varepsilon ,U,\gamma _{2},p_{1},p_{2})=0$ where the
functional $G_{i}^{\alpha }$ has the form
\begin{equation*}
G^\alpha_{i}(\varepsilon ,U,\gamma _{2},p_{1},p_{2})=G_{i1}+G_{i2}+G_{i3},
\end{equation*}%
with
\begin{equation*}
G_{i1}=-U\left( \sin (x)-\frac{\varepsilon |\varepsilon |^{\alpha
}b_{i}^{1+\alpha }p_{i}^{\prime }(x)}{(1+\varepsilon |\varepsilon |^{\alpha
}b_{i}^{1+\alpha }p_{i}(x))}\cos (x)\right) ,
\end{equation*}%
\begin{equation*}
\begin{split}
& {\resizebox{.98\hsize}{!}{$
G_{i2}=\frac{C_\alpha\gamma_i}{\varepsilon|\varepsilon|^{\alpha}b_i^{1+%
\alpha}}\int\!\!\!\!\!\!\!\!\!\; {}-{}
\frac{(1+\varepsilon|\varepsilon|^\alpha b^{1+\alpha}_i
p_i(y))\sin(x-y)dy}{\left(
|\varepsilon|^{2+2\alpha}b_i^{2+2\alpha}\left(p_i(x)-p_i(y)\right)^2+4(1+%
\varepsilon|\varepsilon|^\alpha b^{1+\alpha}_i
p_i(x))(1+\varepsilon|\varepsilon|^\alpha b^{1+\alpha}_i
p_i(y))\sin^2\left(\frac{x-y}{2}\right)\right)^{\frac{\alpha}{2}}}$}} \\
& +C_{\alpha }\gamma _{i}\,{\resizebox{.9\hsize}{!}{$
\int\!\!\!\!\!\!\!\!\!\; {}-{} \frac{(p'_i(y)-p'_i(x))\cos(x-y)dy}{\left(
|\varepsilon|^{2+2\alpha}b_i^{2+2\alpha}\left(p_i(x)-p_i(y)\right)^2+4(1+%
\varepsilon|\varepsilon|^\alpha b^{1+\alpha}_i
p_i(x))(1+\varepsilon|\varepsilon|^\alpha b^{1+\alpha}_i
p_i(y))\sin^2\left(\frac{x-y}{2}\right)\right)^{\frac{\alpha}{2}}}$}} \\
& +{\resizebox{.98\hsize}{!}{$\frac{C_\alpha
\gamma_i\varepsilon|\varepsilon|^\alpha b^{1+\alpha}_i
p'_i(x)}{1+\varepsilon|\varepsilon|^\alpha
b^{1+\alpha}_ip_i(x)}\int\!\!\!\!\!\!\!\!\!\; {}-{}
\frac{(p_i(x)-p_i(y))\cos(x-y)dy}{\left(
|\varepsilon|^{2+2\alpha}b_i^{2+2\alpha}\left(p_i(x)-p_i(y)\right)^2+4(1+%
\varepsilon|\varepsilon|^\alpha b^{1+\alpha}_i
p_i(x))(1+\varepsilon|\varepsilon|^\alpha b^{1+\alpha}_i
p_i(y))\sin^2\left(\frac{x-y}{2}\right)\right)^{\frac{\alpha}{2}}}$}} \\
& +{\resizebox{.98\hsize}{!}{$\frac{C_\alpha\gamma_i\varepsilon|%
\varepsilon|^\alpha b^{1+\alpha}_i }{1+\varepsilon|\varepsilon|^\alpha
b^{1+\alpha}_i p_i(x)}\int\!\!\!\!\!\!\!\!\!\; {}-{}
\frac{p'_i(x)p'_i(y)\sin(x-y)dy}{\left(
|\varepsilon|^{2+2\alpha}b_i^{2+2\alpha}\left(p_i(x)-p_i(y)\right)^2+4(1+%
\varepsilon|\varepsilon|^\alpha b^{1+\alpha}_i
p_i(x))(1+\varepsilon|\varepsilon|^\alpha b^{1+\alpha}_i
p_i(y))\sin^2\left(\frac{x-y}{2}\right)\right)^{\frac{\alpha}{2}}}$}}
\end{split}%
\end{equation*}

and
\begin{equation*}
\begin{split}
& {\resizebox{.98\hsize}{!}{$ G_{i3}=\frac{\gamma_{3-i} C_\alpha
(1+\varepsilon|\varepsilon|^\alpha b_{3-i}^{1+\alpha}
p_{3-i}(x))}{\varepsilon b_{3-i} (1+\varepsilon|\varepsilon|^\alpha
b^{1+\alpha}_i p_i(x)) }\int\!\!\!\!\!\!\!\!\!\; {}-{}
\frac{(1+\varepsilon|\varepsilon|^\alpha b_{3-i}^{1+\alpha}
p_{3-i}(y))\sin(x-y)dy}{\left|(\varepsilon b_{3-i}
R_{3-i}(y)\cos(y)+\varepsilon b_i R_i(x)\cos(x)-d)^2+(\varepsilon b_{3-i}
R_{3-i}(y)\sin(y)+\varepsilon b_i
R_i(x)\sin(x))^2\right|^{\frac{\alpha}{2}}}$}} \\
& {\resizebox{.98\hsize}{!}{$+\frac{\gamma_{3-i} C_\alpha \varepsilon
|\varepsilon|^{2\alpha}b_{3-i}^{1+2\alpha}}{1+\varepsilon|\varepsilon|^%
\alpha b^{1+\alpha}_i p_i(x)}\int\!\!\!\!\!\!\!\!\!\; {}-{} \frac{
p_{3-i}'(x)p_{3-i}'(y)\sin(x-y)dy}{\left|(\varepsilon b_{3-i}
R_{3-i}(y)\cos(y)+\varepsilon b_i R_i(x)\cos(x)-d)^2+(\varepsilon b_{3-i}
R_{3-i}(y)\sin(y)+\varepsilon b_i
R_i(x)\sin(x))^2\right|^{\frac{\alpha}{2}}}$}} \\
& {\resizebox{.98\hsize}{!}{$ +\frac{\gamma_{3-i} C_\alpha
|\varepsilon|^\alpha (1+\varepsilon|\varepsilon|^\alpha b_{3-i}^{1+\alpha}
p_{3-i}(x))}{1+\varepsilon|\varepsilon|^\alpha b^{1+\alpha}_i
p_i(x)}\int\!\!\!\!\!\!\!\!\!\; {}-{}
\frac{b_{3-i}^{\alpha}(p'_{3-i}(y)-p'_{3-i}(x))\cos(x-y)dy}{\left|(%
\varepsilon b_{3-i} R_{3-i}(y)\cos(y)+\varepsilon b_i
R_i(x)\cos(x)-d)^2+(\varepsilon b_{3-i} R_{3-i}(y)\sin(y)+\varepsilon b_i
R_i(x)\sin(x))^2\right|^{\frac{\alpha}{2}}}$}} \\
& {\resizebox{.98\hsize}{!}{$+\frac{\gamma_{3-i} C_\alpha
\varepsilon|\varepsilon|^{2\alpha}p_i'(x)}{1+\varepsilon|\varepsilon|^\alpha
b^{1+\alpha}_i p_i(x)}\int\!\!\!\!\!\!\!\!\!\; {}-{}
\frac{b_{3-i}^{1+2\alpha}(p_{3-i}(x)-p_{3-i}(y))\cos(x-y)dy}{\left|(%
\varepsilon b_{3-i}^{\alpha} R_{3-i}(y)\cos(y)+\varepsilon b_i
R_i(x)\cos(x)-d)^2+(\varepsilon b_{3-i} R_{3-i}(y)\sin(y)+\varepsilon b_i
R_i(x)\sin(x))^2\right|^{\frac{\alpha}{2}}}$}.}
\end{split}%
\end{equation*}%
The study concerning traveling vortex pair in the asymmetric case on the
range $1\leq \alpha <2$ follows the same spirit as in the co-rotating case.
Hence, we will sketch and omit the details.

Set
\begin{equation*}
B_r:=%
\begin{cases}
\{p\in X^{k+\log}: \ \left\Vert p\right\Vert_{X^{k+\log}}<1\}, & \quad
\alpha=1, \\
\{p\in X^{k+\alpha-1}: \ \left\Vert p\right\Vert_{X^{k+\alpha-1}}<1\}, &
\quad 1<\alpha<2.%
\end{cases}%
\end{equation*}

Notice that $G_{i}^{\alpha }(\varepsilon ,U,\gamma _{2},p_{1},p_{2})$ is
essentially equals to functional $F_{i}^{\alpha }(\varepsilon ,\Omega
,\bar{x},p_{1},p_{2})$ which was studied in the previous section. There is
only one difference between the both functionals, which is given by the
first term $G_{i1}^{\alpha }.$ Hence, we only state their main differences.
The first one is that the functional $F_{i}^{\alpha }$ depends on $\bar{x}$
but $G_{i}^{\alpha }$ does not. In the traveling case, it depends on $\gamma
_{1}$. Another difference appears when one needs to compute the critical
traveling velocity $U_{\alpha }^{\ast }$ as in the co-rotating asymmetric
case obtained in Lemma \ref{critical}. These differences will be stated on
the following lemma which is concerning the regularity of the functional $%
G^{\alpha }$.

\begin{lemma}
\label{prop12} The following assertions hold true.

\begin{enumerate}
\item The functional $G^\alpha$ can be extended to $C^1$ function from $\big(%
-\frac12,\frac12\big)\times \mathbb{R}\times\mathbb{R}\times B_r\to Y^{k-1}$.

\item There exist two initial point vortices $\gamma _{1}\pi \delta _{0}$
and $-\gamma _{2}\pi \delta _{(d,0)}$ translating uniformly with velocity
\begin{equation*}
U_{\alpha }^{\ast }\triangleq
\begin{cases}
\frac{\gamma _{1}}{2d^{2}}, & \quad \alpha =1, \\
\frac{\alpha C_{\alpha }\gamma _{1}}{2d^{1+\alpha }}, & \quad 1<\alpha <2.%
\end{cases}%
\end{equation*}

\item For all $h=(\beta _{1},\beta _{2},h_{1},h_{2})\in \mathbb{R}\times
\mathbb{R}\times X^{k+\log }$ $(\alpha =1)$ or $h=(\beta _{1},\beta
_{2},h_{1},h_{2})\in \mathbb{R}\times \mathbb{R}\times X^{k+\alpha -1}$ $%
(1<\alpha <2)$ with $h_{i}(x)=\sum_{j=1}^{\infty }a_{n}^{i}\sin (jx),$ one
has:\newline
For $\alpha =1,$
\begin{align*}
D_{g}G^{\alpha }(0,g_{0})h(x)=& {\beta _{1}}%
\begin{pmatrix}
1 \\
1%
\end{pmatrix}%
\sin (x)-\beta _{2}%
\begin{pmatrix}
\frac{1}{2d^{2}} \\
0%
\end{pmatrix}%
\sin (x) \\
& \qquad \qquad -2Ud^{2}\sum_{j=2}^{\infty }\,j\sigma _{j}%
\begin{pmatrix}
a_{j}^{1} \\
a_{j}^{2}%
\end{pmatrix}%
\sin (jx);
\end{align*}%
For $1<\alpha <2,$
\begin{align*}
D_{g}G^{\alpha }(0,g_{0})h(x)=& {\beta _{1}}%
\begin{pmatrix}
1 \\
1%
\end{pmatrix}%
\sin (x)-\beta _{2}%
\begin{pmatrix}
\frac{\alpha C_{\alpha }}{2d^{1+\alpha }} \\
0%
\end{pmatrix}%
\sin (x) \\
& \qquad \qquad -\frac{2U d^{1+\alpha }}{\alpha C_{\alpha }}%
\sum_{j=2}^{\infty }\,j\sigma _{j}%
\begin{pmatrix}
a_{j}^{1} \\
a_{j}^{2}%
\end{pmatrix}%
\sin (jx),
\end{align*}%
where
\begin{equation*}
g\triangleq (U,\gamma _{2},p_{1},p_{2})\text{ \ and \ }g_{0}\triangleq
(U_{\alpha }^{\ast },\gamma _{1},0,0),
\end{equation*}%
and $\sigma _{j}$ is defined by
\begin{equation*}
\sigma _{j}:=%
\begin{cases}
2^{\alpha }\frac{\Gamma (1-\alpha )}{\Gamma (\frac{\alpha }{2})\Gamma (1-%
\frac{\alpha }{2})}\left( \frac{\Gamma (\frac{\alpha }{2})}{\Gamma (1-\frac{%
\alpha }{2})}-\frac{\Gamma (j+\frac{\alpha }{2})}{\Gamma (j+1-\frac{\alpha }{%
2})}\right) , & \text{for}\ \alpha \neq 1, \\
\sum\limits_{i=1}^{j}\frac{8}{2i-1} & \text{for}\ \alpha =1.
\end{cases}%
\end{equation*}

\item The linearized operator
\begin{equation*}
D_{g}G^{\alpha }(0,U^*_\alpha,\gamma _{1},0,0):%
\begin{cases}
\mathbb{R}\times \mathbb{R}\times X^{k+\log }\rightarrow Y^{k-1}, & \quad
\alpha =1, \\
\mathbb{R}\times \mathbb{R}\times X^{k+\alpha -1}\rightarrow Y^{k-1}, &
\quad 1<\alpha <2,%
\end{cases}%
\end{equation*}%
is an isomorphism
\end{enumerate}
\end{lemma}

\begin{proof}
The proof of this lemma follows the same steps as in the asymmetric
co-rotating case. Hence, we only sketch its proof. We can verify that
\begin{align}
G_{i}^{\alpha }(0,U,\gamma _{2},p_{1},p_{2})& =C_{\alpha }\gamma _{i}\int
\!\!\!\!\!\!\!\!\!\;{}-{}\frac{p_{i}(x-y)\sin (y)dy}{\left( 4\sin ^{2}(\frac{%
y}{2})\right) ^{\frac{\alpha }{2}}}-C_{\alpha }\gamma _{i}\int
\!\!\!\!\!\!\!\!\!\;{}-{}\frac{(p_{i}^{\prime }(x)-p_{i}^{\prime }(x-y))\cos
(y)dy}{\left( 4\sin ^{2}(\frac{y}{2})\right) ^{\frac{\alpha }{2}}}  \notag \\
& \qquad -U\sin (x)+\frac{\alpha C_{\alpha }\gamma _{3-i}\sin (x)}{%
2d^{1+\alpha }}.  \label{g}
\end{align}%
Similarly to Lemma \ref{lem2-8b}, it follows that
\begin{equation*}
\begin{split}
    \partial _{p_{i}}G_{i}^{\alpha }(0,U ,\gamma_2,0,0)h_{i}(x)&=C_{\alpha }(1-%
\frac{\alpha }{2})\gamma _{i}\int \!\!\!\!\!\!\!\!\!\;{}-{}\frac{%
h_{i}(x-y)\sin (y)dy}{\left( 4\sin ^{2}(\frac{y}{2})\right) ^{\frac{\alpha }{%
2}}}\\
&\qquad\qquad \qquad\qquad-C_{\alpha }\gamma _{i}\int \!\!\!\!\!\!\!\!\!\;{}-{}\frac{%
(h_{i}^{\prime }(x)-h_{i}^{\prime }(x-y))\cos (y)dy}{\left( 4\sin ^{2}(\frac{%
y}{2})\right) ^{\frac{\alpha }{2}}},
\end{split}
\end{equation*}%
and
\begin{equation*}
\partial _{p_{3-i}}G_{i}^{\alpha }(0,U,\gamma _{2},0,0)h_{3-i}(x)=0.
\end{equation*}%
Since $G^{\alpha }=(G_{1}^{\alpha },G_{2}^{\alpha }),$ then
\begin{equation*}
D_{p}G^{\alpha }(0,U,\gamma _{2},0,0)h(x)=-\frac{2Ud^{1+\alpha }}{%
\alpha C_{\alpha }}\sum_{j=2}^{\infty }\,j\sigma _{j}%
\begin{pmatrix}
a_{j}^{1} \\
a_{j}^{2}%
\end{pmatrix}%
\sin (jx).
\end{equation*}%
Differentiating \eqref{g} with respect to $U$ at the point $(U_{\alpha
}^{\ast },\gamma _{1})$ yields
\begin{equation*}
\partial _{U}G_{i}^{\alpha }(0,U_{\alpha }^{\ast },\gamma
_{1},p_{1},p_{2})=\sin (x).
\end{equation*}%
Moreover, differentiating \eqref{g} with respect to $\gamma _{2}$ at $%
(U_{\alpha }^{\ast },\gamma _{1}),$ we get
\begin{equation*}
\partial _{\gamma _{2}}G_{i}^{\alpha }(0,U_{\alpha }^{\ast },\gamma
_{1},p_{1},p_{2})=-\frac{\alpha C_{\alpha }\sin (x)}{2d^{1+\alpha }}.
\end{equation*}%
Hence, the differential of $G^{\alpha }$ with respect to the variables $%
(U,\gamma _{2})$ is given by
\begin{equation}
D_{(U,\gamma _{2})}G^{\alpha }(0,U,\gamma _{2},p_{1},p_{2})(\beta _{1},\beta
_{2})=\beta _{1}%
\begin{pmatrix}
1 \\
1%
\end{pmatrix}%
\sin (x)+\beta _{2}%
\begin{pmatrix}
-\frac{\alpha C_{\alpha }}{2d^{1+\alpha }} \\
0%
\end{pmatrix}%
\sin (x).  \label{uderiv}
\end{equation}%
Thus, we can guarantee that the linearization $D_{g}G^{\alpha }(0,g_{0})$ is
invertible for all $\gamma _{1}\in \mathbb{R}$ if and only if the
determinant of the two vectors in (\ref{uderiv}) is not zero. In fact its
determinant is given by $\frac{\alpha C_{\alpha }}{2d^{1+\alpha }}$ which is
non-vanishing.
\end{proof}

\begin{theorem}[Case $1\leq \protect\alpha <2$]
\label{trans} The following assertions hold true.

\begin{enumerate}
\item There exist $\varepsilon _{0}>0$ and a unique $C^{1}$ function $%
g\triangleq (U,\gamma _{2},p_{1},p_{2}):[-\varepsilon _{0},\varepsilon
_{0}]\longrightarrow \mathbb{R}\times \mathbb{R}\times B_{r}$ such that

\begin{equation*}
G^{\alpha }\Big(\varepsilon ,U(\varepsilon ),\gamma _{2}(\varepsilon
),p_{1}(\varepsilon ),p_{2}(\varepsilon )\Big)=0.
\end{equation*}%
Moreover,
\begin{equation*}
\Big(U(0),\gamma _{2}(0),p_{1}(0),p_{2}(0)\Big)=\Big(U_{\alpha }^{\ast
},\gamma _{1},0,0\Big),
\end{equation*}%
where $U_{\alpha }^{\ast }$ is given by
\begin{equation*}
U_{\alpha }^{\ast }:=%
\begin{cases}
\frac{\gamma _{1}}{2d^{2}}, & \quad \alpha =1, \\
\frac{\alpha C_{\alpha }\gamma _{1}}{2d^{1+\alpha }}, & \quad 1<\alpha <2.%
\end{cases}%
\end{equation*}%
In other words, the solution passes through the origin.

\item For all $\varepsilon \in \lbrack -\varepsilon _{0},\varepsilon
_{0}]\backslash \{0\},$ one has
\begin{equation*}
\big(p_{1}(\varepsilon ),p_{2}(\varepsilon )\big)\neq (0,0).
\end{equation*}

\item If $(\varepsilon ,p_{1},p_{2})$ is a solution, then $(-\varepsilon ,%
\tilde{p}_{1},\tilde{p}_{2})$ is also a solution, where
\begin{equation*}
\tilde{p}_{i}(x)=p_{i}(-x),\quad i=1,2.
\end{equation*}

\item For all $\varepsilon \in [-\varepsilon_0, \varepsilon]\backslash\{0\}$%
, the set of solutions $R_i(x)$ parameterizes convex patches at least of
class $C^1$.
\end{enumerate}
\end{theorem}

\begin{proof}
The proof follows the same steps of the proofs of Theorems \ref{main} and %
\ref{mainb}. Then, we only sketch it.\newline

\textbf{(1)} Notice that $G_{i2}(\varepsilon ,g)+G_{i3}(\varepsilon ,g)$ is
equal to $F_{i2}(\varepsilon ,g)+F_{i3}(\varepsilon ,g)$ which are given in %
\eqref{2-2} and \eqref{2-3}. Following the same argument used in Theorem \ref%
{main}, we can guarantee that $G^{\alpha }:(-\frac{1}{2},\frac{1}{2})\times
\mathbb{R}\times \mathbb{R}\times B_{r}\rightarrow {Y^{k-1}}$ is a $C^{1}$
function and, by Lemma \ref{prop12}, the linearization $D_{g}G\big(0,g_{0}%
\big)$ is an isomorphism from $\mathbb{R}\times \mathbb{R}\times X^{k+\log }$
$(\alpha =1)$, or $\mathbb{R}\times \mathbb{R}\times X^{k+\alpha -1}$ $%
(1<\alpha <2),$ to $Y^{k-1}$. Now, we can use the implicit function theorem
and obtain that there exist $\varepsilon _{0}>0$ and a unique $C^{1}$
parametrization given by $g=(U,\gamma _{2},p_{1},p_{2}):[-\varepsilon
_{0},\varepsilon _{0}]\rightarrow B_{r}$ such that the set of solutions
\begin{equation*}
G^{\alpha }\big(\varepsilon ,U(\varepsilon ),\gamma _{2}(\varepsilon
),p_{1}(\varepsilon ),p_{2}(\varepsilon )\big)=0
\end{equation*}%
is not empty. Moreover, we can set $\varepsilon =0$ such that the solution
passes through the origin, that is,
\begin{equation*}
\big(U,\gamma _{2},p_{1},p_{2}\big)(0)=(U_{\alpha }^{\ast },\gamma _{1},0,0),
\end{equation*}%
where
\begin{equation*}
U_{\alpha }^{\ast }:=%
\begin{cases}
\frac{\gamma _{1}}{2d^{2}}, & \quad \alpha =1, \\
\frac{\alpha C_{\alpha }\gamma _{1}}{2d^{1+\alpha }}, & \quad 1<\alpha <2.%
\end{cases}%
\end{equation*}%
\newline
\newline
\textbf{(2)} We can show that $(p_{1}(\varepsilon ),p_{2}(\varepsilon ))\neq
(0,0),$ for each $\varepsilon \in \lbrack -\varepsilon _{0},\varepsilon
_{0}]\setminus \{0\}.$ By applying the Taylor formula in $\mathcal{R}_{3}$,
it follows that
\begin{equation*}
\mathcal{R}_{3}(\varepsilon ,p_{1},p_{2})=\varepsilon \bar{C}\sin
(2x)+o(\varepsilon ),
\end{equation*}%
where $\bar{C}$ is a positive constant that depends only on $\gamma _{3-i}$,
$\alpha $, $d$ and $b_{i}$. Now, by using \eqref{2-20} and \eqref{2-21}, we
arrive at
\begin{equation*}
G_{i1}(\varepsilon,U ,\gamma_2,0,0)=-U\sin (x)\ \ \ \ \ \mbox{and}\ \ \ \
G_{i2}(\varepsilon,U,\gamma_2 ,0,0)=0.
\end{equation*}%
With the above arguments, we can conclude that
\begin{equation*}
G^{\alpha }(\varepsilon ,U,\gamma_2,0,0)\neq 0,\ \ \ \ \ \forall
\,\varepsilon \in \lbrack -\varepsilon _{0},\varepsilon _{0}]\setminus \{0\},
\end{equation*}%
with $\varepsilon _{0}$ chosen small enough. Consequently, for $\varepsilon
\neq 0$ small enough, we have that $G_{i}^{\alpha }$ is not equal to zero.%
\newline
\newline
\textbf{(3)} This step follows by a simple change of variables. In fact, for
$\tilde{p}_{i}(x)=p_{i}(-x)$, it is not difficult to verify that $%
(-\varepsilon ,\tilde{p}_{1},\tilde{p}_{2})$ is a solution for $%
G_{i}^{\alpha }(\varepsilon ,p_{1},p_{2})=0$ if so does $(\varepsilon
,p_{1},p_{2}).$\newline
\newline
\textbf{(4)} As in the co-rotating case, we only need to compute the signed
curvature $\kappa _{i}$ of the interface of the patch, for $i=1,2$. Thus, we
conclude the proof.
\end{proof}


\



\phantom{s} \thispagestyle{empty}

\

\noindent\textsc{Edison Cuba}\newline
Department of Mathematics, State University of Campinas, \newline
13083-859, Rua S\'{e}rgio Buarque de Holanda, 651, Campinas-SP, Brazil%
\newline
\noindent\texttt{ecubah@ime.unicamp.br}

\vspace{0.5cm}

\noindent\textsc{Lucas C. F. Ferreira (Corresponding Author)}\newline
Department of Mathematics, State University of Campinas, \newline
13083-859, Rua S\'{e}rgio Buarque de Holanda, 651, Campinas-SP, Brazil%
\newline
\noindent\texttt{lcff@ime.unicamp.br}


\begin{thebibliography}{99}
\bibitem{cao} D. Cao, G. Qin, W. Shan, C. Zou, Existence and regularity of
co-rotating and traveling-wave vortex solutions for the generalized SQG
equation\textit{,} \textit{Journal of Differential Equations} 299 (2021),
429-462.

\bibitem{Cas1} A. Castro, D. C\'{o}rdoba, J. G\'{o}mez-Serrano, Existence
and regularity of rotating global solutions for the generalized surface
quasi-geostrophic equations, \textit{Duke Math. J.} 165 (5) (2016), 935--984.

\bibitem{Cas4} A. Castro, D. C\'{o}rdoba, J. G\'{o}mez-Serrano, Uniformly
rotating analytic global patch solutions for active scalars, \textit{Ann. PDE%
}, 2 (1) (2016), Art. 1, 34.

\bibitem{Cas3} A. Castro, D. C\'{o}rdoba, J. G\'{o}mez-Serrano, Uniformly
rotating smooth solutions for the incompressible 2D Euler equations, \textit{%
Arch. Ration. Mech. Anal.} 231 (2) (2019), 719--785.

\bibitem{Cas2} A. Castro, D. C\'{o}rdoba, J. G\'{o}mez-Serrano, Global
smooth solutions for the inviscid SQG equation, \textit{Mem. Amer. Math. Soc.%
} 266 (2020), no. 1292.

\bibitem{Chae2} D. Chae, On the Euler equations in the critical
Triebel-Lizorkin spaces, \textit{Arch. Ration. Mech. Anal.} 170 (3) (2003),
185--210.

\bibitem{chae} D. Chae, P. Constantin, D. C\'{o}rdoba, F. Gancedo, J. Wu,
Generalized surface quasi-geostrophic equations with singular velocities,
\textit{Communications on Pure and Applied Mathematics} 65 (8) (2012),
1037--1066.

\bibitem{Che} J.-Y. Chemin, Fluides parfaits incompressibles. (French)
[Incompressible perfect fluids] \textit{Ast\'{e}risque} 230 (1995), 177 pp.

\bibitem{Constantin1994} P. Constantin, A. Majda, E. Tabak, Formation of
strong fronts in the 2-D quasigeostrophic thermal active scalar, \textit{%
Nonlinearity} 7 (6) (1994), 1495-1533.

\bibitem{cordova} D. C\'{o}rdoba, M.A. Fontelos, A.M. Mancho, J.L. Rodrigo,
Evidence of singularities for a family of contour dynamics equations,
\textit{Proc. Natl. Acad. Sci.} USA 102 (2005), 5949--5952.

\bibitem{cordoba2} D. C\'{o}rdoba, L. Mart\'{\i}nez-Zoroa, Non-existence and
strong ill-posedness in $C^{k,\beta }$ for the generalized surface
quasi-geostrophic equation, \url{https://doi.org/10.48550/arXiv.2207.14385}.

\bibitem{de1} F. de la Hoz, Z. Hassainia, T. Hmidi, Doubly connected
V-states for the generalized surface quasigeostrophic equations, \textit{%
Arch. Ration. Mech. Anal.}, 220 (3) (2016), 1209--1281.

\bibitem{Dritschel} D. G. Dritschel, A general theory for two-dimensional vortex interactions, \textit{J. Fluid Mech.} 293 (1995), 269-303.

\bibitem{el} T.M. Elgindi, N. Masmoudi, $L^{\infty }$ ill-posedness for a
class of equations arising in hydrodynamics. \textit{Arch. Ration. Mech.}
235 (3) (2020), 1979--2025.

\bibitem{Gancedo-1} F. Gancedo, N. Patel, On the local existence and blow-up
for generalized SQG patches. \textit{Ann. PDE} 7 (1) (2021), Paper No. 4, 63
pp.

\bibitem{Has} Z. Hassainia, T. Hmidi, On the V-states for the generalized
quasi-geostrophic equations, \textit{Comm. Math. Phys.} 337 (1) (2015),
321--377.

\bibitem{has2} Z. Hassainia, T. Hmidi, Steady asymmetric vortex pairs for
Euler equations, \textit{Discrete and continous dinamical systems} 41 (4)
(2021), 1939--1969.

\bibitem{mutipole} Z. Hassainia, M. Wheeler, Multipole Vortex Patch
Equilibria for Active Scalar Equations. \textit{SIAM J. Math. Anal.} 54 (6)
(2022), 6054--6095.

\bibitem{HM} T. Hmidi, J. Mateu, Existence of corotating and
counter-rotating vortex pairs for active scalar equations, \textit{Comm.
Math. Phys.} 350 (2017), 699--747.

\bibitem{Hmi} T. Hmidi, J.Mateu, J. Verdera, Boundary regularity of rotating
vortex patches, \textit{Arch. Ration. Mech. Anal.}, 209 (1) (2013), 171--208.


\bibitem{kiselev} A. Kiselev, L. Ryzhik, Y. Yao, A. Zlatos, Finite time
singularity for the modified SQG patch equation, \textit{Ann. of Math.} 184
(3) (2016), 909--948.

\bibitem{lazar} O. Lazar, L. Xue, Regularity results for a class of
generalized surface quasi-geostrophic equations, \textit{J. Math. Pures Appl.%
} 130 (2019), 200--250.

\bibitem{Pedlosky-1987} J. Pedlosky, Geophysical Fluid Dynamics, Springer,
New York, 1987.

\bibitem{resnick} S. G. Resnick. Dynamical problems in non-linear advective
partial differential equations. \textit{PhD thesis}, University of Chicago,
Department of Mathematics, 1995.

\bibitem{Wang-1} H. Wang, H. Jia, Local well-posedness for the 2D
non-dissipative quasi-geostrophic equation in Besov spaces, \textit{%
Nonlinear Anal.} 70 (11) (2009), 3791--3798.

\bibitem{Yud} V. I. Yudovich, Non-stationnary flows of an ideal
incompressible fluid, \textit{Zhurnal Vych Matematika} 3 (1963), 1032--1106.
\end{thebibliography}
\end{document}